\documentclass[a4paper,oneside,11pt]{article}
\usepackage[a4paper]{geometry}
\usepackage{aeguill}                 
\usepackage{graphicx}
\usepackage{caption}
\usepackage{tikz-cd}
\usepackage{amsmath,amsfonts,amssymb,amsthm}
\usepackage{mathabx}
\usepackage{pstricks,comment} 
\usepackage{pst-plot}
\usepackage{pstricks-add}
\usepackage{multido}
\usepackage{url}
\usepackage{epstopdf,enumerate}
\usepackage{srcltx}
\usepackage[utf8]{inputenc} 
\usepackage[T1]{fontenc} 
\usepackage[english]{babel} 
\usepackage{hyperref}
\usepackage[all]{xy} 
\usepackage{titlesec}
\usepackage{mathabx}
\usepackage{tikz}
\usepackage{fancyhdr}
\usepackage{mathrsfs}
\usepackage[us,12hr]{datetime}

    \newtheoremstyle{TheoremNum}
        {\topsep}{\topsep}              
        {\itshape}                      
        {}                              
        {\bfseries}                     
        {.}                             
        { }                             
        {\thmname{#1}\thmnote{ \bfseries #3}}

{\theoremstyle {definition} \newtheorem {defi} {Definition}[section]}
{\theoremstyle {plain}  \newtheorem {theo} [defi] {Theorem}}
{\theoremstyle {plain}  \newtheorem {coro} [defi] {Corollary}}
{\theoremstyle {plain} \newtheorem {prop} [defi] {Proposition}}
{\theoremstyle {plain} \newtheorem {ques} [defi] {Questions}}
{\theoremstyle {plain} \newtheorem {lem}[defi] {Lemma}}
{\theoremstyle {plain} \newtheorem {rmq}[defi] {Remark}}
{\theoremstyle {plain} }
{\theoremstyle{TheoremNum} }
{\theoremstyle{TheoremNum} }
{\theoremstyle{TheoremNum} }

\newcommand*{\myproofname}{Proof}
\newenvironment{proofblack}[1][\myproofname]{\begin{proof}[#1]\renewcommand*{\qedsymbol}{\(\blacksquare\)}}{\end{proof}}

\newcommand{\Aut}{\mathrm{Aut}}
\newcommand{\Out}{\mathrm{Out}}
\newcommand{\Inn}{\mathrm{Inn}}

\newcommand{\Stab}{\mathrm{Stab}}

\newcommand{\Fix}{\mathrm{Fix}}
\newcommand{\IA}{\mathrm{IA}}

\newcommand{\id}{\mathrm{id}}

\newcommand{\ZZ}{\mathbb{Z}}

\newcommand{\NN}{\mathbb{N}}
\newcommand{\RR}{\mathbb{R}}

\title{Aperiodicity properties of automorphism groups of free products}
\author{Yassine Guerch}
\date{\today}
\begin{document}
\maketitle
\renewcommand*\labelenumi{(\theenumi)}

\begin{abstract}
Let $G=G_1 \ast \ldots \ast G_k \ast F_N$ be a free product of finitely presented groups, where $F_N$ is a free group of rank $N \in \NN$. Let $\mathrm{Out}(G,\mathcal{G})$ be the subgroup of $\mathrm{Out}(G)$ preserving the set of conjugacy classes $\mathcal{G}=\{[G_1],\ldots,[G_k]\}$. Under natural conditions on the groups $G_i$ with $i \in \{1,\ldots,k\}$, we prove that the group $\mathrm{Out}(G,\mathcal{G})$ has a finite index subgroup $\mathrm{IA}(G,\mathcal{G},3)$ with notable aperiodicity properties. We show that the group $\mathrm{IA}(G,\mathcal{G},3)$ is torsion free and, if $\phi \in \mathrm{IA}(G,\mathcal{G},3)$, every $\phi$-periodic conjugacy class of elements of $G$ is in fact fixed by $\phi$ and every $\phi$-periodic conjugacy class of free factors of $G$ is fixed by $\phi$. 

As an application, we prove that, for every toral relatively hyperbolic group $G$, the group $\mathrm{Out}(G)$ has a finite index subgroup $\mathrm{IA}(G,3)$ with the same above mentioned aperiodicity properties. We in particular give another proof of the theorem, due to Handel--Mosher, that the kernel of the action of $\mathrm{Out}(F_N)$ on $H_1(F_N,\ZZ/3\ZZ)$ satisfies natural aperiodicity properties.
\end{abstract}

\section{Introduction}

Let $G$ be a finitely generated group and let $\phi \in \Out(G)$ be an outer automorphism of $G$. When studying dynamical properties of $\phi$, one often needs to pass to a power of $\phi$ so that any periodic orbit is pointwise fixed. However, this power generally depends on the chosen outer automorphism and is not stable by composition. This is why constructing a finite index subgroup of $\Out(G)$ in which every outer automorphism has pointwise fixed periodic orbits is particularly useful in practice.

We illustrate this philosophy with three standard examples. The first famous example of a group $G$ whose outer automorphism group contains such an aperiodic finite index subgroup is the free abelian group $\ZZ^n$ of rank $n \geq 1$. Indeed, it is a standard result of Minkowski that the group 
\[\IA(\ZZ^n,3)=\ker\left( \mathrm{GL}_n(\ZZ) \to \mathrm{GL}_n(\ZZ/3\ZZ) \right) \] is torsion free. From this it is easy to deduce that, for every $\phi \in \IA(\ZZ^n,3)$, every $\phi$-periodic element of $\ZZ^n$ is fixed by $\phi$ (see for instance Proposition~\ref{prop:abeliancase}).

When $G$ is the fundamental group of a closed, connected, hyperbolic surface $\Sigma$, the mapping class group $\mathrm{MCG}(\Sigma)=\Out(\pi_1(\Sigma))$ of the surface $\Sigma$ also contains a finite index subgroup with remarkable aperiodicity properties. Indeed, Serre~\cite{Serre61} proved that the group 
\[\IA(\pi_1(\Sigma),3)=\ker\left( \mathrm{MCG}(\Sigma) \to \Aut(H_1(\Sigma,\ZZ/3\ZZ)) \right) \] is also torsion free. Moreover, Ivanov~\cite[Theorem~1.2]{Ivanov92} later proved that if $\phi \in \mathrm{MCG}(\Sigma)$, the homotopy class of any $\phi$-periodic simple closed curve of $\Sigma$ is fixed by $\phi$ and the homotopy class of any $\phi$-periodic subsurface of $\Sigma$ is fixed by $\phi$. Ivanov then used the group $\IA(\pi_1(\Sigma),3)$ to prove a Tits alternative for $\mathrm{MCG}(\Sigma)$: every subgroup of $\IA(\pi_1(\Sigma),3)$ either contains a nonabelian free subgroup or is free abelian. His proof relies on the action of $\IA(\pi_1(\Sigma),3)$ on the Teichmüller space $\mathcal{T}(\Sigma)$ and its Thurston compactification $\overline{\mathcal{T}(\Sigma)}$, where the above mentioned aperiodicity properties of $\IA(\pi_1(\Sigma),3)$ translates into aperiodicity properties of the action of $\IA(\pi_1(\Sigma),3)$ on $\overline{\mathcal{T}(\Sigma)}$. Later, Kida~\cite{Kida10} worked in the group $\IA(\pi_1(\Sigma),3)$ to prove that the group $\mathrm{MCG}(\Sigma)$ is rigid in the sense of measure equivalence. This similarly uses the fact that aperiodicity properties of $\IA(\pi_1(\Sigma),3)$ gives aperiodicity properties of the action of $\IA(\pi_1(\Sigma),3)$ on the curve complex of $\Sigma$.

Let $F_N$ be a free group of rank $N \geq 1$. Handel--Mosher~\cite{HandelMosher20} proved that $\Out(F_N)$ also contains a finite index subgroup with notable aperiodicity properties. Indeed, let
\[\IA(F_N,3)=\ker\left( \Out(F_N) \to \Aut(H_1(F_N,\ZZ/3\ZZ)) \right). \] The group $\IA(F_N,3)$ is torsion free (see for instance~\cite[Corollary 5.7.6]{BesFeiHan00}). Handel--Mosher proved that for every $\phi \in \IA(F_N,3)$, every $\phi$-periodic conjugacy class of elements in $F_N$ is fixed by $\phi$ (\cite[Theorem~II.4.1]{HandelMosher20}) and that every $\phi$-periodic conjugacy class of free factors of $F_N$ is preserved by $\phi$ (\cite[Theorem~II.3.1]{HandelMosher20}). Later, they proved that every virtually abelian subgroup of $\IA(F_N,3)$ is abelian~\cite{HandelMosher20abelian}, and the author proved that $\IA(F_N,3)$ satisfies the unique root property~\cite{Guerch2022Roots}. As in the surface case, these aperiodicity properties turn out to be useful to understand structural properties of the group $\Out(F_N)$, such as the Tits alternative~\cite{BesFeiHan00,BesFeiHan04,BesFeiHan05}, subgroup decomposition properties~\cite{HandelMosher20} or the measure equivalence rigidity~\cite{Guirardelhorbez}. This relies on the fact that the aperiodicity properties of $\IA(F_N,3)$ also translates into aperiodicity properties of the action of $\IA(F_N,3)$ on some natural $\Out(F_N)$-complexes, such as Culler--Vogtmann's Outer space~\cite{Vogtmann1986} or the free factor complex~\cite{BestvinaFeighn14}. 

The aim of the present paper is to construct finite index subgroups with nice aperiodicity properties for larger families of outer automorphisms of groups. Recall that a \emph{toral relatively hyperbolic group} is a torsion free finitely generated group which is hyperbolic relative to finitely generated free abelian groups. This includes in particular all torsion free Gromov hyperbolic groups, but also all Sela's limit groups (see~\cite{Alibegovic2005,Dah2003}). We in particular prove the following theorem.

\begin{theo}[see Corollary~\ref{coro:IAtoralrelativelyhyperbolicgroup}]\label{thm:IntroToral}
Let $G$ be a toral relatively hyperbolic group. There exists a finite index subgroup $\IA(G,3)$ of $\Out(G)$ with the following properties.
\begin{enumerate}
\item The group $\IA(G,3)$ is torsion free.
\item If $\phi \in \IA(G,3)$, every $\phi$-periodic conjugacy class of elements of $G$ is fixed by $\phi$.
\item If $\phi \in \IA(G,3)$, every $\phi$-periodic conjugacy class of free factors of $G$ is fixed by $\phi$.
\item If $\Phi \in \Aut(G)$ is such that $[\Phi] \in \IA(G,3)$, every $\Phi$-periodic element of $G$ is fixed by $\Phi$. 
\end{enumerate}
\end{theo}

The construction of $\IA(G,3)$ is explicit and it coincides with the above mentioned groups when $G$ is free, free abelian or the fundamental group of a surface. We in particular give another proof of the free group case which does not use the CT map technology introduced by Feighn--Handel~\cite{FeiHan06}. Instead, we use group actions on canonical trees and the first generation of \emph{(relative) train tracks for free products}, see~\cite{ColTur94,FrancaMar2015,Lyman2022} and the seminal work of Bestvina--Handel~\cite{BesHan92} in the free group setting. In fact, most of the results in the present article are proved in the general context of free products.

A \emph{free product} is a couple $(G,\mathcal{G})$ where $G$ is a finitely generated group and $\mathcal{G}=\{[G_1],\ldots,[G_k]\}$ is a finite set of conjugacy classes of finitely generated subgroups of $G$ such that 
\[G=G_1 \ast \ldots \ast G_k \ast F_N,\] where $F_N$ is a free group of rank $N \in \NN$. By Stallings' theorem about ends of groups~\cite{Stallings68}, every finitely generated, infinitely-ended, torsion free group admits a nontrivial free product structure. Let $\xi(G,\mathcal{G})=k+N$ be the \emph{Grushko rank of $(G,\mathcal{G})$}.

We study the group $\Out(G,\mathcal{G})$ of outer automorphisms of $G$ which preserve the set $\mathcal{G}$. This group is a natural generalisation of $\Out(F_N)$ and has been intensively studied over the last decades~\cite{Guirardel,Horbez14,FrancaMar2015,Guirardelhorbez19laminations,Guirardelhorbez19,DahmaniLi19,Lyman2022CT}. If, for every $[G_i] \in \mathcal{G}$, the group $G_i$ is one-ended, then $\Out(G)=\Out(G,\mathcal{G})$. 

Our aim in this paper is to find an \emph{aperiodic finite index subgroup} of $\Out(G,\mathcal{G})$, that is, a finite index subgroup which satisfy the four assertions in Theorem~\ref{thm:IntroToral}. There are natural obstructions for the existence of such an aperiodic finite index subgroup. For instance, if there exists $[G_i] \in \mathcal{G}$ such that $\Out(G_i)$ has no aperiodic finite index subgroup, then we cannot find such an aperiodic finite index subgroup for $\Out(G,\mathcal{G})$. We prove that this is essentially the only restriction. Indeed, we gather in what we call the \emph{\hyperlink{SA}{Standing Assumptions}} all the hypotheses on the groups $G_i$, $\Out(G_i)$ and $\Aut(G_i)$ that we need in order to construct an aperiodic finite index subgroup for $\Out(G,\mathcal{G})$ (see Section~\ref{sec:Standing_Assumptions}). Basically, we assume that $G_i$ is finitely presented, that $G_i$ is torsion free and that $\Out(G_i)$ and $\Aut(G_i)$ contain aperiodic finite index subgroups. We in fact assume that $\Out(G_i)$ satisfies a stronger statement than Item~$(2)$ of Theorem~\ref{thm:IntroToral} since we also assume that the same is true replacing conjugacy classes of elements by conjugacy classes of \emph{periodic subgroups}. Our main theorem is the following. 

\begin{theo}[see~Proposition~\ref{prop:IAtorsionfree}, Theorems~\ref{thm:invariantfamilycyclicfixed},~\ref{thm:periodicfreefactorfixed}, Corollary~\ref{coco:periodicelementfixed}]\label{thm:Intromain}
Let $(G,\mathcal{G})$ be a free product satisfying the \hyperlink{SA}{Standing Assumptions}. Then $\Out(G,\mathcal{G})$ contains an aperiodic finite index subgroup $\IA(G,\mathcal{G},3)$.
\end{theo}

Again the construction of the aperiodic finite index subgroup $\IA(G,\mathcal{G},3)$ is explicit. It is built out of the aperiodic finite index subgroups of the groups $\Out(G_i)$ with $[G_i] \in \mathcal{G}$, the finite index subgroup of $\Out(G,\mathcal{G})$ fixing elementwise the set $\mathcal{G}$ and the finite index subgroup of $\Out(G,\mathcal{G})$ acting trivially on $H_1(G,\ZZ/3\ZZ)$. 

It must be noted that the construction of a finite index subgroup of $\Out(G,\mathcal{G})$ satisfying Item~$(3)$ of Theorem~\ref{thm:IntroToral} depends on the construction of a finite index subgroup of $\Out(G,\mathcal{G})$ satisfying Item~$(2)$ of Theorem~\ref{thm:IntroToral}. Indeed, let $G=G_1 \ast G_2 \ast F_3$, let $\Phi_i \in \Aut(G_i)$ be of infinite order and let $\Phi_3\in \Aut(F_3)$ be a nontrivial automorphism of finite order. Let $x \in F_3$ be a primitive element of $F_3$ such that the size of the $\Phi_3$-orbit of $x$ is the order of $\Phi_3$. Let $\Phi \in \Aut(G)$ acting as $\Phi_i$ on $G_i$ and as $\Phi_3$ on $F_3$. Then $\Phi$ is of infinite order and $G_1 \ast \langle x \rangle$ is a periodic free factor of $\Phi$ and its period depends on the order of $\Phi_3$.

We now sketch the proof of Theorem~\ref{thm:Intromain}. Let $\IA(G,\mathcal{G},3)$, let $\phi \in \IA(G,\mathcal{G},3)$ and let $\Phi \in \phi$. Consider the group $G \rtimes_\Phi \ZZ$, that is, the semi-direct product of $G$ by $\ZZ$ induced by the automorphism $\Phi$. The isomorphism class of this group only depends on the outer class $\phi$, so we will denote $G \rtimes_\Phi \ZZ$ by $G_\phi$. As explained before, the proof of each item of the aperiodicity properties of $\phi$ relies on the construction of an action of $G$ on a tree $T$ which is $\phi$-invariant: the action of $G$ on $T$ extends to an action of $G_\phi$. The quotient $G \backslash T$ has a natural structure of a graph of groups and $\Phi$ induces an automorphism of the graph of groups $G \backslash T$. 

Informally, the source of periodic behaviours for an automorphism of a graph of groups is twofold. Firstly, an automorphism of a graph of groups $G \backslash T$ induces an automorphism of the underlying graph $\overline{G \backslash T}$ of $G \backslash T$. This first source of periodicity is ruled out by the fact that $\Phi$ acts trivially on $H_1(G,\ZZ/3\ZZ)$. Indeed, this will imply that $\Phi$ also acts trivially on $H_1(\overline{G \backslash T},\ZZ/3\ZZ)$ and an elementary observation (see Lemma~\ref{lem:graphautotrivial}) shows that $\Phi$ induces a trivial automorphism of the graph $\overline{G \backslash T}$. Secondly, if an automorphism of $G \backslash T$ fixes pointwise the graph $\overline{G \backslash T}$, it induces automorphisms of the vertex groups of $G \backslash T$ which may have periodic behaviors. This source of periodicity is more delicate to deal with and it heavily depends on the tree $T$. 

The construction of the $\phi$-invariant tree $T$ is different according to the item of Theorem~\ref{thm:IntroToral} one wants to prove. We explain it in the case of $\phi$-periodic conjugacy classes of elements and $\phi$-periodic conjugacy classes of free factors. Let $g \in G$ whose conjugacy class $[g]$ is $\phi$-periodic and let $\mathcal{H}=\{[\langle \Phi^n(g) \rangle] \}_{n \in \NN}$. Since $[g]$ is $\phi$-periodic, it is a finite set of conjugacy classes of cyclic subgroups of $G$. Since $(G,\mathcal{G})$ is a free product, the group $G$ is hyperbolic relative to $\mathcal{G}$. Moreover, by induction on the Grushko rank of $(G,\mathcal{G})$, we may assume that $G$ is \emph{one-ended relative to $\mathcal{G} \cup \mathcal{H}$}, that is, there does not exist a proper free product structure $(G,\mathcal{F})$ of $G$ such that, for every $[A] \in \mathcal{G} \cup \mathcal{H}$ there exists $[B] \in \mathcal{F}$ with $A \subseteq B$. In that case, Guirardel--Levitt~\cite{GuirardelLevitt2015} constructed a canonical (JSJ) $\phi$-invariant tree $T$ with a $G$-action such that all edge stabilisers are either cyclic or conjugate into some $G_i$. Moreover, every subgroup $\langle \Phi^n(g) \rangle$ with $n \in \NN$ is elliptic in $T$, and every vertex stabiliser is either a conjugate of some $G_i$, cyclic, a surface group or a so-called \emph{rigid stabiliser}, that is the image of $\phi$ in the outer automorphism group of the vertex group is finite (see Theorem~\ref{thm:JSJrelhyp}). The vertex groups are therefore sufficiently simple so that one can show that the action of $\phi$ on every vertex group is aperiodic. This is done in Section~\ref{Section:perodicelements}.

We now turn to $\phi$-periodic conjugacy classes of free factors. In fact, it is equivalent to show that, if $(G,\mathcal{F})$ is a free product structure on $G$ such that $\mathcal{F}$ is $\phi$-periodic, then $\mathcal{F}$ is fixed by $\phi$. Let $(G,\mathcal{F})$ be such a $\phi$-periodic free product structure. Again, an inductive argument shows that we may assume that the Grushko rank of $(G,\mathcal{F})$ is maximal among all $\phi$-periodic free product structures $(G,\mathcal{F})$. There are two sorts of such maximal free product structures, depending on whether there exists or not a $\phi$-periodic tree $T$ with a $G$-action such that every edge stabiliser is trivial and $\mathcal{F}$ is the set of conjugacy classes of nontrivial vertex stabilisers of $T$. 

If there does not exist such a $\phi$-periodic tree, then one can show using standard arguments from the theory of relative train tracks (see~\cite{Lyman2022CT}) that we can associate to $\mathcal{F}$ an \emph{attracting lamination} $\Lambda_\mathcal{F}$ of $\phi$. Let $\partial_\infty(G,\mathcal{G})$ be the Bowditch boundary associated with the relatively hyperbolic structure $(G,\mathcal{G})$. A \emph{lamination} of $(G,\mathcal{G})$ is a closed, $G$-invariant, flip-invariant subset of the double boundary 
\[\partial^2(G,\mathcal{G})=\partial_\infty(G,\mathcal{G}) \times \partial_\infty(G,\mathcal{G}) - \mathrm{Diag}.\] An attracting lamination of $\phi$ is a particular $\phi$-periodic lamination of $(G,\mathcal{G})$ which encodes informations regarding the dynamics of $\phi$ on $G$. It is standard to prove that, in order to show that $\phi$ fixes $\mathcal{F}$, it suffices to show that $\phi$ fixes $\Lambda_\mathcal{F}$ (see Lemma~\ref{lem:preimagerho}). The action of $\phi$ on its (finite) set of attracting laminations is encoded in a \emph{relative train track of $\phi$}, which is a way to represent $\phi$ as a homotopy equivalence of a graph of groups. It then suffices to know that $\phi$ fixes elementwise the sets $\mathcal{G}$ and $H_1(G,\ZZ/3\ZZ)$ to show that $\phi$ fixes elementwise its set of attracting laminations as well as $\mathcal{F}$.

Finally, assume that $\mathcal{F}$ is the set of conjugacy classes of nontrivial vertex stabilisers of some $\phi$-periodic $G$-tree $T$ with trivial edge stabilisers. Let $\ell >0$ be such that $T$ is $\phi^\ell$-invariant. The key observation that we use, which we think is of independent interest, is that any $\phi^\ell$-invariant such tree $T$ induces a graph of groups $G_{\phi^\ell} \backslash T$ with infinite cyclic edge groups such that $\pi_1(G_{\phi^\ell} \backslash T) \simeq G_{\phi^\ell}$. Call such a graph of groups decomposition a \emph{$G_{\phi^\ell}$-cyclic splitting}. Conversely, any such $G_{\phi^\ell}$-cyclic splitting induces a $\phi^\ell$-invariant $G$-tree with trivial edge groups. Thus, we are reduced to understand the possible graph of groups decomposition $G_{\phi^\ell} \backslash T$ with infinite cyclic edge groups. This observation was used for instance by Dahmani~\cite{Dah2016} to solve the conjugacy problem for outer automorphisms $\phi$ of $F_N$ whose induced suspension $G_\phi$ is Gromov hyperbolic. This idea was also suggested by Mutanguha~\cite{Mutanguha24} to show that being polynomially growing among free-by-cyclic groups is a geometric invariant.

The advantage of this point of view is that we can invoke the theory of JSJ decompositions of groups introduced by Rips--Sela~\cite{RipsSela97} to understand all cyclic splittings of a given finitely presented one-ended group. This theory, combined with the construction of \emph{tree of cylinders} introduced by Guirardel--Levitt~\cite{GuirardelLevittCylinders2011}, produces an $\Aut(G_{\phi^\ell})$-invariant $G_{\phi^\ell}$-tree $U$ which encodes all $G_{\phi^\ell}$-cyclic splittings. Even though $U$ does not necessarily have cyclic edge stabilisers, the fact that it is $\Aut(G_{\phi^\ell})$-invariant is sufficient for our considerations. Indeed, the group $G_{\phi^\ell}$ is a normal subgroup of $G_\phi$, so that $G_\phi$ acts on $U$ and hence $U$ can be seen as a $\phi$-invariant $G$-tree. As $U$ encodes all $\phi$-periodic $G$-trees $T$ with trivial edge stabilisers, it therefore suffices to understand the action of $\phi$ on $U$ and its vertex stabilisers to show that $\phi$ must preserve $T$ and $\mathcal{F}$. We insist on the fact that this is the only place where we need the assumption that $G$ is finitely presented so as to apply Rips--Sela theory.

We end this introduction with an open question. Our main application of Theorem~\ref{thm:Intromain} deals with toral relatively hyperbolic groups as we manage to prove that they satisfy the \hyperlink{SA}{Standing Assumptions}. However, the \hyperlink{SA}{Standing Assumptions} are quite flexible and we hope that they are satisfied by larger classes of groups. We in particular ask the following.

\begin{ques}\label{ques:Raag}
\begin{enumerate}
    \item Let $\Gamma$ be a finite simple graph. Does the associated Right Angled Artin group $A_\Gamma$ satisfy the \hyperlink{SA}{Standing Assumptions}? 
    \item Is the kernel $\IA(A_\Gamma,3)$ of the action of $\Out(A_\Gamma)$ on $H_1(A_\Gamma,\ZZ/3\ZZ)$ an aperiodic finite index subgroup?
\end{enumerate}
\end{ques}

There are evidence towards a positive answer for Question~\ref{ques:Raag}. Firstly, the group $\IA(A_\Gamma,3)$ is torsion free (in fact of type~$\mathrm{F}$) by a result of Day--Wade (\cite[Theorem~B]{DayWade2019}). Secondly, Fioravanti~\cite{Fioravanti25} recently proved how to construct canonical (JSJ) trees for subgroups of $\Out(A_\Gamma)$ that are close in spirit to the trees we use in this paper. We hope that these trees may be useful to prove that $A_\Gamma$ satisfies the \hyperlink{SA}{Standing Assumptions}.

\medskip

We now explain the structure of the paper. In Section~\ref{Section:Background}, we present the relevant background regarding free factor systems, splittings of groups, JSJ decompositions of groups or relative train tracks. In Section~\ref{Section:Firstproperties}, we describe the \hyperlink{SA}{Standing Assumptions} and prove first immediate properties of the group $\IA(G,\mathcal{G},3)$, such as being torsion free (see Proposition~\ref{prop:IAtorsionfree}). In Section~\ref{Section:perodicelements}, we prove Theorem~\ref{thm:invariantfamilycyclicfixed}, which implies that, if $\phi \in \IA(G,\mathcal{G},3)$, every $\phi$-periodic conjugacy class of elements of $G$ is fixed by $\phi$. Section~\ref{Section:periodicffs} is devoted to the proof of Theorem~\ref{thm:periodicfreefactorfixed}, which implies that every $\phi$-periodic conjugacy class of free factors of $G$ is fixed by $\phi$. This in particular concludes the proof of Theorem~\ref{thm:Intromain}. Finally, in Section~\ref{Section:Toral}, we explain how to deduce the case where $G$ is a toral relatively hyperbolic group (Theorem~\ref{thm:IntroToral}) from Theorem~\ref{thm:Intromain}. We in particular show that any toral relatively hyperbolic group satisfies the \hyperlink{SA}{Standing Assumptions}.

\subsubsection*{Acknowledgements}

I warmly thank Naomi Andrew, Camille Horbez, Sam Hughes, Gilbert Levitt and Ric Wade for helpful conversations. I also thank the LABEX MILYON (ANR-10-LABX-0070) for its financial support during my postdoctoral position; during which part of this project was written.

\subsubsection*{Dedication}

I dedicate this paper to my grand-father, Jean-Jacques Lesage.

\section{Background}\label{Section:Background}

\subsection{Free products, Free factor systems}\label{sec:free_factors}

Let $N \in \NN$, let $G_1,\ldots,G_k$ be finitely generated groups and let $$G=G_1 \ast \ldots \ast G_k \ast F_N$$ be the corresponding free product, where $F_N$ is a free group of rank $N$. Let $\mathcal{G}=\{[G_1],\ldots,[G_k]\}$, where $[.]$ denotes the $G$-conjugacy class. The \emph{Grushko rank} of $(G,\mathcal{G})$ is $\xi(G,\mathcal{G})=k+N$.

\begin{defi}[Free factor system, free factor]
A \emph{$(G,\mathcal{G})$-free factor system} is a finite set $\mathcal{F}=\{[H_1],\ldots,[H_n]\}$ of conjugacy classes of subgroups of $G$ such that: 
\begin{itemize}
\item for every $i \in \{1,\ldots,k\}$, there exists $j \in \{1,\ldots,n\}$ such that $G_i \subseteq H_j$;
\item there exists a (necessarily free) subgroup $B$ of $G$ such that $$G=H_1 \ast \ldots \ast H_n \ast B.$$
\end{itemize}
A $(G,\mathcal{G})$-free factor system $\mathcal{F}$ is \emph{proper} if $\mathcal{F} \neq \{[G]\}$. 

A \emph{$(G,\mathcal{G})$-free factor} is a subgroup of $G$ whose conjugacy class is contained in some $(G,\mathcal{G})$-free factor system.
\end{defi}

If $\mathcal{F}$ is a $(G,\mathcal{G})$-free factor system, then the pair $(G,\mathcal{F})$ is also a free product, so that $\xi(G,\mathcal{F})$ is well-defined.
The set $\mathrm{FF}(G,\mathcal{G})$ of proper $(G,\mathcal{G})$-free factor systems is equipped with a poset structure, where $\mathcal{F}_1 \sqsubseteq \mathcal{F}_2$ if, for every $[A] \in \mathcal{F}_1$, there exists $[B] \in \mathcal{F}_2$ with $A \subseteq B$. Note that, in that case, $\xi(G,\mathcal{F}_2) \leq \xi(G,\mathcal{F}_1)$, so that 
\[\begin{array}{cccc}
     \xi \colon & \mathrm{FF}(G,\mathcal{G}) & \to & \NN  \\
     {} & \mathcal{F} & \mapsto & \xi(G,\mathcal{F})
\end{array}
\] 
is an order-reversing poset map.

A $(G,\mathcal{G})$-free factor system $\mathcal{F}$ is \emph{sporadic} if $\xi(G,\mathcal{F}) \leq 2$. This happens exactly in the following cases:

\begin{itemize}
\item $\mathcal{F}=\{[G]\}$;
\item $\mathcal{F}=\{[H_1],[H_2]\}$ and $G=H_1 \ast H_2$;
\item $\mathcal{F}=\{[H]\}$ and $G=H \ast \ZZ$.
\end{itemize}

Let $\mathcal{F}$ be a $(G,\mathcal{G})$-free factor system. We denote by $\Aut(G,\mathcal{F})$ the subgroup of $\Aut(G)$ which preserves $\mathcal{F}$. Note that an element of $\Aut(G,\mathcal{F})$ may permute the elements of $\mathcal{F}$. The group $\Inn(G)$ of inner automorphisms of $G$ is contained in $\Aut(G,\mathcal{F})$. The group $\Out(G,\mathcal{F})$ of \emph{outer automorphisms of $G$ preserving $\mathcal{F}$} is the quotient group $\Aut(G,\mathcal{F})/\Inn(G)$.  

More generally, let $\mathcal{H}$ be a finite set of conjugacy classes of subgroups of $G$. We denote by $\Out(G,\mathcal{H})$ the subgroup of $\Out(G)$ which preserves the set $\mathcal{H}$. Let $\Out(G,\mathcal{H}^{(t)})$ be the subgroup of $\Out(G,\mathcal{H})$ consisting of all $\phi \in \Out(G,\mathcal{H})$ such that, for every $[H] \in \mathcal{H}$, there exists $\Phi \in \phi$ with $\Phi(H)=H$ and $\Phi_{|H}=\mathrm{Id}_H$.

\bigskip

Let $H$ be a subgroup of $G$. The group $H$ is \emph{$(G,\mathcal{G})$-peripheral} if it is contained in a conjugate of some $G_i$ and is \emph{$(G,\mathcal{G})$-nonperipheral} otherwise. 

Assume additionally that $H$ is finitely generated. If $X$ is a subset of $H$, we denote by $[X]_H$ the $H$-conjugacy class of $X$. Any $(G,\mathcal{G})$-free factor system $\mathcal{F}$ induces a free factor system $\mathcal{F}_{|H}$ of $H$ by considering, for every $A \subseteq G$ with $[A] \in \mathcal{F}$, the $H$-conjugacy class $[H \cap A]_H$. We remark that it is possible that there exist $A,B \subseteq G$ with $[A]=[B] \in \mathcal{F}$ but $[H \cap A]_H$ and $[H \cap B]_H$ induce distinct elements of $\mathcal{F}_{|H}$. The free factor system $\mathcal{F}_{|H}$ turns $H$ into a free product $(H,\mathcal{F}_{|H})$. 

\subsection{Splittings of a group}\label{section:splitting}

Let $\Gamma$ be a finitely generated group (not necessarily a free product). Let $\mathcal{A}$ be a nonempty family of subgroups closed under conjugation and passing to subgroups. A \emph{$(\Gamma;\mathcal{A})$-splitting} is a $\Gamma$-equivariant isometry class of an action of $\Gamma$ by isometry on a simplicial tree $T$ such that: 

\begin{itemize}
\item Edge stabilisers belong to $\mathcal{A}$;
\item $\Gamma$ does not preserve any proper subtree of $T$ (we say that the action is \emph{minimal});
\item there does not exist $g \in \Gamma$ which preserves an edge of $T$ while flipping its endpoints.
\end{itemize}

We denote by $VT$ the set of vertices of $T$ and by $ET$ the set of (unoriented) edges of $T$. If $v \in VT$ (resp. $e \in ET$), we denote by $\Gamma_v$ (resp. $\Gamma_e$) the stabiliser of $v$ (resp. $e$) in $\Gamma$. A $(\Gamma;\mathcal{A})$-splitting $T$ is \emph{degenerate} if $T$ is reduced to a point. Let $v \in VT$ with $\Gamma_v$ finitely generated. If $\Gamma$ is equipped with a free product structure $(\Gamma,\mathcal{G})$ and if $\mathcal{F}$ is a $(\Gamma,\mathcal{G})$-free factor system, we will denote by $\mathcal{F}_v$ the free factor system $\mathcal{F}_{|\Gamma_v}$ of $\Gamma_v$.

We will also need relative versions of splittings. Let $H_1,\ldots,H_n \subseteq \Gamma$ be finitely generated subgroups and let $\mathcal{H}=\{[H_1],\ldots, [H_n]\}$. A \emph{$(\Gamma,\mathcal{H};\mathcal{A})$-splitting} is a $(\Gamma;\mathcal{A})$-splitting $T$ such that, for every $i \in \{1,\ldots,n\}$, the group $H_i$ fixes a point in $T$. 

We will mostly work with the following classes of splittings. 

\begin{defi}[free, cyclic splittings]
Let $\Gamma$ be a finitely generated group, let $H_1,\ldots,H_n$ be finitely generated subgroups of $\Gamma$ and let $\mathcal{H}=\{[H_1],\ldots, [H_n]\}$.
\begin{itemize}
\item A \emph{$(\Gamma,\mathcal{H})$-splitting} is a $(\Gamma,\mathcal{H};\mathcal{A})$-splitting where $\mathcal{A}$ is the family of all subgroups of $\Gamma$. 

\item A \emph{$(\Gamma,\mathcal{H})$-free splitting} is a $(\Gamma,\mathcal{H};\mathcal{A})$-splitting where $\mathcal{A}$ is the family consisting only of the trivial group.

\item A \emph{$(\Gamma,\mathcal{H})$-cyclic splitting} is a $(\Gamma,\mathcal{H};\mathcal{A})$-splitting where $\mathcal{A}$ is the family of all cyclic (possibly trivial) subgroups of $\Gamma$.
\end{itemize}
\end{defi}

A $(\Gamma,\mathcal{H})$-free splitting is in particular a $(\Gamma,\mathcal{H})$-cyclic splitting. When $\mathcal{H}=\varnothing$, we will usually refer to a $(\Gamma,\mathcal{H})$-splitting simply as a $\Gamma$-splitting.

The group $\Aut(\Gamma)$ acts on the set of $\Gamma$-splittings by precomposition of the action. Since $\Inn(\Gamma)$ acts trivially on this set (recall that a splitting is defined up to $\Gamma$-equivariant isometry), the $\Aut(\Gamma)$-action induces an action of $\Out(\Gamma)$ on the set of $\Gamma$-splittings.

Let $S$ be a $\Gamma$-splitting. We now describe the stabiliser $\Stab(S)$ of $S$ in $\Out(\Gamma)$. This description is due to Levitt~\cite{levitt2005}.

Note that, since $\Gamma$ is finitely generated and since the action of $\Gamma$ on $S$ is minimal, the quotient $\Gamma \backslash S$ has the structure of a finite graph of groups. The group $\Stab(S)$ acts by graph automorphism on the underlying graph $\overline{\Gamma \backslash S}$ of $\Gamma \backslash S$. This gives a homomorphism $$\Stab(S) \to \Aut_{gr}(\overline{\Gamma \backslash S}),$$ where $\Aut_{gr}(\overline{\Gamma \backslash S})$ is the group of graph automorphisms of $\overline{\Gamma \backslash S}$. Since $\overline{\Gamma \backslash S}$ is a finite graph, the kernel $\mathcal{K}(S)$ of the above map is a finite index subgroup of $\Stab(S)$. 

Let $\phi \in \mathcal{K}(S)$ and let $v \in VS$. Any representative of $\phi$ induces a unique $\Gamma$-equivariant isometry of $S$. If $\Phi_1,\Phi_2$ are two representatives of $\phi$ which fix the vertex $v$, then they differ by an inner automorphism $\mathrm{ad}_g$ with $g\in \Gamma_v$. This shows that, to any outer automorphism $\phi \in \mathcal{K}(S)$, one can naturally associate an outer automorphism $\phi_v \in \Out(\Gamma_v)$. Similarly, for every edge $e \in ES$ and every $\phi \in \mathcal{K}(S)$, one can naturally associate an outer automorphism $\phi_e \in \Out(\Gamma_e)$. We thus have a natural restriction homomorphism $$\mathcal{K}(S) \to \prod_{v \in V\overline{\Gamma\backslash S}} \Out(\Gamma_v),$$ whose kernel is denoted by $\mathcal{T}(S)$.

The following proposition describes the structure of $\mathcal{T}(S)$ when $S$ is a $\Gamma$-free splitting.

\begin{prop}[{\cite[Proposition~3.1]{levitt2005}}] \label{lem:grouptwists}
Let $S$ be a $\Gamma$-free splitting. The group $\mathcal{T}(\mathcal{S})$ is isomorphic to a direct product $$\prod_{v \in V\overline{\Gamma \backslash S}}\Gamma_v^{n_{v}}/Z(\Gamma_v),$$ where $n_v$ is the valence of the vertex $v$ of $\overline{\Gamma \backslash S}$ and $Z(\Gamma_v)$ embeds diagonally in $\Gamma_v^{n_{v}}$.
\end{prop} 

When $S$ is a $\Gamma$-free splitting, we also have another description of the group $\mathcal{K}(S)$.

\begin{prop}[{\cite[Proposition~4.2]{levitt2005}}]\label{prop:stabfreesplittingautversion}
Let $S$ be a $\Gamma$-free splitting. The group $\mathcal{K}(S)$ splits as a direct product $$\mathcal{K}(S)=\prod_{v \in V\overline{\Gamma \backslash S}} M_v $$ where, for every $v \in V\overline{\Gamma \backslash S}$, there exists $n_v \geq 1$ such that the group $M_v$ fits in a short exact sequence $$1 \to \Gamma_v^{n_v-1} \to M_v \to \Aut(\Gamma_v) \to 1.$$ 
\end{prop}

\subsection{JSJ decompositions of finitely presented groups and trees of cylinders}

Let $\Gamma$ be a finitely presented group. Let $H_1,\ldots,H_n \subseteq \Gamma$ be finitely generated subgroups and let $\mathcal{H}=\{[H_1],\ldots, [H_n]\}$. In this section, we define the notion of a \emph{JSJ decomposition of $\Gamma$ relative to $\mathcal{H}$}. JSJ decompositions of groups are $\Gamma$-splittings which have been introduced by Sela~\cite{Sela97JSJ} for hyperbolic groups and greatly generalised afterwards~\cite{RipsSela97,Bowditch98,DunwoodySageev99,FujPap2006}. Here we present the point of view due to Guirardel--Levitt~\cite{guirardel2016jsj}.

We will give two constructions of JSJ decompositions. One is general and does not require any other assumption on $\Gamma$. The second one requires that $\Gamma$ is hyperbolic relative to a finite set of finitely generated subgroups. The advantage of the second construction is that the JSJ decomposition is invariant under every automorphism of $\Gamma$ preserving the parabolic structure. Moreover, as explained below (see~Theorem~\ref{thm:JSJrelhyp}), we have a better understanding of the action of $\Aut(\Gamma)$ on the so-called \emph{rigid vertices} of the JSJ decomposition.

\subsubsection{On the terminology of JSJ decompositions of groups}\label{sec:terminologyJSJ}

Let $\mathcal{A}$ be a family of subgroups of $\Gamma$, closed under taking subgroups and conjugation.

\begin{defi}[Universally elliptic splitting]
Let $S$ be a $(\Gamma,\mathcal{H};\mathcal{A})$-splitting. We say that $S$ is \emph{universally elliptic} if, for every $e \in ES$, the group $\Gamma_e$ is elliptic in every $(\Gamma,\mathcal{H};\mathcal{A})$-splitting.
\end{defi}

\begin{defi}[Rigid, flexible vertex]
Let $S$ be a $(\Gamma,\mathcal{H};\mathcal{A})$-splitting. Let $v \in VS$. We say that $v$ is a \emph{rigid vertex} if $\Gamma_v$ is elliptic in every $(\Gamma,\mathcal{H};\mathcal{A})$-splitting. Otherwise, we say that $v$ is a \emph{flexible vertex}.
\end{defi}

Let $\Sigma$ be a connected, compact, hyperbolic (not necessarily orientable) surface. A \emph{boundary subgroup of $\Sigma$} is a subgroup of $\pi_1(\Sigma)$ conjugate to the fundamental group $B=\pi_1(C)$ of a boundary component of $\Sigma$.

\begin{defi}[QH vertex, mapping class group]
Let $S$ be a $(\Gamma,\mathcal{H};\mathcal{A})$-splitting. Let $v \in VS$. We say that $v$ is a \emph{QH vertex} if there exists a connected, compact, hyperbolic (not necessarily orientable) surface $\Sigma_v$ such that the following hold. 
\begin{itemize}
\item The group $\Gamma_v$ is isomorphic to $\pi_1(\Sigma_v)$.
\item For every $e \in ES$ adjacent to $v$, the group $\Gamma_e$ is contained in a boundary subgroup of $\Sigma_v$.
\item If $[H] \in \mathcal{H}$ is such that $H \cap \Gamma_v \neq \{1\}$, then $H \cap \Gamma_v$ is contained in a boundary subgroup of $\Sigma_v$.
\end{itemize}

The \emph{mapping class group of $\Gamma_v$}, denoted by $\mathrm{MCG}(\Gamma_v)$, is the subgroup of $\Out(\Gamma_v)$ preserving the set of conjugacy classes of boundary subgroups of $\Sigma_v$.
\end{defi}

Note that, if $v$ is a QH vertex of a splitting $S$ and if $S$ is not reduced to a point, then $\Gamma_v$ is infinitely-ended since $\partial\Sigma_v \neq \varnothing$.

\subsubsection{JSJ splittings, tree of cylinders}

Recall that $\Gamma$ is \emph{one-ended relative to $\mathcal{H}$} if there does not exist a nondegenerate $(\Gamma,\mathcal{H})$-splitting with finite edge stabilisers. 


\begin{theo}[{\cite[Theorems~2.20,~6.2]{guirardel2016jsj}}]\label{thm:JSJ}
Let $\Gamma$ be a torsion free finitely presented group and let $\mathcal{H}$ be a finite set of conjugacy classes of finitely generated subgroups of $\Gamma$. Suppose that $\Gamma$ is one-ended relative to $\mathcal{H}$. There exists a $(\Gamma,\mathcal{H})$-cyclic splitting $T_{JSJ}^{\mathcal{H}}$ with the following properties.

\begin{enumerate}
\item The tree $T_{JSJ}^{\mathcal{H}}$ is universally elliptic.
\item Edge stabilisers of $T_{JSJ}^{\mathcal{H}}$ are infinite cyclic. 
\item Vertex stabilisers of $T_{JSJ}^{\mathcal{H}}$ are elliptic in every universally elliptic splitting.
\item If $v \in VT_{JSJ}^\mathcal{H}$, one of the following hold.
\begin{enumerate}
\item The group $\Gamma_v$ is infinite cyclic.
\item The vertex $v$ is QH. In particular, its stabiliser is infinitely ended if $T_{JSJ}^\mathcal{H}$ is not reduced to a point.
\item The vertex $v$ is rigid.
\end{enumerate}
\end{enumerate}
\end{theo}

A splitting $T_{JSJ}^{\mathcal{H}}$ given by Theorem~\ref{thm:JSJ} is called a \emph{JSJ $(\Gamma,\mathcal{H})$-splitting}. In general, a JSJ $(\Gamma,\mathcal{H})$-splitting is not unique. However, if $T$ and $T'$ are two distinct JSJ $(\Gamma,\mathcal{H})$-splittings, they have the same \emph{tree of cylinders}, as defined by Guirardel--Levitt~\cite{GuirardelLevittCylinders2011}. Let $T$ be a $\Gamma$-cyclic splitting. Let $\mathcal{E}$ be the equivalence relation on the edges of $T$ generated by $e \sim e'$ if and only if $\Gamma_e$ and $\Gamma_{e'}$ have a common finite index subgroup. The equivalence relation $\mathcal{E}$ is an \emph{admissible relation} in the sense of \cite[Definition~3.1]{GuirardelLevittCylinders2011}. A \emph{cylinder of $T$} is a subforest of $T$ consisting of all the edges in a given equivalence class.  By~\cite[Lemma~4.1]{GuirardelLevittCylinders2011}, a cylinder is a subtree of $T$. 

\begin{defi}[Tree of cylinders]
Let $T$ be a $\Gamma$-cyclic splitting. The \emph{tree of cylinders $T_c$ of $T$} is the bipartite tree $VT_c=V_0 \amalg V_1$ defined as follows:
\begin{enumerate}
\item there is one vertex in $V_0$ for every vertex of $T$ which is not contained in at most one cylinder;
\item there is one vertex in $V_1$ for every cylinder of $T$;
\item there exists an edge $(x,C)$ between $x \in V_0$ and $C \in V_1$ if and only if $x$ (considered as a vertex of $T$) is contained in $C$ (considered as a subtree of $T$).
\end{enumerate}
\end{defi}

The tree of cylinders of a $\Gamma$-cyclic splitting is naturally equipped with a $\Gamma$-action, which turns it into a $\Gamma$-splitting. However, beware that, in general, the tree of cylinders of a $\Gamma$-cyclic splitting is not necessarily a $\Gamma$-cyclic splitting. Indeed, edge stabilisers of a tree of cylinders are not necessarily cyclic.

The advantage of working with trees of cylinders instead of the JSJ $\Gamma$-splittings comes from the fact that it is canonical in the following sense.

\begin{prop}[{\cite[Corollary~4.10]{GuirardelLevittCylinders2011}}]\label{prop:treecylinderscanonical}
Let $\Gamma$ be a finitely presented torsion free group and let $\mathcal{H}$ be a finite set of conjugacy classes of finitely generated subgroups of $\Gamma$. Suppose that $\Gamma$ is one-ended relative to $\mathcal{H}$. 

If $T_1$ and $T_2$ are two JSJ $(\Gamma,\mathcal{H})$-splittings, their trees of cylinders are $\Gamma$-equivariantly isomorphic. In particular, if $T$ is a JSJ $(\Gamma,\mathcal{H})$-splitting, then its tree of cylinders $T_c$ is $\Aut(\Gamma,\mathcal{H})$-invariant.
\end{prop}

Let $T_1$, $T_2$ be two $\Gamma$-splittings. We say that $T_1$ \emph{collapes onto $T_2$} if $T_2$ is obtained from $T_1$ by collapsing some orbits of edges. In that case, we say that $T_1$ is a \emph{refinement of $T_2$}. Two trees $T_1$ and $T_2$ are \emph{compatible} if they have a common refinement. The following lemma shows that the tree of cylinders of a JSJ $\Gamma$-splitting is compatible with every $\Gamma$-cyclic splitting.

\begin{lem}\label{lem:treecylinderscompatible}
Let $\Gamma$ be a finitely presented torsion free group and let $\mathcal{H}$ be a finite set of conjugacy classes of finitely generated subgroups of $\Gamma$. Suppose that $\Gamma$ is one-ended relative to $\mathcal{H}$ and let $T$ be a JSJ $(\Gamma,\mathcal{H})$-splitting.

Then the tree of cylinders $T_c$ of $T$ is compatible with every $(\Gamma,\mathcal{H})$-cyclic splitting. Moreover, if all vertex stabilisers of $T$ are rigid, a common refinement is obtained by equivariantly blowing-up trees at vertices of $T_c$ corresponding to cylinder subtrees of $T$.
\end{lem}

\begin{proof}
The first assertion follows from~\cite[Corollary~8.4]{GuirardelLevittCylinders2011}. Even though the statement only covers the case $\mathcal{H}=\varnothing$, the proof is identical in the relative setting. 

If all vertex stabilisers of $T$ are rigid, then, in the language of Guirardel--Levitt, $T$ \emph{dominates} every $(\Gamma,\mathcal{H})$-cyclic splitting. Therefore, the last assertion follows from the construction of the compatible tree in the proof of~\cite[Proposition~8.1]{GuirardelLevittCylinders2011}. 
\end{proof}

\subsubsection{The relatively hyperbolic construction}

Recall the definition of a relatively hyperbolic group from for instance~\cite{Bowditch2012}. Let $\Gamma$ be a finitely presented torsion free group. Suppose that $\Gamma$ is hyperbolic relative to a set $\mathcal{P}=\{[P_1],\ldots,[P_k]\}$ of conjugacy classes of finitely generated subgroups. Let $H_1,\ldots,H_n \subseteq \Gamma$ be finitely generated subgroups and let $\mathcal{H}=\{[H_1],\ldots, [H_n]\}$.

\begin{defi}[Elementary subgroups, splittings]

\begin{enumerate}
\item A subgroup of $\Gamma$ is \emph{elementary} if it is cyclic or conjugate into one of the groups $P_i$ with $i \in \{1,\ldots,k\}$.

\item A \emph{$(\Gamma,\mathcal{P} \cup \mathcal{H})$-elementary splitting} is a $(\Gamma,\mathcal{P} \cup \mathcal{H};\mathcal{A})$-splitting, where $\mathcal{A}$ is the family of elementary subgroups.
\end{enumerate}
\end{defi}

The advantage of working with relatively hyperbolic groups is that one can construct a \emph{canonical elementary splitting}, as explained in the following theorem. Recall that a $\Gamma$-splitting $S$ is \emph{$K$-acylindrical} for some $K>0$ if the $\Gamma$-stabiliser of any arc in $S$ of length at least $K+1$ is finite.

\begin{theo}[{\cite{GuirardelLevitt2015,guirardel2016jsj}}]\label{thm:JSJrelhyp}
Let $\Gamma$ be a torsion free finitely generated hyperbolic group relative to $\mathcal{P}$ and let $\mathcal{H}$ be a finite set of conjugacy classes of finitely generated subgroups of $\Gamma$. Suppose that $\Gamma$ is one-ended relative to $\mathcal{P}\cup \mathcal{H}$. There exists a $(\Gamma,\mathcal{P}\cup \mathcal{H})$-elementary splitting $T_{\mathcal{P}\cup \mathcal{H}}$ with the following properties.
\begin{enumerate}
\item \cite[Theorem~9.18]{guirardel2016jsj} Edge stabilisers are infinite elementary. Moreover, every edge has one endpoint with elementary stabiliser and one endpoint with nonelementary stabiliser.
\item \cite[Proposition~5.13]{GuirardelLevittCylinders2011} The action of $\Gamma$ on $T_{\mathcal{P}\cup \mathcal{H}}$ is $2$-acylindrical.
\item Any elementary subgroup fixes at most one vertex with elementary stabiliser.
\item \cite[Lemma~3.7]{GuirardelLevitt2015} Edge and vertex stabilisers are finitely generated.
\item \cite[Theorem~9.18]{guirardel2016jsj} If $v \in VT_{\mathcal{P}\cup \mathcal{H}}$, one of the following hold:
\begin{enumerate}
\item the group $\Gamma_v$ is two-ended;
\item the conjugacy class of $\Gamma_v$ is contained in $\mathcal{P}$;
\item the vertex $v$ is nonelementary and QH;
\item the vertex $v$ is nonelementary and rigid.

\end{enumerate}
\item \cite[Theorem~9.18]{guirardel2016jsj} The group $\Aut(\Gamma,\mathcal{P}, \mathcal{H})$ of automorphisms preserving $\mathcal{P}$ and $\mathcal{H}$ preserves the $\Gamma$-equivariant isometry class of $T_{\mathcal{P}\cup \mathcal{H}}$.
\item \cite[Proposition~4.1]{GuirardelLevitt2015} For every edge $e \in ET_{\mathcal{P}\cup \mathcal{H}}$, the image $$\mathcal{K}( T_{\mathcal{P}\cup \mathcal{H}}) \cap \Out(\Gamma,\mathcal{P},\mathcal{H}^{(t)}) \to \Out(\Gamma_e)$$ is finite.
\item \cite[Proposition~4.1]{GuirardelLevitt2015} For every rigid vertex $v \in VT_{\mathcal{P}\cup \mathcal{H}}$, the image $$\mathcal{K}( T_{\mathcal{P}\cup \mathcal{H}})\cap \Out(\Gamma,\mathcal{P},\mathcal{H}^{(t)}) \to \Out(\Gamma_v)$$ is finite.
\item \cite[Proposition~4.1]{GuirardelLevitt2015} For every QH vertex $v \in VT_{\mathcal{P}\cup \mathcal{H}}$, the image $\mathcal{K}(T_{\mathcal{P}\cup \mathcal{H}}) \to \Out(\Gamma_v)$ is contained in $\mathrm{MCG}(\Gamma_v)$.
\end{enumerate}
\end{theo}

The splitting $T_{\mathcal{P}\cup \mathcal{H}}$ given by Theorem~\ref{thm:JSJrelhyp} is called the \emph{JSJ $(\Gamma,\mathcal{P}\cup \mathcal{H})$-elementary splitting}.

\begin{rmq}
Theorem~\ref{thm:JSJrelhyp}$(3)$ follows from the following facts. First, the tree $T_{\mathcal{P} \cup \mathcal{H}}$ is a tree of cylinders for the equivalence relation of coelementarity~\cite[Proposition~6.1]{GuirardelLevittCylinders2011}. Moreover, vertices of $T_{\mathcal{P} \cup \mathcal{H}}$ with elementary stabilisers correspond to vertices associated with cylinders. Finally, any elementary subgroup preserves at most one cylinder.
\end{rmq}

\begin{rmq}\label{Rmk:3puncturerigid}
If $v \in VT_{\mathcal{P}\cup \mathcal{H}}$ is a QH vertex such that $\Sigma_v$ is isomorphic to a three-punctured sphere, then $\mathrm{MCG}(\Gamma_v)$ is finite. This is why, in the rest of the paper, we will assume that, if $v \in VT_{\mathcal{P}\cup \mathcal{H}}$ is a QH vertex such that $\Sigma_v$ is isomorphic to a three-punctured sphere, then $v$ is a rigid vertex. This is coherent with Theorem~\ref{thm:JSJrelhyp}$(8)$ and the assumptions of Guirardel--Levitt in~\cite[Remark~3.6]{GuirardelLevitt2015}.
\end{rmq}

We record an immediate consequence of Theorem~\ref{thm:JSJrelhyp}.

\begin{lem}\label{lem:elemornonelem}
Let $\Gamma$ be a torsion free finitely generated hyperbolic group relative to $\mathcal{P}$ and let $\mathcal{H}$ be a finite set of finitely generated subgroups of $\Gamma$. Suppose that $\Gamma$ is one-ended relative to $\mathcal{P}\cup \mathcal{H}$.

Let $H$ be a subgroup of $\Gamma$ which fixes a point in $T_{\mathcal{P}\cup \mathcal{H}}$. Then $H$ fixes a unique elementary vertex of $T_{\mathcal{P}\cup \mathcal{H}}$ or a unique nonelementary vertex of $T_{\mathcal{P}\cup \mathcal{H}}$.
\end{lem}

\begin{proof}
This follows from the fact that the action of $\Gamma$ on $T_{\mathcal{P}\cup \mathcal{H}}$ is $2$-acylindrical (Theorem~\ref{thm:JSJrelhyp}$(2)$) together with the fact that every edge has an elementary endpoint and a nonelementary endpoint (Theorem~\ref{thm:JSJrelhyp}$(1)$).
\end{proof}

\subsection{The outer space of a free product}

Let $(G,\mathcal{G})$ be a free product of groups. In this section, following Guirardel--Levitt~\cite{Guirardel} (see~\cite{Vogtmann1986} for the case $(G,\mathcal{G})=(F_N,\varnothing)$), we define a simplicial graph on which $\Out(G,\mathcal{G})$ acts cocompactly, called the \emph{spine of outer space of $(G,\mathcal{G})$}. An important property that we will use is that any finite subgroup of $\Out(G,\mathcal{G})$ fixes a point in it.

\begin{defi}[Grushko free splitting]
A \emph{Grushko $(G,\mathcal{G})$-free splitting} is a $(G,\mathcal{G})$-free splitting $T$ such that, for every $v \in VT$, either $G_v$ is trivial or $[G_v] \in \mathcal{G}$.
\end{defi}

\begin{defi}[{\cite[Spine of outer space]{Guirardel}}]
Let $(G,\mathcal{G})$ be a free product. The \emph{spine of outer space of $(G,\mathcal{G})$}, denoted by $\mathcal{O}(G,\mathcal{G})$, is the graph whose vertices are the Grushko $(G,\mathcal{G})$-free splittings, two such splittings $S$ and $T$ being adjacent if $S$ collapses onto $T$ or conversely. 
\end{defi}

The group $\Aut(G,\mathcal{G})$ acts on $\mathcal{O}(G,\mathcal{G})$ by precomposition of the action. Since the group $\Inn(G)$ acts trivially on $\mathcal{O}(G,\mathcal{G})$, this action passes to the quotient to give an action of $\Out(G,\mathcal{G})$ on $\mathcal{O}(G,\mathcal{G})$. 

We now recall a theorem, due to Hensel--Kielak~\cite{hensel2018}, describing the action of finite subgroups of $\Out(G,\mathcal{G})$ on $\mathcal{O}(G,\mathcal{G})$.

\begin{theo}[{\cite[Corollary~6.1]{hensel2018}}]\label{thm:fixedpointouterspace}
Let $(G,\mathcal{G})$ be a free product of groups (recall that $G$ is finitely generated) and let $H$ be a finite subgroup of $\Out(G,\mathcal{G})$. Then $H$ fixes a point of $\mathcal{O}(G,\mathcal{G})$.
\end{theo}

We will also frequently construct $(G,\mathcal{G})$-free factor systems using Grushko $(G,\mathcal{G})$-free splittings, as follows. Let $S$ be a Grushko $(G,\mathcal{G})$-free splitting and let $T$ be a $G$-equivariant subforest of $S$ containing every vertex with nontrivial stabiliser. The quotient $G \backslash T$ is a (not necessarily connected) graph of groups, which identifies with a subgraph of groups of $G \backslash S$. Let $U_1,\ldots,U_t$ be the connected components of $G \backslash T$ (considered as graphs of groups) and let $\mathcal{F}(T)=\{[\pi_1(U_i)]\}_{i=1}^t$, where $[\pi_1(U_i)]$ is the conjugacy class of the fundamental group $\pi_1(U_i)$ seen as a subgroup of $G$. Since $T$ contains every vertex with nontrivial stabiliser, the set $\mathcal{F}(T)$ is a $(G,\mathcal{G})$-free factor system. We call $\mathcal{F}(T)$ the \emph{$(G,\mathcal{G})$-free factor system induced by $T$}.

\subsection{Double boundary of a free product and its laminations}

Let $(G,\mathcal{G})$ be a free product of groups. The group $G$ is hyperbolic relative to the set $\mathcal{G}$. Therefore, one can consider the \emph{Bowditch boundary} $\partial_\infty(G,\mathcal{G})$ of $(G,\mathcal{G})$ (see~\cite{Bowditch2012}). It may be described as follows. Let $S$ be a Grushko $(G,\mathcal{G})$-free splitting. Let $\partial_\infty S$ be the Gromov boundary of $S$ and let $V_\infty(S)$ be the set of vertices of $S$ with infinite valence. Let $\overline{S}=S \cup \partial_\infty S$, equipped with the metric topology. For two distinct points $x,y \in \overline{S}$, the \emph{direction of $x$ at $y$} is the connected component of $\overline{S}\setminus \{y\}$ which contains $x$. The \emph{observers' topology} for $\overline{S}$ (see~\cite{CouHilLus07}) is the topology generated by the sets of directions in $\overline{S}$. The \emph{Bowditch boundary} $\partial_\infty(G,\mathcal{G})$ of $(G,\mathcal{G})$ is the set $\partial_\infty S \cup V_\infty(S)$ equipped with the observers' topology. Its homeomorphism class does not depend on the choice of $S$.

Let $$\partial^2(G,\mathcal{G})=(\partial_\infty(G,\mathcal{G}) \times \partial_\infty(G,\mathcal{G}))-\Delta/\sim $$ be the \emph{double boundary of $(G,\mathcal{G})$}, where $\Delta$ is the diagonal and $\sim$ is the equivalence relation generated by the flip $(x,y) \sim (y,x)$. 
One can geometrically see an element of $\partial^2(G,\mathcal{G})$ as follows. A point $\{x,y\}$ of $\partial^2(G,\mathcal{G})$ is represented by a (not necessarily finite) unoriented geodesic path between the points of $\overline{S}$ representing $x$ and $y$. For this reason, we call an element of $\partial^2(G,\mathcal{G})$ a \emph{line of $(G,\mathcal{G})$}.

Note that the group $G$ naturally acts on $\partial^2(G,\mathcal{G})$. A \emph{lamination} of $(G,\mathcal{G})$ is a closed $G$-invariant subset of $\partial^2(G,\mathcal{G})$. 

Let $A$ be a $(G,\mathcal{G})$-free factor. The inclusion $A \subseteq G$ induces an $A$-equivariant inclusion $\partial_\infty (A,\mathcal{G}_{|A}) \hookrightarrow \partial_\infty (G,\mathcal{G})$ and an $A$-equivariant inclusion $\partial^2(A,\mathcal{G}_{|A}) \hookrightarrow \partial^2(G,\mathcal{G})$ with closed image. We identify $\partial^2(A,\mathcal{G}_{|A})$ with its image in $\partial^2(G,\mathcal{G})$. A $(G,\mathcal{G})$-free factor system $\mathcal{F}$ induces a lamination $\partial^2(\mathcal{F},\mathcal{G})$ of $(G,\mathcal{G})$ by looking at the union of the (disjoint) sets $g\partial^2(A,\mathcal{G}_{|A})$ with $[A] \in \mathcal{F}$ and $g \in G$.

\subsection{Relative train track maps for free products}

Let $(G,\mathcal{G})$ be a free product of groups. Let $\phi \in \Out(G,\mathcal{G})$. In this section, we recall the construction of particular topological representatives of $\phi$, called \emph{relative train tracks}. They encode several dynamical invariants of $\phi$. Relative train tracks were first introduced by Bestvina--Handel~\cite{BesHan92} in the context of automorphisms of free groups and were then extended to the context of free products by Collins--Turner~\cite{ColTur94} (see also the work of Francaviglia--Martino~\cite{FrancaMar2015} and Lyman~\cite{Lyman2022,Lyman2022CT}). 

Let $T$ be a Grushko $(G,\mathcal{G})$-free splitting and let $\Phi \in \Aut(G,\mathcal{G})$. A \emph{morphism $F \colon T \to T$ associated with $\Phi$} is a map sending vertices to vertices, edges of $S$ to nontrivial edge paths and such that, for all $g \in G$ and $x \in T$, we have $$F(gx)=\Phi(g)F(x).$$ 

Let $\gamma$ be a path in $T$. There exists a unique geodesic in $T$ which is homotopic to $\gamma$ relative to its endpoints. This geodesic is denoted by $[\gamma]$ and is called the \emph{reduced path} associated with $\gamma$. 

A \emph{filtration of $T$} is an increasing sequence $$\varnothing=T_0 \subseteq T_1 \subseteq \ldots \subseteq T_n=T$$ of subgraphs which are both $G$ and $F$-invariant. The graphs of the filtration are not necessarily connected.

Let $r \in \{0,\ldots,n\}$. The \emph{$r$-th stratum of $T$}, denoted by $H_r$, is the closure of $T_r-T_{r-1}$. 

To every stratum $H_r$, one can naturally associate a \emph{transition matrix $M_r$} as follows. The matrix $M_r$ is a square matrix indexed by the $G$-orbits of unoriented edges of $H_r$. Let $e,e' \in EH_r$. Then the value at the $(Ge',Ge)$-entry of $M_r$ is equal to the number of times that the path $[F(e)]$ contains edges in $Ge'$. Since $F$ is $\Phi$-equivariant, this does not depend on the choice of a representative of $Ge$. A filtration is \emph{maximal} if every transition matrix is either irreducible or the zero matrix.

Suppose that the transition matrix $M_r$ is irreducible. Associated to $M_r$ is a Perron--Froebenius eigenvalue $\lambda_r \geq 1$. We say that $H_r$ is \emph{exponentially growing (EG)} if $\lambda_r >1$ and \emph{nonexponentially growing (NEG)} otherwise.

Let $v \in VT$. Recall that a \emph{direction at $v$} is a connected component of $T \setminus \{v\}$, or equivalently an oriented edge whose origin is $v$. If $d$ is a direction at $v$, then $F$ sends $d$ to a direction $DF(d)$ at $F(v)$. Thus, the map $F$ induces a map $DF$ of the set of directions of $T$. A \emph{turn} is an unordered pair of directions at a vertex of $T$. A turn $\{d,d'\}$ is \emph{degenerate} if $d=d'$, it is \emph{illegal} if there exists $m \geq 1$ such that $(DF)^m(d)=(DF)^m(d')$, and is \emph{legal} otherwise. 

Let $\gamma=e_1\ldots e_m$ be an oriented geodesic edge path in $T$. Note that, for every $i \in \{1,\ldots,m\}$, the set $\{e_i^{-1},e_{i+1}\}$ is a turn. We say that $\gamma$ is a \emph{legal path} if, for every $i \in \{1,\ldots,m\}$, the turn $\{e_i^{-1},e_{i+1}\}$ is legal.

Let $r \in \{0,\ldots,n\}$. A turn $\{e,e'\}$ is of \emph{height $r$} if $e,e' \in EH_r$. A geodesic edge path $\gamma$ is \emph{$r$-legal} if all its height $r$ turns are legal.

We can now define a relative train track map.

\begin{defi}[Relative train track map]
Let $\Phi \in \Aut(G,\mathcal{G})$, let $F \colon T \to T$ be a morphism associated with $\Phi$ and let $\varnothing=T_0 \subseteq T_1 \subseteq \ldots \subseteq T_n=T$ be a filtration. The map $F$ is a \emph{relative train track} if the filtration is maximal and if, for every EG stratum $H_r$, the following hold.

\begin{enumerate}
\item Directions in $H_r$ are mapped by $DF$ to directions in $H_r$.

\item If $\sigma \subseteq T_{r-1}$ is a nondegenerate geodesic edge path with endpoints in $H_r$, then $[F(\sigma)]$ is a nondegenerate geodesic edge path with endpoints in $H_r$.

\item If $\sigma$ is an $r$-legal geodesic edge path, so is $[F(\sigma)]$. 
\end{enumerate}
\end{defi}

We can now state the main existence theorem of relative train tracks. It appears in several forms in the literature (see~\cite{ColTur94},~\cite[Corollary~8.25]{FrancaMar2015},~\cite[Theorem~A]{Lyman2022}).

\begin{theo}[{\cite{ColTur94,FrancaMar2015,Lyman2022}}]\label{thm:existencereltraintracks}
Let $(G,\mathcal{G})$ be a free product and let $\phi \in \Out(G,\mathcal{G})$. There exists a relative train track map $F \colon T \to T$ associated with an automorphism representing $\phi$. Moreover, if $\phi$ preserves a $(G,\mathcal{G})$-free factor system $\mathcal{F}$, one may choose $F \colon T \to T$ such that there exists an element $T_i$ of the filtration of $T$ with $\mathcal{F}(T_i)=\mathcal{F}$.
\end{theo}

\subsection{Relatively fully irreducible outer automorphisms}

Let $(G,\mathcal{G})$ be a free product and let $\phi \in \Out(G,\mathcal{G})$. We say that $\phi$ is \emph{fully irreducible relative to $\mathcal{G}$} if, for every $n>0$, the outer automorphism $\phi^n$ does not preserve any proper $(G,\mathcal{G})$-free factor system distinct from $\mathcal{G}$. The main property we will use regarding fully irreducible outer automorphisms is the existence of an $\RR$-tree equipped with a $(G,\mathcal{G})$-action obtained as a limit of iterates of a relative train track.

\begin{theo}[\cite{FrancaMarSyri2021}]\label{theo:limittreeiwip}
Let $(G,\mathcal{G})$ be a nonsporadic free product, let $\phi \in \Out(G,\mathcal{G})$ be fully irreducible relative to $\mathcal{G}$ and let $F \colon T \to T$ be a relative train track representative of $\phi$. There exist $\lambda>1$ and an $\RR$-tree $X_\phi$ equipped with an action of $G$ by isometries such that
\[X_\phi=\lim_{n \to \infty} \frac{1}{\lambda^n}\phi^nT. \]

The tree $X_\phi$ does not depend on the choice of the relative train track map $F \colon T \to T$ and has the following properties.

\begin{enumerate}
    \item Arc stabilisers in $X_\phi$ are trivial.
    \item There exists a (possibly trivial) cyclic subgroup $H \subseteq G$ such that the set of conjugacy classes of point stabilisers in $X_\phi$ is exactly $\mathcal{G} \cup \{[H]\}$.
\end{enumerate}
\end{theo}

\begin{rmq}
The existence of $X_\phi$ follows from \cite[Lemma 2.14.1]{FrancaMarSyri2021}. The fact that it does not depend on the relative train track map follows from~\cite[Theorem~5.1.1]{FrancaMarSyri2021}. The description of the action of $G$ on $X_\phi$ follows for instance from the fact that $X_\phi$ is \emph{arational} (see~\cite[Theorems 3.4, 4.1]{Guirardelhorbez19}).
\end{rmq}

\section{First properties of congruence subgroups}\label{Section:Firstproperties}

Let $(G,\mathcal{G})$ be a free product of groups with $G=G_1 \ast \ldots \ast G_k \ast F_N$. We now describe the finite index subgroup of $\Out(G,\mathcal{G})$ which will be of central interest in the rest of the paper. 

The group $\Out(G,\mathcal{G})$ naturally permutes the conjugacy classes in $\mathcal{G}$. The kernel of this action is a finite index subgroup, denoted by $\Out^0(G,\mathcal{G})$. We now pass to a further finite index subgroup using the action on homology. The group $\Out^0(G,\mathcal{G})$ naturally acts on the finite group $H_1(G,\ZZ/3\ZZ)$. We denote the kernel of this action by $\IA^0(G,\mathcal{G},3)$. We record in this section the first immediate aperiodicity properties of $\IA^0(G,\mathcal{G},3)$. 

We start with some standard lemmas regarding the homology of a graph of groups. If $S$ is a splitting of a group $G$, recall the definition of $\mathcal{K}(S)$ from Section~\ref{section:splitting}.

\begin{lem}\label{lem:MayerVietoris}
Let $G$ be a finitely generated group, let $A$ be a trivial $G$-module and let $\mathbb{X}$ be the graph of groups associated with a $G$-splitting $S$. Let $v \in \overline{\mathbb{X}}$ and let $\mathrm{Inc}_v$ be the set of $G_v$-conjugacy classes of edge groups associated with oriented edges of $\overline{\mathbb{X}}$ with origin $v$. 

Let $H$ be the image of $H_1(G_v,A)$ in $H_1(G,A)$ and let $K_{\mathrm{Inc}_v}$ be the subgroup of $H_1(G_v,A)$ generated by the images of $H_1(G_e,A)$ with $[G_e]_{G_v} \in \mathrm{Inc}_v$ in $H_1(G_v,A)$. 

There exists a subgroup $K \subseteq K_{\mathrm{Inc}_v}$ such that $H$ fits in an exact sequence 
\begin{equation}\label{eq:MayerVietoris}
    K \to H_1(G_v,A) \to H.
\end{equation}
Moreover, for every $\phi \in \mathcal{K}(S)$, the sequence in Equation~\eqref{eq:MayerVietoris} is $\phi$-equivariant.
\end{lem}

\begin{proof}
We denote by $\vec{E}\overline{\mathbb{X}}$ the set of oriented edges of $\overline{\mathbb{X}}$. For every $w \in V\overline{\mathbb{X}}$ (resp. $e \in \vec{E}\overline{\mathbb{X}}$), let $(X_w,\ast_w)$ (resp. $(X_e,\ast_e)$) be a pointed CW complex with contractible universal cover whose fundamental group is isomorphic to $G_w$ (resp. $G_e$). We additionally assume that, for every $e \in  \vec{E}\overline{\mathbb{X}}$ with origin $w$, there exists an embedding $\iota_{w,e}\colon X_e \hookrightarrow X_w$ and a path $p_e$ between $\ast_w$ and $\iota_{w,e}(\ast_e)$ so that the map $\iota_{w,e}$ induces via $p_e$ an embedding $G_e \hookrightarrow G_w$. Let 
\[X_G=\left(\coprod_{w \in V\overline{\mathbb{X}}} X_w \amalg \coprod_{e \in \vec{E}\overline{\mathbb{X}}} X_e \times [0,1] \right)/\sim\] where, for every $e \in \vec{E}\overline{\mathbb{X}}$ with origin $w_1$ and terminal vertex $w_2$, $X_e \times [0,1]$ is identified with $X_{e^{-1}} \times [0,1]$ equipped with the reverse orientation and the subspaces $X_e \times \{0\}$ and $X_e \times \{1\}$ are glued to $X_{w_1}$ and $X_{w_2}$ via $\iota_{w_1,e}$ and $\iota_{w_2,e}$. Then $\pi_1(X_G)$ is isomorphic to $G$ and its universal cover is contractible. 

Let $U$ be the open subset of $X_G$ which is the image of $$X_v \; \amalg \coprod_{e \in \vec{E}\overline{\mathbb{X}}, o(e)=v, t(e) \neq v} X_e \times \left[0,\frac{3}{4} \right) \; \amalg \coprod_{e \in \vec{E}\overline{\mathbb{X}}, o(e)=t(e)=v} X_e \times \left[0,\frac{1}{3} \right).$$ Let $V$ be the open subset of $X_G$ which is the image of 
\begin{align*}
    \coprod_{w \in V\overline{\mathbb{X}}, w \neq v} X_w \; \amalg & \coprod_{e \in \vec{E}\overline{\mathbb{X}}, o(e),t(e) \neq v} X_e \times [0,1] \amalg \coprod_{e \in \vec{E}\overline{\mathbb{X}}, t(e)=v, o(e) \neq v} X_e \times \left[0,\frac{3}{4}\right) \amalg \\ 
 {} & \coprod_{e \in \vec{E}\overline{\mathbb{X}}, o(e)=t(e)=v} X_e \times \left(\frac{1}{4},\frac{3}{4} \right).
\end{align*}
Then $U \cup V$ covers $X_G$. Thus, we can apply the Mayer Vietoris exact sequence to the triple $(X_G,U,V)$, that is, the following sequence is exact

\begin{equation}\label{eq:MVproof}
H_1(U \cap V,A) \to H_1(U,A) \oplus H_1(V,A) \to H_1(G,A).
\end{equation}

The subset $U$ deformation retracts to $X_v$, so that $H_1(U,A)=H_1(G_v,A)$ and the image of $H_1(U,A)$ in $H_1(G,A)$ is $H$. Moreover, the subset $U \cap V$ deformation retracts to $\coprod_{e \in \vec{E}\overline{\mathbb{X}}, o(e)=v} X_e$, so that $H_1(U \cap V,A)$ is isomorphic to $\bigoplus_{e \in \vec{E}\overline{\mathbb{X}}, o(e)=v} H_1(G_e,A)$ and the image of $H_1(U \cap V,A)$ in $H_1(G_v,A)$ is exactly $K_{\mathrm{Inc}_v}$. Thus, Equation~\eqref{eq:MayerVietoris} follows from Equation~\eqref{eq:MVproof} and the above descriptions of $H_1(U,A)$ and $H_1(U \cap V,A)$.

Let $\phi \in \mathcal{K}(S)$. Then $\phi$ induces a homotopy equivalence $\phi_{X_G}$ of $X_G$. Since $\phi \in \mathcal{K}(S)$, $\phi$ acts trivially on the graph $\overline{\mathbb{X}}$. Thus, $\phi_{X_G}$ preserves (up to homotopy) the subsets $U$ and $V$. This shows that the exact sequence in Equation~\eqref{eq:MVproof} is $\phi_{X_G}$-equivariant and the exact sequence in Equation~\eqref{eq:MayerVietoris} is $\phi$-equivariant.
\end{proof}

\begin{lem}\label{lem:projectionhomology}
Let $\mathcal{F}$ be a $(G,\mathcal{G})$-free factor system and let $[A]\in \mathcal{F}$. 
\begin{itemize}
    \item The group $H_1(A,\ZZ/3\ZZ)$ embeds as a direct summand of $H_1(G,\ZZ/3\ZZ)$.
    \item For every $[B]\in \mathcal{F}$ distinct from $[A]$, the inclusion $A \ast B \subseteq G$ induces an embedding $H_1(A,\ZZ/3\ZZ) \oplus H_1(B,\ZZ/3\ZZ)\hookrightarrow H_1(G,\ZZ/3\ZZ)$.
    \item If, for all $g \in G$ and $[G_i] \in \mathcal{G}$, the group $gG_ig^{-1}$ is not contained in $A$, then $H_1(A,\ZZ/3\ZZ)$ is a nontrivial direct summand of $H_1(G,\ZZ/3\ZZ)$. 
\end{itemize}
\end{lem}

\begin{proof}
Let $[B ] \in \mathcal{F}-\{[A]\}$, let $C \subseteq G$ be such that $G=A \ast C$ and let $D \subseteq G$ be such that $G=A \ast B \ast D$. The fact that $H_1(A,\ZZ/3\ZZ) \oplus H_1(B,\ZZ/3\ZZ)$ is a direct summand of $H_1(G,\ZZ/3\ZZ)$ follows from the fact that $A \ast B$ is a free factor of $G$ and Lemma~\ref{lem:MayerVietoris} applied with $\mathbb{X}$ induced by $G=A\ast B \ast D$. This proves the first two assertions of the lemma.

Suppose now that for all $g \in G$ and $[G_i] \in \mathcal{G}$, the group $gG_ig^{-1}$ is not contained in $A$. We now prove that $H_1(A,\ZZ/3\ZZ)$ is nontrivial. By assumption on $A$, for every $[G_i] \in \mathcal{G}$, some conjugate of $G_i$ is contained in $C$. Since $G=G_1 \ast \ldots \ast G_k \ast F_N$, we may assume, up to changing the representative $F_N$, that $A \subseteq F_N$ is a nontrivial free factor of $F_N$. Thus, we have $H_1(A,\ZZ/3\ZZ) \neq \{0\}$. 
\end{proof}

\begin{lem}\label{lem:periodic freefactorsystem}
Let $\phi \in \IA^0(G,\mathcal{G},3)$. Let $\mathcal{F}$ be a $\phi$-invariant $(G,\mathcal{G})$-free factor system. Then every element of $\mathcal{F}$ is fixed by $\phi$.
\end{lem} 

\begin{proof}
Let $[A] \in \mathcal{F}$. Suppose first that there exists $i \in \{1,\ldots,k\}$ such that $G_i \subseteq A$. Then for every $[B]\in \mathcal{F}$ distinct from $[A]$, we have $G_i \cap B=\{1\}$. Thus, since $\phi \in \Out^0(G,\mathcal{G})$, we have $\phi([A])=[A]$.

Suppose now that, for all $i \in \{1,\ldots,k\}$ and all $g \in G$, we have $A \cap gG_ig^{-1}=\{1\}$. By Lemma~\ref{lem:projectionhomology}, the group $H_1(A,\ZZ/3\ZZ)$ is a nontrivial direct summand of $H_1(G,\ZZ/3\ZZ)$. Moreover, for every $[B]\in \mathcal{F}$ distinct from $[A]$, the image of $H_1(B,\ZZ/3\ZZ)$ in $H_1(G,\ZZ/3\ZZ)$ is distinct from $H_1(A,\ZZ/3\ZZ)$. As $\phi$ acts trivially on $H_1(G,\ZZ/3\ZZ)$, we see that $\phi([A])=[A]$.
\end{proof} 

\begin{lem}\label{lem:IA_0_passes_free_factor}
Let $\phi \in \IA^0(G,\mathcal{G},3)$ and let $A$ be a $(G,\mathcal{G})$-free factor whose conjugacy class is $\phi$-invariant. The image $\phi_{|A}$ of $\phi$ in $\Out(A)$ is contained in $\IA^0(A,\mathcal{G}_{|A},3)$.
\end{lem}

\begin{proof}
Since $A$ is a $(G,\mathcal{G})$-free factor, for every $[H]_{A} \in \mathcal{G}_{|A}$, there exists a unique $[G_i] \in \mathcal{G}$ such that $[H]=[G_i]$. In particular, since $\phi \in \Out^0(G,\mathcal{G})$, the image of $\phi$ is contained in $\Out^0(A,\mathcal{G}_{|A})$. 

By Lemma~\ref{lem:projectionhomology}, since $A$ is a $(G,\mathcal{G})$-free factor, the image of $A$ in $H_1(G,\ZZ/3\ZZ)$ is a direct factor $F$ preserved by $\phi$. As $\phi$ acts trivially on $H_1(G,\ZZ/3\ZZ)$, we see that $\phi$ also acts trivially on $F$, so that $\phi_{|A} \in \IA^0(A,\mathcal{G}_{|A},3)$.
\end{proof}

The next lemma gives a useful criterion to show that the automorphism of a finite graph is trivial. 

\begin{lem}[{\cite[Lemma~1.1]{Ivanov92}}]\label{lem:graphautotrivial}
Let $X$ be a finite connected graph and let $f \colon X \to X$ be a graph automorphism. Suppose that $f$ fixes every leaf of $X$ and that the action of $f$ on $H_1(X,\ZZ/3\ZZ)$ is trivial. Then one of the following holds.
\begin{itemize}
    \item The graph $X$ is homeomorphic to a circle and $f$ acts on $X$ as a rotation.
    \item The map $f$ is the identity.
\end{itemize}
\end{lem}

\begin{lem}\label{lem:actionhomograph}
Let $G$ be a group, let $S$ be a $G$-splitting and let $\phi \in \Out(G)$ be an outer automorphism preserving $S$. If $\phi$ acts trivially on $H_1(G,\ZZ/3\ZZ)$, then $\phi$ acts trivially on $H_1(\overline{G \backslash S},\ZZ/3\ZZ)$.
\end{lem}

\begin{proof}
Let $\phi_{gr}$ be the graph automorphism of $\overline{G\backslash S}$ induced by $\phi$. Denote by $\phi_\ast$ (resp. $\phi_{gr\ast}$) the action on $H_1(G,\ZZ/3\ZZ)$ (resp. $H_1(\overline{G \backslash S},\ZZ/3\ZZ)$) induced by $\phi$ (resp. $\phi_{gr}$). Since $G$ is the fundamental group of the graph of groups $G \backslash S$, there exists a natural surjective homomorphism $G \to \pi_1(\overline{G \backslash S})$ inducing in turn a surjection $p \colon H_1(G,\ZZ/3\ZZ) \to H_1(\overline{G \backslash S},\ZZ/3\ZZ)$. The following diagram is commutative:
\[
\begin{tikzcd}
H_1(G,\ZZ/3\ZZ) \arrow{r}{\phi_\ast} \arrow[swap]{d}{p} & H_1(G,\ZZ/3\ZZ) \arrow{d}{p} \\
H_1(\overline{G \backslash S},\ZZ/3\ZZ) \arrow{r}{\phi_{gr \ast}} & H_1(\overline{G \backslash S},\ZZ/3\ZZ).
\end{tikzcd}
\]
Since $p$ is surjective and $\phi_\ast$ is trivial, the map $\phi_{gr \ast}$ is also trivial.
\end{proof}

\begin{lem}\label{lem:trivialactionfreesplitting}
Let $\phi \in \IA^0(G,\mathcal{G},3)$. Let $S$ be a $\phi$-invariant $(G,\mathcal{G})$-free splitting. 
\begin{enumerate}
\item The graph automorphism of $\overline{G \backslash S}$ induced by $\phi$ is trivial. 
\item For every $v \in VS$, the image of the natural homomorphism $\langle \phi \rangle \to \Out(G_v)$ is contained in $\IA^0(G_v,\mathcal{G}_v,3)$.
\end{enumerate}
\end{lem}

\begin{proof}
We may assume that $S$ is nondegenerate as otherwise there is nothing to prove.

We prove the first assertion. By Lemma~\ref{lem:actionhomograph}, the action of $\phi$ on $H_1(\overline{G \backslash S},\ZZ/3\ZZ)$ is trivial. Thus, in order to apply Lemma~\ref{lem:graphautotrivial}, it suffices to prove that $\phi$ fixes every leaf of $\overline{G \backslash S}$. Note that $\phi$ preserves the $(G,\mathcal{G})$-free factor system induced by the set of vertex stabilisers of $S$. In particular, by Lemma~\ref{lem:periodic freefactorsystem}, the outer automorphism $\phi$ fixes pointwise the set of vertices of $\overline{G \backslash S}$ with nontrivial vertex stabiliser. Since every leaf of $\overline{G \backslash S}$ has nontrivial stabiliser by minimality of $S$, we see that $\phi$ fixes every leaf of $\overline{G \backslash S}$. Thus, by Lemma~\ref{lem:graphautotrivial}, if $\overline{G \backslash S}$ is not a circle, then $\phi$ acts trivially on $\overline{G \backslash S}$.

It remains to study the case where $\overline{G \backslash S}$ is a circle. By Lemma~\ref{lem:graphautotrivial}, $\phi$ acts as a rotation on $\overline{G \backslash S}$. Thus, in order to show that $\phi$ acts trivially on $\overline{G \backslash S}$, it suffices to prove that $\phi$ fixes a point of $\overline{G \backslash S}$. If $\xi(G,\mathcal{G}) \geq 2$, then $\overline{G \backslash S}$ has a vertex with nontrivial stabiliser, and is therefore fixed by $\phi$. Otherwise, as $S$ is nondegenerate, we have $G=\ZZ$ and $\phi=\mathrm{id}$, so that $\phi$ acts as the identity on $\overline{G \backslash S}$. This concludes the proof of the first assertion. 

The second assertion follows from Lemma~\ref{lem:IA_0_passes_free_factor} since every nontrivial vertex stabiliser of a $(G,\mathcal{G})$-free splitting is a $(G,\mathcal{G})$-free factor.
\end{proof} 

\subsection{The Standing Assumptions and the group $\IA(G,\mathcal{G},3)$}\label{sec:Standing_Assumptions}

In this section we construct a further finite index subgroup $\IA(G,\mathcal{G},3)$ of $\IA^0(G,\mathcal{G},3)$ which will satisfy the desired aperiodicity properties. The construction of $\IA(G,\mathcal{G},3)$ relies on further assumptions on $G$ which we call the \emph{Standing Assumptions}. These assumptions will be used in the rest of the paper. We first recall some standard definitions.

Let $\phi \in \Out(G,\mathcal{G})$. A \emph{periodic subgroup of $\phi$} is a subgroup $H \subseteq G$ such that there exist $\ell>0$ and $\Psi \in \phi^\ell$ with $\Psi(H)=H$ and $\Psi_{|H}=\mathrm{id}_H$. If $\Phi \in \phi$, let $$\mathrm{Per}(\Phi)=\{g \in G \;|\; \exists \ell >0 \text{ such that } \Phi^\ell(g)=g\}$$ be the \emph{periodic subgroup} of $\Phi$.

\bigskip

\noindent{\bf \hypertarget{SA}{Standing Assumptions}. } 
{\it Let $G$ be a finitely presented torsion free group, let $G=G_1 \ast \ldots \ast G_k \ast F_N$ be the Grushko decomposition of $G$ and let $\mathcal{G}=\{[G_1],\ldots,[G_k]\}$. Let $$p \colon \IA^0(G,\mathcal{G},3) \to \prod_{i=1}^k \Out(G_i)$$ be the natural restriction homomorphism. 

Suppose that, for every $i \in \{1,\ldots,k\}$, there exists a subgroup $\Out^*(G_i)$ of $\Out(G_i)$ with the following properties.
 
\begin{enumerate}
\item The group $\Out^*(G_i)$ is a finite index subgroup of $\Out(G_i)$.
\item The group $\Out^*(G_i)$ is torsion free.
\item For every $\phi \in \Out^*(G_i)$ and every $\phi$-periodic finitely generated subgroup $H$ of $\phi$, we have $\phi \in \Out(G_i,[H]^{(t)})$.
\item Denote by $\Aut^*(G_i)$ the full preimage of $\Out^*(G_i)$ in $\Aut(G_i)$. For every $\Phi \in \Aut^*(G_i)$, every $\Phi$-periodic element of $G_i$ is fixed by $\Phi$, that is $\mathrm{Per}(\Phi)=\mathrm{Fix}(\Phi)$.
\end{enumerate}

Let $\IA(G,\mathcal{G},3)$ be the preimage by $p$ of $\prod_{i=1}^k \Out^*(G_i)$. It is a finite index subgroup of $\IA^0(G,\mathcal{G},3)$.}

\begin{rmq}
When $G=F_N$ and $\mathcal{G}=\varnothing$, the group $\IA(G,\mathcal{G},3)$ is exactly the kernel of the natural homomorphism $\Out(F_N) \to \mathrm{GL}_N(\ZZ/3\ZZ)$ given by the action on $H_1(F_N,\ZZ/3\ZZ)$.
\end{rmq}

We observe immediate consequences of the \hyperlink{SA}{Standing Assumptions}.

\begin{lem}\label{lem:IA_passes_free_factor}
Let $G$ be a group satisfying the \hyperlink{SA}{Standing Assumptions}. Let $\phi \in \IA(G,\mathcal{G},3)$ and let $A$ be a $(G,\mathcal{G})$-free factor whose conjugacy class is $\phi$-invariant. The image $\phi_{|A}$ of $\phi$ in $\Out(A)$ is contained in $\IA(A,\mathcal{G}_{|A},3)$.
\end{lem}

\begin{proof}
By Lemma~\ref{lem:IA_0_passes_free_factor}, we have $\phi_{|A} \in \IA^0(A,\mathcal{G}_{|A},3)$. For every $[H]_A \in \mathcal{G}_{|A}$, there exists a unique $[G_i] \in \mathcal{G}$ such that $[H]=[G_i]$. Thus, the restriction homomorphism $p_{H} \colon \IA^0(G,\mathcal{G},3) \to \Out(H)$ factors through $\IA^0(G,\mathcal{G},3) \to \IA^0(A,\mathcal{G}_{|A},3)$. As $\phi$ maps into $\Out^*(H)$, so does $\phi_{|A}$ and we have $\phi_{|A} \in \IA(A,\mathcal{G}_{|A},3)$.
\end{proof}

\begin{lem}\label{lem:G/ZGtorsionfree}
Let $G$ be a group satisfying the \hyperlink{SA}{Standing Assumptions}.

\begin{enumerate}
\item For every $i \in \{1,\ldots,k\}$, the group $G_i/Z(G_i)$ is torsion free.
\item Let $S$ be a Grushko $(G,\mathcal{G})$-free splitting. The group $\mathcal{T}(S)$ is torsion free.
\end{enumerate}
\end{lem}

\begin{proof}
Let $i \in \{1,\ldots,k\}$. The group $\Inn(G_i) \subseteq \IA^\Aut(G_i)$ is isomorphic to $G_i/Z(G_i)$. By Item~$(4)$ of the \hyperlink{SA}{Standing Assumptions}, the group $\IA^\Aut(G_i)$ is torsion free. Therefore, the group $G_i/Z(G_i)$ is torsion free. This proves the first assertion.

We now prove the second assertion. Let $S$ be a Grushko $(G,\mathcal{G})$-free splitting. By Proposition~\ref{lem:grouptwists}, the group $\mathcal{T}(S)$ is isomorphic to a direct product $\prod_{v \in V\overline{G \backslash S}} G_v^{n_{v}}/Z(G_v)$, where $Z(G_v)$ embeds diagonnally. The group $\prod_{v \in V\overline{G \backslash S}} G_v^{n_{v}}/Z(G_v)$ projects onto the group $\prod_{v \in V\overline{G \backslash S}} \left(G_v/Z(G_v)\right)^{n_{v}}$ with kernel isomorphic to $\prod_{v \in V\overline{G \backslash S}} Z(G_v)^{n_{v}}/Z(G_v)$. The group $\prod_{v \in V\overline{G \backslash S}} \left(G_v/Z(G_v)\right)^{n_{v}}$ is torsion free by Assertion~$(1)$ and the kernel is also torsion free, being the quotient of a direct product of free abelian groups by its diagonal subgroup (recall that $G$ is torsion free). Hence $\mathcal{T}(S)$ is torsion free.
\end{proof}

We now check an easy aperiodicity property of $\IA(G,\mathcal{G},3)$.

\begin{prop}\label{prop:IAtorsionfree}
Let $G$ be a group satisfying the \hyperlink{SA}{Standing Assumptions}. The group $\IA(G,\mathcal{G},3)$ is torsion free.
\end{prop}

\begin{proof}
The natural homomorphism $p_{\IA}\colon \IA(G,\mathcal{G},3) \to \prod_{i=1}^k \Out(G_i)$ has torsion free image by Item~$(2)$ of the \hyperlink{SA}{Standing Assumptions}. Thus, it suffices to show that $\ker(p_{\IA})$ is torsion free. 

Let $H$ be a finite subgroup of $\ker(p_{\IA})$. By Theorem~\ref{thm:fixedpointouterspace}, the group $H$ fixes a Grushko $(G,\mathcal{G})$-free splitting $S$. By Lemma~\ref{lem:trivialactionfreesplitting}, the group $H$ is contained in $\mathcal{K}(S)$. Since $H \subseteq \ker(p_{\IA})$, the group $H$ is contained in $\mathcal{T}(S)$. By Lemma~\ref{lem:G/ZGtorsionfree}, the group $\mathcal{T}(S)$ is torsion free. Hence $H$ is trivial.
\end{proof}

\section{Periodic conjugacy classes of elements}\label{Section:perodicelements}

Let $G$ be a group satisfying the \hyperlink{SA}{Standing Assumptions}. In this section we describe the dynamics of periodic conjugacy classes of elements of $G$ for automorphisms in $\IA(G,\mathcal{G},3)$. The aim of the section is to prove the following theorem.

\begin{theo}\label{thm:invariantfamilycyclicfixed}
Let $G$ be a group satisfying the \hyperlink{SA}{Standing Assumptions}. Let $\phi \in \IA(G,\mathcal{G},3)$. Let $\mathcal{H}$ be a $\phi$-invariant finite set of conjugacy classes of finitely generated periodic subgroups of $\phi$. Then $\phi \in \Out(G,\mathcal{H}^{(t)})$.
\end{theo}

As a corollary, we have:

\begin{coro}\label{coro:periodicconjclass}
Let $G$ be a group satisfying the \hyperlink{SA}{Standing Assumptions}. Let $\phi \in \IA(G,\mathcal{G},3)$. Every $\phi$-periodic conjugacy class of elements of $G$ is fixed by $\phi$.
\end{coro}

\begin{proof}
Let $g \in G$ whose conjugacy class is $\phi$-periodic, let $\Phi$ be a representative of $\phi$ and let $\mathcal{H}=\{[\langle \Phi^i(g) \rangle]\}_{i \in \NN}$. Then $\mathcal{H}$ is a $\phi$-invariant finite set of conjugacy classes of finitely generated periodic subgroups of $\phi$. By Theorem~\ref{thm:invariantfamilycyclicfixed}, we have $\phi \in \Out(G,\mathcal{H}^{(t)})$, so that there exists $\Psi \in \phi$ such that $\Psi(\langle g \rangle)=\langle g \rangle$ and $\Psi_{|\langle g \rangle}=\mathrm{Id}_{\langle g \rangle}$. Then $\Psi(g)=g$ and the $G$-conjugacy class of $g$ is fixed by $\phi$.
\end{proof}

The proof uses elementary JSJ decompositions of the relatively hyperbolic pair $(G,\mathcal{G})$ relative to the set $\mathcal{H}$. Indeed, using an inductive argument, we may assume that $G$ is one-ended relative to $\mathcal{G} \cup \mathcal{H}$. In that case, by Theorem~\ref{thm:JSJrelhyp} (with $G=\Gamma$ and $\mathcal{P}=\mathcal{G}$), there exists a canonical $\phi$-invariant JSJ elementary splitting $T_{\mathcal{G} \cup \mathcal{H}}$. We then show that the action of $\phi$ on this JSJ splitting is sufficiently simple to prove the desired result. The proof being quite technical, we subdivide it into several subsections.

\subsection{Trivial action on the quotient graph}

Let $G$ be a group satisfying the \hyperlink{SA}{Standing Assumptions} and let $\mathcal{H}$ be a finite set of conjugacy classes of finitely generated subgroups of $G$ such that $G$ is one-ended relative to $\mathcal{G} \cup \mathcal{H}$. The aim of this section is to prove that any outer automorphism $\phi \in \IA(G,\mathcal{G},3)$ preserving $\mathcal{H}$ and such that $\mathcal{H}$ consists of periodic subgroups of $\phi$ acts trivially on the graph $\overline{G \backslash T_{\mathcal{G}\cup\mathcal{H}}}$ (see Proposition~\ref{lem:trivialactionJSJsplitting}). 

The strategy is to apply Lemma~\ref{lem:graphautotrivial}. However, in order to apply it, we need to prove that $\phi$ fixes every leaf of $\overline{G \backslash T_{\mathcal{G}\cup\mathcal{H}}}$. This is done in Lemma~\ref{lem:trivialactionJSJsplittingtreecase}. During the proof of the later, we will need the following result, due to Horbez.

\begin{lem}[{\cite[Lemma~6.11]{Horbez17}}]\label{lem:cyclicsplittinghasvertexfreefactor}
Let $(G,\mathcal{G})$ be a free product and let $T$ be a $(G,\mathcal{G})$-cyclic splitting whose edge stabilisers are all nontrivial. There exists $v \in VT$ such that $G_v$ splits as a free product relative to $\mathcal{G}_v$ and incident edge stabilisers. 
\end{lem}

\begin{lem}\label{lem:trivialactionJSJsplittingtreecase}
Let $G$ be a group satisfying the \hyperlink{SA}{Standing Assumptions}.  Let $\mathcal{H}$ be a finite set of conjugacy classes of finitely generated subgroups of $G$ such that $G$ is one-ended relative to $\mathcal{G}\cup\mathcal{H}$.
Let $\phi \in \IA(G,\mathcal{G},3)$ be an outer automorphism preserving $\mathcal{H}$. The graph automorphism induced by $\phi$ on $\overline{G \backslash T_{\mathcal{G}\cup\mathcal{H}}}$ fixes every leaf of $\overline{G \backslash T_{\mathcal{G}\cup\mathcal{H}}}$.
\end{lem}

\begin{proof}
We assume that $T_{\mathcal{G}\cup\mathcal{H}}$ is not reduced to a point as otherwise there is nothing to prove. We first prove that $\phi$ fixes every leaf of $\overline{G \backslash T_{\mathcal{G}\cup\mathcal{H}}}$ with noncyclic stabiliser. Let $\overline{v}$ be a leaf of $\overline{G \backslash T_{\mathcal{G}\cup\mathcal{H}}}$ with noncyclic stabiliser and let $v$ be a lift of $\overline{v}$ in $T_{\mathcal{G}\cup\mathcal{H}}$. Note that, since $G_v$ is noncyclic, by Theorem~\ref{thm:JSJrelhyp}$(5)$, the stabiliser of $v$ either is nonelementary or its conjugacy class belongs to $\mathcal{G}$. 

Suppose first that $[G_v] \in \mathcal{G}$. Since the action of $G$ on $T_{\mathcal{G}\cup\mathcal{H}}$ is minimal, and since $\overline{v}$ is a leaf, we see that $v$ is the unique point fixed by $G_v$. Since $\phi \in \Out^0(G,\mathcal{G})$, we deduce that $v$ is fixed by $\phi$. 

Suppose that $G_v$ is nonelementary. Assume first that there exists $i \in \{1,\ldots,k\}$ such that $v$ is the unique fixed point of a conjugate of $G_i$. Since $\phi \in \Out^0(G,\mathcal{G})$, we see that $v$ is fixed by $\phi$. 

Assume otherwise that $v$ is not the only fixed point of a conjugate of some $G_i$ with $i \in \{1,\ldots,k\}$. We claim that $|\mathcal{G}_v| \leq 1$. Indeed, let $\overline{e} \in E\overline{G \backslash T_{\mathcal{G}\cup\mathcal{H}}}$ be the edge adjacent to $\overline{v}$ and let $e \in ET_{\mathcal{G}\cup\mathcal{H}}$ be a lift of $\overline{e}$ adjacent to $v$. Since $\overline{v}$ is a leaf of $\overline{G \backslash T_{\mathcal{G}\cup\mathcal{H}}}$, the group $G_v$ acts transitively on the set of edges adjacent to $v$ and there exists a unique $G_v$-conjugacy class of edges groups in $G_v$. Let $[A]\in \mathcal{G}$ be such that $A \cap G_v \neq \{1\}$. Since $T_{\mathcal{G}\cup\mathcal{H}}$ is a $(G,\mathcal{G}\cup \mathcal{H})$-elementary tree, the group $A$ is elliptic in $T_{\mathcal{G}\cup\mathcal{H}}$. By assumption, $v$ is not the only fixed point of $A$, so that, up to taking a $G_v$-conjugate of $A$, we have $A \cap G_v \subseteq G_{e}$. Thus, for every $[B]_{G_v}\in \mathcal{G}_v$, there exists $B' \in [B]_{G_v}$ such that $B' \cap G_v \subseteq G_e$. Since $G_e$ is elementary, we necessarily have $|\mathcal{G}_v| \leq 1$. The claim follows.

The claim implies that the free product decomposition of $G_v$ induced by $\mathcal{G}_v$ is:
$$G_v=H \ast F_\ell,$$ where $H$ is a (possibly trivial) subgroup of some conjugate of some $G_i$ and $F_\ell$ is a nonabelian free group of rank $\ell \geq 0$. Since the action of $G$ on $T_{\mathcal{G}\cup\mathcal{H}}$ is minimal and since, by assumption, $v$ is not the only fixed point of $H$, we have $\ell \geq 1$.

\medskip

\noindent{\bf Claim. } The group $F_\ell$ induces a nontrivial subspace of $H_1(G,\ZZ/3\ZZ)$ which is distinct from any subspace of $H_1(G,\ZZ/3\ZZ)$ induced by any other leaf group of $G\backslash T_{\mathcal{G}\cup\mathcal{H}}$.

\medskip

\begin{proof}

If $H$ is nontrivial, then the stabiliser of the edge adjacent to $v$ is contained in $H$ as $v$ is not the fixed point of a maximal parabolic subgroup. Thus, $F_\ell$ is a $(G,\mathcal{G})$-free factor and the image of $F_\ell$ in $H_1(G,\ZZ/3\ZZ)$ is not contained in any subspace of $H_1(G,\ZZ/3\ZZ)$ induced by any other leaf group of $G\backslash T_{\mathcal{G}\cup\mathcal{H}}$ by Lemma~\ref{lem:projectionhomology}. 

Suppose that $H$ is trivial. Since edge stabilisers of $T_{\mathcal{G}\cup\mathcal{H}}$ are elementary, the stabiliser of the edge adjacent to $v$ is infinite cyclic. Since $G_v$ is nonelementary, we have $\ell \geq 2$. Let $U$ be the tree obtained from $T_{\mathcal{G}\cup\mathcal{H}}$ by collapsing every orbit of edges except the one of the edge $e$ adjacent to $v$. Then $U$ is a $(G,\mathcal{G})$-cyclic splitting. Since $\overline{v}$ is a leaf, the splitting $U$ induces an amalgam $G \simeq F_\ell \ast_{G_e} G_w$, for some $w \in VU$. Let $L'$ be the set of leaves of $\overline{G\backslash T_{\mathcal{G}\cup\mathcal{H}}}$ distinct from $v$. Then the image $V'$ of $\langle H_1(G_{v'},\ZZ/3\ZZ) \rangle_{v' \in L'}$ in $H_1(G,\ZZ/3\ZZ)$ is contained in the image of $H_1(G_w,\ZZ/3\ZZ)$. Let $V$ be a supplementary subspace of $H_1(G_e,\ZZ/3\ZZ)$ in $H_1(F_\ell,\ZZ/3\ZZ)$, which exists since $\ell \geq 2$ and $G_e \simeq \ZZ$. Then the image of $V$ in $H_1(G,\ZZ/3\ZZ)$ is not contained in $V'$ by Lemma~\ref{lem:MayerVietoris}. Therefore, the image of $F_\ell$ in $H_1(G,\ZZ/3\ZZ)$ is not contained in any subspace of $H_1(G,\ZZ/3\ZZ)$ induced by any other leaf of $\overline{G\backslash T_{\mathcal{G}\cup\mathcal{H}}}$.
\end{proof}

Since $\phi \in \IA(G,\mathcal{G},3)$, the claim implies that $v$ is fixed by $\phi$. This shows that $\phi$ fixes every leaf of $\overline{G \backslash T_{\mathcal{G}\cup\mathcal{H}}}$ with noncyclic stabiliser.

It remains to prove that $\phi$ fixes every leaf of $\overline{G \backslash T_{\mathcal{G}\cup\mathcal{H}}}$ with cyclic stabiliser. Let $v$ be a leaf of $\overline{G \backslash T_{\mathcal{G}\cup\mathcal{H}}}$ with cyclic stabiliser. By Theorem~\ref{thm:JSJrelhyp}$(3)$, every elementary subgroup of $G$ fixes at most one vertex of $T_{\mathcal{G} \cup \mathcal{H}}$ with elementary stabiliser. Thus, in order to prove that $\phi$ fixes $v$, its suffices to prove that $\phi$ fixes $[G_v]$.

Suppose that $G_v$ is $(G,\mathcal{G})$-peripheral. Since $\phi \in \Out^0(G,\mathcal{G})$, there exist $[G_i] \in \mathcal{G}$ and $\Phi \in \phi$ such that $G_v \subseteq G_i$, $\Phi(G_i)=G_i$, and $[G_v]_{G_i}$ is a $\Phi_{|G_i}$-periodic conjugacy class of a cyclic subgroup of $G_i$. By Item~$(3)$ of the \hyperlink{SA}{Standing Assumptions}, we see that $[G_v]$ is fixed by $\Phi_{|G_i}$. Thus, the vertex $v$ is fixed by $\phi$. 

Suppose now that $G_v$ is $(G,\mathcal{G})$-nonperipheral. Let $\mathcal{L}$ be the set of leaves $v'$ of $\overline{G \backslash T_{\mathcal{G}\cup\mathcal{H}}}$ with cyclic $(G,\mathcal{G})$-nonperipheral stabiliser. Let $U$ be the $(G,\mathcal{G})$-cyclic splitting obtained by collapsing every edge of $\overline{G \backslash T_{\mathcal{G}\cup\mathcal{H}}}$ except the ones adjacent to leaves of $\mathcal{L}$. Note that every edge of $U$ has nontrivial infinite cyclic edge stabiliser by Theorem~\ref{thm:JSJrelhyp}$(1)$. 

Let $x$ be the central vertex of $\overline{G \backslash U}$. Then $x$ is the unique vertex of $\overline{G \backslash U}$ with noncyclic stabiliser. In particular, $x$ is the unique vertex of $\overline{G \backslash U}$ such that incident edge stabilisers can induce a nontrivial free product decomposition. Thus, by Lemma~\ref{lem:cyclicsplittinghasvertexfreefactor}, the group $G_x$ splits as a free product relative to the $(G_x,\mathcal{G}_x)$-free factor system $\mathcal{F}_x$ obtained from $\mathcal{G}_x$ by adding one conjugacy class of cyclic subgroup for every $v' \in \mathcal{L}$.

Since every leaf of $\overline{G \backslash U}$ has cyclic $(G,\mathcal{G})$-nonperipheral stabiliser, for every $[H] \in \mathcal{G}$, the group $H$ fixes (a $G$-translate of) $x$. Thus, for every $[H] \in \mathcal{G}$, there exists a unique $[H_x]_{G_x} \in \mathcal{G}_x$ such that $[H]=[H_x]$. This shows that the set $\mathcal{F}=\mathcal{G}\cup \{[G_{v'}]\}_{v' \in \mathcal{L}}$ is a $(G,\mathcal{G})$-free factor system preserved by $\phi$. By Lemma~\ref{lem:periodic freefactorsystem}, the outer automorphism $\phi$ preserves the conjugacy class of every $G_{v'}$ with $v' \in \mathcal{L}$. In particular, it preserves $[G_v]$ and it therefore fixes $v$.
\end{proof}

The next ingredient we need in the proof of the fact that $\phi$ acts trivially on the quotient graph $\overline{G \backslash T_{\mathcal{G} \cup \mathcal{H}}}$ is the following lemma, which studies QH vertices of $T_{\mathcal{G} \cup \mathcal{H}}$ and more generally QH vertices of $G$-splittings (see Section~\ref{sec:terminologyJSJ} for the terminology). We in particular describe, for $v$ a QH vertex of a $G$-splitting $S$, the image of the homomorphism $H_1(G_v,\ZZ/3\ZZ) \to H_1(G,\ZZ/3\ZZ)$ induced by $G_v \hookrightarrow G$. Let $\Sigma_v$ be the compact hyperbolic surface associated with $v$. Let $\mathrm{Bd}_v$ be the set of conjugacy classes of boundary subgroups of $\Sigma_v$. Let $\mathrm{Bd}_S$ be the set of conjugacy classes of boundary subgroups $C$ of $\Sigma_v$ such that there exists $e \in ES$ such that $C$ and $G_e$ are commensurable. Let $\chi(\Sigma_v)$ be the Euler characteristic of $\Sigma_v$ and let $\chi_S(\Sigma_v)=\chi(\Sigma_v)+|\mathrm{Bd}_S|$. Note that, since $\Sigma_v$ is hyperbolic, we have $\chi(\Sigma_v)<0$.

\begin{lem}\label{lem:HomologyQHvertex}
Let $G$ be a torsion free finitely generated group, let $S$ be a $G$-splitting and let $v \in VS$ be a QH vertex. 

\begin{enumerate}
\item Suppose that $\chi_S(\Sigma_v) <1$. For every $w \in VS$ not in the $G$-orbit of $v$, the image of $H_1(G_v,\ZZ/3\ZZ)$ in $H_1(G,\ZZ/3\ZZ)$ is not contained in the image of $H_1(G_w,\ZZ/3\ZZ)$ in $H_1(G,\ZZ/3\ZZ)$.
\item Let $C$ be a boundary subgroup of $\Sigma_v$. One of the following hold.
\begin{enumerate}
\item There exists $e \in ES$ such that $G_e$ and $C$ are commensurable.
\item The image of $H_1(C,\ZZ/3\ZZ)$ in $H_1(G,\ZZ/3\ZZ)$ is nontrivial. Moreover, for every boundary subgroup $C'$ of $\Sigma_v$ which is not $G$-conjugate to $C$, the images of $H_1(C,\ZZ/3\ZZ)$ and $H_1(C',\ZZ/3\ZZ)$ are distinct.
\item The surface $\Sigma_v$ is orientable and $C$ is the unique (up to $G$-conjugacy) boundary subgroup of $\Sigma_v$ which is not contained in $\mathrm{Bd}_S$.
\item The surface $\Sigma_v$ is orientable, $C$ is not contained in $\mathrm{Bd}_S$ and, up to $G$-conjugacy, there exists exactly one boundary subgroup $D$ which is not $G$-conjugate to $C$ such that $D$ is not contained in $\mathrm{Bd}_S$. Moreover, if we fix an orientation of $\Sigma_v$ and we let $c$ and $d$ be the generators of $C$ and $D$ with positive orientation, then $c$ and $d$ have distinct nontrivial images in $H_1(G,\ZZ/3\ZZ)$.
\end{enumerate}
\end{enumerate}

\end{lem}

\begin{proof}
We first fix some notations that we will use in the proof.  We always consider homology groups with coefficients in $\ZZ/3\ZZ$. In this proof only, $\langle. \rangle$ denotes the $\ZZ/3\ZZ$-module freely generated by the subset.  If $\Sigma_v$ is orientable, we consider the following presentation of its first homology group: 
\begin{equation}\label{eq:homologyorientable}
H_1(\Sigma_v)=H_1(\Sigma_v,\ZZ/3\ZZ)=\langle a_1,b_1,\ldots,a_t,b_t,c_1,\ldots,c_s \rangle/\langle \sum_{j=1}^s c_j \rangle, \; t \geq 0 \text{ and }s \geq 0.
\end{equation}
Moreover, the image of every boundary subgroup is equal to some $\langle c_i \rangle$ with $i \in \{1,\ldots,s\}$ and we identify the boundary subgroup with the corresponding $\langle c_i \rangle$. If $\Sigma_v$ is nonorientable, we consider 
\begin{align}\label{eq:homologynonorientable}
H_1(\Sigma_v)&= \left\langle a_1,\ldots,a_t,c_1,\ldots,c_s \right\rangle/\langle 2\sum_{i=1}^t a_i +\sum_{j=1}^s c_j \rangle, \nonumber \\
{} & =  \left\langle a_1,\ldots,a_t,c_1,\ldots,c_s \right\rangle/\langle \sum_{j=1}^s c_j - \sum_{i=1}^t a_i\rangle, \; t \geq 1 \text{ and } s \geq 0. 
\end{align}
As above, the image of every boundary subgroup is equal to some $\langle c_i \rangle$ with $i \in \{1,\ldots,s\}$ and we identify the boundary subgroup with the corresponding $\langle c_i \rangle$. In either case, denote by $p_{\Sigma_v} \colon \langle a_1,b_1,\ldots,a_t,b_t,c_1,\ldots,c_s \rangle \to H_1(\Sigma_v)$ the natural projection.

\medskip

We recall some general facts about the image of $H_1(G_v)$ in $H_1(G)$. Let $K_S=p_{\Sigma_v}(\langle c_i \rangle_{c_i \in \mathrm{Bd}_S})$. By definition of a QH vertex, for every edge $e \in ES$ which is adjacent to $v$, the group $G_e$ is contained in a boundary subgroup of $\Sigma_v$. By Lemma~\ref{lem:MayerVietoris}, we see that the image of $H_1(G_v)=H_1(\Sigma_v)$ in $H_1(G)$ is a quotient $H_1(G_v)/K$, where $K$ is a subgroup of $K_S$. Moreover, for every $w \in VS$ not in the $G$-orbit of $v$, the intersection of the images of $H_1(G_v)$ and $H_1(G_w)$ in $H_1(G)$ is contained in $K_S/K$. 

\medskip

We now prove the first assertion. Suppose that $\chi_S(\Sigma_v) <1$. According to the above discussion, it suffices to prove that $K_S \neq H_1(\Sigma_v)$. We distinguish between two cases, according to the orientability of $\Sigma_v$.

Suppose first that $\Sigma_v$ is nonorientable. Recall the notations from Equation~\eqref{eq:homologynonorientable}. If $t \geq 2$, then $p_{\Sigma_v}(\langle a_1 \rangle)$ is not contained in $K_S$ by Equation~\eqref{eq:homologynonorientable} and the conclusion follows. Suppose that $t=1$. Then $\chi(\Sigma_v)=1-s$. Since $\chi_S(\Sigma_v) <1$, there exists $j \in \{1,\ldots,s\}$ such that $p_{\Sigma_v}(c_j) \notin K_S$. Thus, we have $K_S\neq H_1(\Sigma_v)$. This proves the first case.

Suppose now that $\Sigma_v$ is orientable. Recall the notations from Equation~\eqref{eq:homologyorientable}. If $\Sigma_v$ is not a punctured sphere, then $t \geq 1$ and $K_S\neq H_1(\Sigma_v)$. Suppose now that $\Sigma_v$ is a punctured sphere. Then $\chi(\Sigma_v)=2-s$. Since $\chi_S(\Sigma_v)<1$, we have $|\mathrm{Bd}_v|=s\geq |\mathrm{Bd}_S|+2$. Therefore, by Equation~\eqref{eq:homologyorientable}, we see that $K_S \subsetneq p_{\Sigma_v}(\langle c_i \rangle_{i \in \{1,\ldots,s\}})$. This proves the first assertion of the lemma.

\medskip

We now prove the second assertion. Let $C$ be a boundary subgroup of $\Sigma_v$. Suppose that there does not exist an edge $e \in ES$ such that $C$ and $G_e$ are commensurable. Thus, up to reordering, we may suppose that the image of $C$ in $H_1(\Sigma_v)$ is equal to $p_{\Sigma_v}(\langle c_1 \rangle)$ and that the group $H$ generated by the image in $H_1(\Sigma_v)$ of every boundary subgroup distinct from $C$ is contained in $p_{\Sigma_v}(\langle c_2,\ldots,c_s \rangle)$.  We prove that $C$ satisfies Assertion~$(2)(b)$,$(c)$ or $(d)$. Again, we distinguish between two cases, according to the orientability of $\Sigma_v$. 

Suppose first that $\Sigma_v$ is nonorientable. We prove that $C$ satisfies Assertion~$(2)(b)$. As explained in the paragraph above the proof of the first assertion, since $K_S \subseteq H$, it suffices to prove that the subspace $p_{\Sigma_v}(\langle c_1 \rangle)$ is not contained in $H$. Since $t \geq 1$ by Equation~\eqref{eq:homologynonorientable}, the subspace $p_{\Sigma_v}(\langle c_1 \rangle)$ is not contained in $p_{\Sigma_v}(\langle c_2,\ldots,c_s \rangle)$. In particular, we see that the image of $C$ in $H_1(\Sigma_v)$ (which is contained in $p_{\Sigma_v}(\langle c_1 \rangle)$) is not contained in $H \subseteq p_{\Sigma_v}(\langle c_2,\ldots,c_s \rangle)$. This proves the nonorientable case.

Suppose now that $\Sigma_v$ is orientable. If $|\mathrm{Bd}_v|= |\mathrm{Bd}_S|+1$, then $C$ satisfies Assertion~$(2)(c)$. So we may suppose that $|\mathrm{Bd}_v|\geq |\mathrm{Bd}_S|+2$. We prove that $C$ satisfies Assertion~$(2)(b)$ or Assertion~$(2)(d)$. If $|\mathrm{Bd}_v|\geq |\mathrm{Bd}_S|+3$, then, by Equation~\eqref{eq:homologyorientable}, the group $p_{\Sigma_v}(\langle c_1 \rangle)$ is not contained in $H$, so that Assertion~$(2)(b)$ holds. 

Finally, suppose that $|\mathrm{Bd}_v|=|\mathrm{Bd}_S|+2$ and fix an orientation of $\Sigma_v$. Let $C,D \in \mathrm{Bd}_v-\mathrm{Bd}_S$ be nonconjugate boundary subgroups and let $c \in C$ and $d \in D$ be generators of $C$ and $D$ inducing the orientation of $C$ and $D$. We want to show that $c$ and $d$ have distinct images in $H_1(G)$. Up to reordering, we may suppose that $c$ is sent to the image of $c_1$ in $H_1(G)$ and $d$ is sent to the image of $c_2$. 

Recall that the only relation in $H_1(\Sigma_v)$ is $\sum_{i=1}^s p_{\Sigma_v}(c_i)=0$ and that the image of $H_1(\Sigma_v)$ in $H_1(G)$ is isomorphic to a quotient of $H_1(\Sigma_v)$ by $K \subseteq K_S \subseteq p_{\Sigma_v}(\langle c_3,\ldots,c_s \rangle)$. Thus, if the image of $\sum_{i=3}^s p_{\Sigma_v}(c_i)$ in $H_1(G)$ is nontrivial, then $c$ and $d$ generate distinct subspaces of $H_1(G)$. If the image of $\sum_{i=3}^s p_{\Sigma_v}(c_i)$ in $H_1(G)$ is trivial, then $c$ and $d$ are sent in $H_1(G)$ to distinct nontrivial elements: the image of $c$ is equal to the image of $d^{-1}$.
\end{proof}

\begin{prop}\label{lem:trivialactionJSJsplitting}
Let $G$ be a group satisfying the \hyperlink{SA}{Standing Assumptions}. Let $\mathcal{H}$ be a finite set of conjugacy classes of finitely generated subgroups of $G$ such that $G$ is one-ended relative to $\mathcal{G}\cup\mathcal{H}$. Let $\phi \in \IA(G,\mathcal{G},3)$ be an outer automorphism preserving $\mathcal{H}$ and such that $\mathcal{H}$ consists of periodic subgroups of $\phi$. The graph automorphism induced by $\phi$ on $\overline{G \backslash T_{\mathcal{G}\cup \mathcal{H}}}$ is trivial.
\end{prop}

\begin{proof}
Again we assume that $T_{\mathcal{G}\cup \mathcal{H}}$ is not reduced to a point. In particular, by Theorem~\ref{thm:JSJrelhyp}$(1)$, there exists a nonelementary vertex stabiliser. 

The proof is similar to the one of Lemma~\ref{lem:trivialactionfreesplitting}, but we also have to deal with nontrivial edge stabilisers. By Lemma~\ref{lem:trivialactionJSJsplittingtreecase}, the automorphism $\phi$ fixes every leaf of $\overline{G \backslash T_{\mathcal{G}\cup \mathcal{H}}}$. By Lemma~\ref{lem:actionhomograph}, $\phi$ acts trivially on $H_1(\overline{G \backslash T_{\mathcal{G}\cup \mathcal{H}}},\ZZ/3\ZZ)$. Thus, by Lemma~\ref{lem:graphautotrivial}, either $\phi$ acts trivially on $\overline{G \backslash T_{\mathcal{G}\cup \mathcal{H}}}$, or else $\overline{G \backslash T_{\mathcal{G}\cup \mathcal{H}}}$ is a circle and $\phi$ acts on it by rotation. The conclusion in the first case being immediate, we may assume that we are in the second case. It therefore suffices to prove that $\phi$ fixes a point in $\overline{G \backslash T_{\mathcal{G}\cup \mathcal{H}}}$. 

Let $[H] \in \mathcal{G}$. By Lemma~\ref{lem:elemornonelem}, the group $H$ fixes a unique nonelementary vertex, or a unique elementary vertex. In both cases, one can canonically associate to $[H]$ a unique vertex of $\overline{G \backslash T_{\mathcal{G}\cup \mathcal{H}}}$. Since $\phi \in \Out^0(G,\mathcal{G})$, if $\mathcal{G} \neq \varnothing$, we see that $\phi$ fixes a point in $\overline{G \backslash T_{\mathcal{G}\cup \mathcal{H}}}$.

Thus, we are reduced to the case where $\mathcal{G}=\varnothing$. This implies that $G=F_N$ is a nonabelian free group. Moreover, every elementary subgroup is infinite cyclic. For every vertex $v$ of $\overline{G \backslash T_{\mathcal{G}\cup \mathcal{H}}}$, let $\mathrm{Inc}_v$ be the set of $G_v$-conjugacy classes of edge groups adjacent to $v$. Let $\mathcal{F}_v$ be the minimal free factor system of $G_v$ such that, for every $[K_v] \in  \mathrm{Inc}_v$, the group $K_v$ is $(G_v,\mathcal{F}_v)$-peripheral. 

By Lemma~\ref{lem:cyclicsplittinghasvertexfreefactor}, there exists a vertex $v$ of $\overline{G \backslash T_{\mathcal{G}\cup \mathcal{H}}}$ such that $\mathcal{F}_v \neq \{[G_v]_{G_v}\}$. Since $\mathcal{F}_v \neq \{[G_v]_{G_v}\}$, and since elementary subgroups are cyclic, the group $G_v$ is nonelementary. Let $n_v \geq 1$ be the minimal integer such that $\phi^{n_v}(v)=v$. Note that $\mathcal{F}_v$ is preserved by $\phi^{n_v}$. Moreover, since $\phi$ acts as a rotation on $\overline{G \backslash T_{\mathcal{G}\cup \mathcal{H}}}$, the graph automorphism of $\overline{G \backslash T_{\mathcal{G}\cup \mathcal{H}}}$ induced by $\phi^{n_v}$ is trivial. 

\medskip

\noindent{\bf Case 1. } Suppose that the image of $\phi^{n_v}$ in $\Out(G_v,\mathcal{F}_v)$ is finite. 

\medskip

Notice that, by Theorem~\ref{thm:JSJrelhyp}$(8)$, this includes in particular the case where $v$ is rigid (recall that $\mathcal{H}$ consists of periodic subgroups of $\phi$). 

By Theorem~\ref{thm:fixedpointouterspace} applied to the free product $(G_v,\mathcal{F}_v)$ and to $\langle \phi_v^{n_v} \rangle$, the group $\langle \phi_v^{n_v} \rangle$ fixes a Grushko $(G_v,\mathcal{F}_v)$-free splitting $U_v$ in the outer space of $G_v$ relative to $\mathcal{F}_v$. 

For every $w \in V\overline{G \backslash T_{\mathcal{G}\cup \mathcal{H}}}$ in the $\phi$-orbit of $v$, let $U_w$ be the Grushko $(G_w,\mathcal{F}_w)$-free splitting whose underlying tree is the same as $U_v$, and the action of $g \in G_w$ is given by the action of $\Phi_w^{-n_w}(g) \in G_v$, where $n_w$ is the minimal positive integer such that $\phi^{n_w}(v)=w$ and $\Phi_w \in \phi^{n_w}$ is such that $\Phi_w(G_v)=G_w$. We claim that the Grushko $(G_w,\mathcal{F}_w)$-free splitting $U_w$ only depends on $U_v$ and $\phi$. Indeed, if $\Phi_w,\Phi_w' \in \phi^{n_w}$ are such that $\Phi_w(G_v)=G_w=\Phi_w'(G_v)$, then $\Phi_w' \circ \Phi_w^{-1}$ is an inner automorphism $\mathrm{ad}_h$ of $F_N$ preserving $G_w$. Since $G_w$ is nonelementary, it fixes a unique vertex in $T_{\mathcal{G}\cup \mathcal{H}}$, so that it is its own normaliser in $F_N$. In particular, $h \in G_w$ and the construction of $U_w$ does not depend on the choice of $\Phi_w$. Therefore, the splitting $U_w$ is preserved by $\phi_w^{n_v}$.

Let $U$ be the splitting of $F_N$ obtained from $T_{\mathcal{G}\cup \mathcal{H}}$ by blowing up, for every $m \geq 1$, the splitting $U_{\phi^m(v)}$ at $\phi^m(v)$ and attaching the adjacent edges of $\phi^m(v)$ to the unique points fixed by their edge groups. Note that $U$ is fixed by $\phi$ since the construction of the splitting $U_w$ with $w \in \langle \phi \rangle v$ is $\phi$-equivariant. 

Let $S'$ be the free splitting of $F_N$ obtained from $U$ by collapsing all edges with nontrivial stabiliser. Then $S'$ is a $\phi$-invariant free splitting of $F_N$. 

We claim that, if $\phi$ acts nontrivially on $\overline{F_N \backslash T_{\mathcal{G}\cup \mathcal{H}}}$, then it also acts nontrivially on $\overline{F_N \backslash S'}$. Indeed, any nontrivial action of $\phi$ on $\overline{F_N \backslash T_{\mathcal{G}\cup \mathcal{H}}}$ will permute the subgraphs of $S'$ corresponding to points in the outer space of $(G_w,\mathcal{F}_w)$ with $w \in V\overline{F_N \backslash T_{\mathcal{G}\cup \mathcal{H}}}$. 

By Lemma~\ref{lem:trivialactionfreesplitting}, the automorphism $\phi$ acts trivially on $\overline{F_N \backslash S'}$. Thus, $\phi$ acts trivially on $\overline{F_N \backslash T_{\mathcal{G}\cup \mathcal{H}}}$. This concludes the proof of Case~1.

\medskip

\noindent{\bf Case~2. } Suppose that $v$ is flexible. 

\medskip

By Theorem~\ref{thm:JSJrelhyp}$(5)$, the vertex $v$ is a QH vertex. Let $\Sigma_v$ be the associated compact surface. Recall the notations $\mathrm{Bd}_v$, $\mathrm{Bd}_S$ and $\chi_S(\Sigma_v)$ from just above Lemma~\ref{lem:HomologyQHvertex}. Since $\overline{F_N \backslash T_{\mathcal{G}\cup \mathcal{H}}}$ is a circle, we have $|\mathrm{Bd}_S| \leq 2$. In particular, since $\Sigma_v$ is hyperbolic, we have $\chi(\Sigma_v) \leq - 1$ and $\chi_S(\Sigma_v) \leq 1$.

Suppose first that $\chi_S(\Sigma_v)=1$. Then $\chi(\Sigma_v) = - 1$ and $|\mathrm{Bd}_S| = 2$, so that $\Sigma_v$ has at least two boundary components. This implies that $\Sigma_v$ is isomorphic to a pair of pants or a twice-punctured projective plane. In both cases, these surfaces have a finite mapping class group (see~\cite[Corollary~4.6]{korkmaz2002} for the twice-punctured projective plane). Since $\phi^{n_v}$ maps into the mapping class group of $\Sigma_v$ by Theorem~\ref{thm:JSJrelhyp}$(9)$, the case where $\chi_S(\Sigma_v)=1$ follows from Case~1.

Finally, suppose that $\chi_S(\Sigma_v)<1$. By Lemma~\ref{lem:HomologyQHvertex}$(1)$, for every $w \in VT_{\mathcal{G}\cup \mathcal{H}}$ not in the $G$-orbit of $v$, the image of $H_1(G_v,\ZZ/3\ZZ)$ in $H_1(G,\ZZ/3\ZZ)$ is distinct from the image of $H_1(G_w,\ZZ/3\ZZ)$ in $H_1(G,\ZZ/3\ZZ)$. In particular, since $\phi \in \IA(G,\mathcal{G},3)$, the vertex $v$ is fixed by $\phi$. Hence $\phi$ acts trivially on $\overline{F_N \backslash T_{\mathcal{G}\cup \mathcal{H}}}$. 
\end{proof}

\subsection{Understanding the outer action on the vertex groups}

By Proposition~\ref{lem:trivialactionJSJsplitting}, for every vertex $v \in VT_{\mathcal{G}\cup \mathcal{H}}$ (resp. every edge $e \in ET_{\mathcal{G}\cup \mathcal{H}}$), the outer automorphism $\phi$ induces an outer automorphism $\phi_v \in \Out(G_v)$ (resp. ($\phi_e \in \Out(G_e)$). We now investigate the outer action of $\phi_v$ on $G_v$ according to the nature of $v$.

\begin{lem}\label{lem:flexiblevertexacttrivially}
Let $G$ be a torsion free finitely generated group and let $S$ be a $G$-splitting. Let $v \in VS$ be a QH vertex and let $\Sigma_v$ be the associated compact surface. Let $\phi \in \Out(G) \cap \mathcal{K}(S)$ be an outer automorphism acting trivially on $H_1(G,\ZZ/3\ZZ)$.

Suppose that, for every edge $e \in ES$ adjacent to $v$ with $G_e \simeq \ZZ$, the outer automorphism $\phi_e$ is trivial. Then $\phi$ preserves the $G$-conjugacy class of a generator of every boundary subgroup of $\pi_1(\Sigma_v)$.
\end{lem}

\begin{proof}
Suppose first that $S$ is reduced to the vertex $v$. Then $G$ is a surface group and the lemma follows from the work of Ivanov~\cite[Theorem~1.2]{Ivanov92}.

So we may suppose that $S$ is not reduced to a vertex. Let $C$ be a boundary subgroup of $\pi_1(\Sigma_v)$ and let $c$ be a generator of $C$. Note that, since $G$ is torsion free, the group $C$ is infinite cyclic. By definition of a QH vertex, for every edge $e \in ES$ adjacent to $v$, the group $G_e$ is contained in a boundary subgroup of $\Sigma_v$. We now examine the four possible cases given by Lemma~\ref{lem:HomologyQHvertex}$(2)$.

Suppose that Lemma~\ref{lem:HomologyQHvertex}$(2)(a)$ occurs, so that there exists an edge $e$ of $S$ such that $G_e$ is a finite index subgroup of $C$. Let $\Phi \in \phi$ be an automorphism preserving $G_e$ and $G_v$. By assumption, we have $\Phi_{|G_e}=\id_{G_e}$. Since $G_v\simeq \pi_1(\Sigma_v)$ is a finitely generated free group, $G_e$ is contained in a unique maximal cyclic subgroup $C'$, so that $\Phi$ preserves $C'$ and $C \subseteq C'$. As $\Phi$ fixes $G_e$ elementwise, it also fixes $C'$ and $C$ elementwise. Thus, we see that $\phi$ preserves the conjugacy class of $c$.

Suppose that Lemma~\ref{lem:HomologyQHvertex}$(2)(b)$ occurs: the image of $H_1(C,\ZZ/3\ZZ)$ in $H_1(G,\ZZ/3\ZZ)$ is nontrivial and, for every boundary subgroup $C'$ of $\Sigma_v$ which is not conjugate to $C$, the images of $H_1(C,\ZZ/3\ZZ)$ and $H_1(C',\ZZ/3\ZZ)$ in $H_1(G,\ZZ/3\ZZ)$ are distinct. Since $\phi$ permutes the conjugacy classes of the generators of the boundary subgroups of $\Sigma_v$ and since $\phi$ acts trivially on $H_1(G,\ZZ/3\ZZ)$, we see that the conjugacy class of $c$ is fixed by $\phi$.

\medskip

\noindent{\bf Claim. } Suppose that Lemma~\ref{lem:HomologyQHvertex}$(2)(c)$ or~$(d)$ occur so that $\Sigma_v$ is orientable. Then $\phi_v$ preserves the orientation of $\Sigma_v$.

\begin{proof}
As $S$ is a connected $G$-splitting, there exists a boundary component $C'$ of $\pi_1(\Sigma_v)$ which is commensurable to an edge stabiliser $G_e$ of $e \in ES$. Since $\phi_e$ acts trivially on $G_e$, any automorphism in the outer class of $\phi$ preserving $e$ also acts trivially on $C'$. Thus, $\phi_v$ preserves the orientation at any point contained in the boundary of $\Sigma_v$ associated with $C'$. This shows that $\phi_v$ induces an orientation-preserving mapping class of $\Sigma_v$.
\end{proof}

Suppose that Lemma~\ref{lem:HomologyQHvertex}$(2)(c)$ occurs, so that $\Sigma_v$ is orientable and $C$ is the only boundary subgroup which is not commensurable to an edge stabiliser. We fix an orientation of $\Sigma_v$ so that $c$ is a positive generator of $C$. By the Claim, $\phi_v$ preserves the orientation of $\Sigma_v$. Thus $\phi_v$ sends $[c]$ to the conjugacy class of a generator of $C$ inducing the orientation, that is $\phi([c])=[c]$.

Finally, suppose that Lemma~\ref{lem:HomologyQHvertex}$(2)(d)$ occurs, so that $\Sigma_v$ is orientable and that there exists a unique (up to conjugacy) boundary subgroup $D$ of $\Sigma_v$ which is not conjugate to $C$ and which is not commensurable to an edge stabiliser. Again, fixing an orientation of $\Sigma_v$, we may suppose that $c$ is a positive generator of $C$ and that $d \in D$ is a positive generator of $D$. By the Claim, $\phi_v$ preserves the orientation of $\Sigma_v$, so that $\phi_v$ sends $[c]$ to either itself or $[d]$. By Lemma~\ref{lem:HomologyQHvertex}$(2)(d)$ the images of $c$ and $d$ in $H_1(G,\ZZ/3\ZZ)$ are distinct. Since $\phi$ acts trivially on $H_1(G,\ZZ/3\ZZ)$, we see that $\phi$ preserves the conjugacy class of $c$.
\end{proof}

\begin{lem}\label{lem:rigidvertexacttrivially}
Let $G$ be a group satisfying the \hyperlink{SA}{Standing Assumptions}. Let $\mathcal{H}$ be a finite set of conjugacy classes of finitely generated subgroups of $G$ such that $G$ is one-ended relative to $\mathcal{G}\cup\mathcal{H}$. Let $\phi \in \IA(G,\mathcal{G},3)$ be an outer automorphism preserving $\mathcal{H}$ and such that $\mathcal{H}$ consists of periodic subgroups of $\phi$. 

Let $v \in VT_{\mathcal{G}\cup \mathcal{H}}$ be a rigid vertex. Suppose that, for every edge $e \in ET_{\mathcal{G}\cup \mathcal{H}}$ adjacent to $v$ with $G_e \simeq \ZZ$, the outer automorphism $\phi_e$ is trivial. Then $\phi_v=\mathrm{id}$.
\end{lem}

\begin{proof}
By Theorem~\ref{thm:JSJrelhyp}$(4)$, the group $G_v$ is finitely generated. Thus, the pair $(G_v,\mathcal{G}_v)$ defines a free product. By Theorem~\ref{thm:JSJrelhyp}$(8)$, the outer automorphism $\phi_v$ of $G_v$ induced by $\phi$ is of finite order. We will prove that $\phi_v$ is in fact trivial. By Theorem~\ref{thm:fixedpointouterspace}, there exists a Grushko $(G_v,\mathcal{G}_v)$-free splitting $S_v$ which is preserved by $\phi_v$. 

\medskip

\noindent{\bf Claim. } The outer automorphism $\phi_v$ acts trivially on the quotient graph $\overline{G_v \backslash S_v}$.

\begin{proof}
We will prove the Claim using Lemma~\ref{lem:graphautotrivial}. Notice that we do not know that $\phi_v$ acts trivially on $H_1(G_v,\ZZ/3\ZZ)$, so that we cannot apply Lemma~\ref{lem:actionhomograph}. In order to overcome this difficulty, we need a more detailed analysis of the action of $\phi_v$ on $\overline{G_v \backslash S_v}$. 

\medskip

\noindent{\it Step~1. } $\phi_v$ fixes every vertex of $\overline{G_v \backslash S_v}$ with nontrivial stabiliser. In particular, $\phi_v$ fixes every leaf of $\overline{G_v \backslash S_v}$.

\begin{proofblack}
Since edge stabilisers in $S_v$ are trivial, and since the action of $G_v$ on $S_v$ is minimal, it suffices to prove that $\phi_v$ preserves the $G_v$-conjugacy class $[H]_{G_v}$ of every nontrivial vertex stabiliser $H$ of $S_v$. Let $[H]_{G_v} \in \mathcal{G}_v$ be the $G_v$-conjugacy class of a vertex stabiliser $H$ of $S_v$. Since $T_{\mathcal{G} \cup \mathcal{H}}$ is a $(G,\mathcal{G})$-elementary splitting, $H$ has a fixed point in $T_{\mathcal{G} \cup \mathcal{H}}$. Moreover, either $[H]_G \in \mathcal{G}$ and $v$ is the unique fixed point of $H$ or else $H$ is the intersection of $G_v$ with an elementary vertex stabiliser, that is, $H$ is a stabiliser of an edge adjacent to $v$. Let us examine both cases.

Suppose that $H$ is the stabiliser of an edge $e$ adjacent to $v$. By Proposition~\ref{lem:trivialactionJSJsplitting}, the outer automorphism $\phi$ acts trivially on the quotient graph $\overline{G \backslash T_{\mathcal{G} \cup \mathcal{H}}}$. In particular, $\phi$ has a representative $\Phi$ fixing $e$. Hence $\Phi(G_v)=G_v$ and $\Phi(H)=H$, that is, $\phi_v$ preserves the $G_v$-conjugacy class of $H$.

Suppose now that $v$ is the unique fixed point of $H$ and that $[H]_G \in \mathcal{G}$. Since $\phi \in \Out^0(G,\mathcal{G})$, the outer class $\phi$ preserves the $G$-conjugacy class of $H$. By~\cite[Lemma~2.2]{GuirardelLevitt2015} (applied with $K=G_v$), as $H$ does not fix an edge in $T_{\mathcal{G} \cup \mathcal{H}}$, the outer automorphism $\phi_v$ preserves the $G_v$-conjugacy class of $H$. The first assertion of the claim follows. Since the action of $G_v$ on $S_v$ is minimal, every leaf of $\overline{G_v \backslash S_v}$ has nontrivial stabiliser, so that $\phi_v$ fixes every leaf of $\overline{G_v \backslash S_v}$ by the first assertion.
\end{proofblack}

Let $\mathrm{Inc}_v$ be the set of $G_v$-conjugacy classes of edge stabilisers of $T_{\mathcal{G}\cup \mathcal{H}}$ incident to $v$. Let $\mathrm{Inc}_v^{cyc}$ be the subset of $\mathrm{Inc}_v$ consisting of $G_v$-conjugacy classes of edge stabilisers of $T_{\mathcal{G}\cup \mathcal{H}}$ incident to $v$ with $(G,\mathcal{G})$-nonperipheral cyclic stabiliser. 

\medskip

\noindent{\it Step~2. } Let $[G_e]_{G_v} \in \mathrm{Inc}_v^{cyc}$ and let $c \in G_e$ be a generator. Then $c$ acts loxodromically on $S_v$. Moreover, there exists $\Phi_v \in \phi_v$ such that $\Phi_v(c)=c$. Thus, $\Phi_v$ preserves the axis $\mathrm{Ax}(c)$ of $c$ in $S_v$ as well as its orientation. 

\begin{proofblack}
Since $c$ is $(G,\mathcal{G})$-nonperipheral, it does not fix a point in the Grushko $(G_v,\mathcal{G}_v)$-free splitting $S_v$. Thus, $c$ acts loxodromically on $S_v$. Since $\phi$ acts trivially on every edge group $G_e$ with $G_e \simeq \ZZ$ by assumption, there exists $\Phi_v \in \phi_v$ such that $\Phi_v(c)=c$ (take any representative fixing the edge $e$). Thus, $\Phi_v$ preserves the axis $\mathrm{Ax}(c)$ of $c$ in $S_v$ as well as its orientation. 
\end{proofblack}

If $[G_e]_{G_v} \in \mathrm{Inc}_v^{cyc}$, and if $c_e$ is a generator of $G_e$, we denote by $\gamma_{c_e}$ the projection of $\mathrm{Ax}(c_e)$ in $\overline{G_v \backslash S_v}$. Let $K_{\mathrm{Inc}_v^{cyc}}^{S_v}$ be the subgroup of $H_1(\overline{G_v\backslash S_v},\ZZ)$ generated by the classes $[\gamma_{c_e}]$ with $[\langle c_e \rangle]_{G_v} \in  \mathrm{Inc}_v^{cyc}$. 
Let $K_{\mathrm{Inc}_v}$ (resp. $K_{\mathrm{Inc}_v^{cyc}}$) be the subgroup of $H_1(G_v,\ZZ)$ generated by the image of $\langle G_e \rangle_{[G_e]_{G_v} \in \mathrm{Inc}_v}$ (resp. of $\langle G_e \rangle_{[G_e]_{G_v} \in \mathrm{Inc}_v^{cyc}}$).  

\medskip

\noindent{\it Step~3. } The outer automorphism $\phi_v$ acts trivially on $H_1(\overline{G_v \backslash S_v},\ZZ)$. 

\begin{proofblack}
By Step~2, $\phi_v$ fixes elementwise the group $K_{\mathrm{Inc}_v^{cyc}}^{S_v}$. Let $T$ be the torsion part of the quotient $H_1(\overline{G_v \backslash S_v},\ZZ)/K_{\mathrm{Inc}_v^{cyc}}^{S_v}$ and let $V$ be the preimage of $T$ in $H_1(\overline{G_v \backslash S_v},\ZZ)$. Note that $V$ is a direct factor of $H_1(\overline{G_v \backslash S_v},\ZZ)$. Then $V$ is preserved by $\phi_v$ and $K_{\mathrm{Inc}_v^{cyc}}^{S_v}$ has finite index in $V$. Since $\phi_v$ fixes elementwise $K_{\mathrm{Inc}_v^{cyc}}^{S_v}$ and since $V$ is a free abelian group, it follows that $\phi_v$ fixes elementwise $V$.

We now claim that $\phi_v$ fixes elementwise the free abelian group $H_1(\overline{G_v \backslash S_v},\ZZ)/V$. Since $\phi_v$ has finite order, it induces a finite order automorphism of $H_1(\overline{G_v \backslash S_v},\ZZ)/V$. Since \[\ker\left(\mathrm{GL}(H_1(\overline{G_v \backslash S_v},\ZZ)/V) \to \mathrm{GL}(H_1(\overline{G_v \backslash S_v},\ZZ)/V \otimes \ZZ/3\ZZ) \right)\] is torsion free,  
it suffices to prove that $\phi_v$ acts trivially on $H_1(\overline{G_v \backslash S_v},\ZZ)/V \otimes \ZZ/3\ZZ$.

Let $H$ be the image of the homomorphism $H_1(G_v,\ZZ/3\ZZ) \to H_1(G,\ZZ/3\ZZ)$. Let $K_{\mathrm{Inc}_v}(\ZZ/3\ZZ)$ be the image of $K_{\mathrm{Inc}_v}$ in $H_1(G_v,\ZZ/3\ZZ)$. By Proposition~\ref{lem:trivialactionJSJsplitting}, we have $\phi \in \mathcal{K}(T_{\mathcal{G}\cup \mathcal{H}})$. Thus, by Lemma~\ref{lem:MayerVietoris}, there exists a subgroup $K$ of $K_{\mathrm{Inc}_v}(\ZZ/3\ZZ)$ such that the group $H$ fits in an exact sequence of $\phi_v$-equivariant maps
\begin{equation}\label{eq:surjectionH}
  K \to H_1(G_v,\ZZ/3\ZZ) \to H. 
\end{equation}
The projection $\widetilde{p}_v \colon H_1(G_v,\ZZ) \to H_1(\overline{G_v \backslash S_v},\ZZ)$ is surjective and $\phi_v$-equivariant. If $[G_e]_{G_v} \in \mathrm{Inc}_v \setminus \mathrm{Inc}_v^{cyc}$, then $G_e$ is a fixed point of $S_v$, so that the images of $G_e$ in $\pi_1(\overline{G_v \backslash S_v})$ and $H_1(\overline{G_v \backslash S_v},\ZZ)$ are trivial. Therefore, we have $$\widetilde{p}_v(K_{\mathrm{Inc}_v})=  \widetilde{p}_v(K_{\mathrm{Inc}_v^{cyc}}) = K_{\mathrm{Inc}_v^{cyc}}^{S_v} \subseteq V.$$ Thus, the projection 
\[p_v \colon H_1(G_v,\ZZ)/K_{\mathrm{Inc}_v} \to H_1(\overline{G_v \backslash S_v},\ZZ)/V\]
is well-defined, surjective and $\phi_v$-equivariant. Tensoring by $\ZZ/3\ZZ$, the homomorphism 
\begin{equation}\label{eq:surjection3}
    H_1(G_v,\ZZ)/K_{\mathrm{Inc}_v} \otimes \ZZ/3\ZZ \to H_1(\overline{G_v \backslash S_v},\ZZ)/V \otimes \ZZ/3\ZZ
\end{equation}
is surjective and $\phi_v$-equivariant. The group $H_1(G_v,\ZZ)/K_{\mathrm{Inc}_v} \otimes \ZZ/3\ZZ$ is isomorphic to $H_1(G_v,\ZZ/3\ZZ)/K_{\mathrm{Inc}_v}(\ZZ/3\ZZ)$ (recall that we defined $K_{\mathrm{Inc}_v}(\ZZ/3\ZZ)$ as the image of $K_{\mathrm{Inc}_v}$ in $H_1(G_v,\ZZ/3\ZZ)$). By Equation~\eqref{eq:surjectionH}, the group $H_1(G_v,\ZZ/3\ZZ)/K_{\mathrm{Inc}_v}(\ZZ/3\ZZ)$ is a quotient of $H$. Therefore, $H$ surjects $\phi_v$-equivariantly onto the group $H_1(G_v,\ZZ)/K_{\mathrm{Inc}_v} \otimes \ZZ/3\ZZ$. Combining Equations~\eqref{eq:surjectionH} and~\eqref{eq:surjection3}, we obtain a $\phi_v$-equivariant surjective homomorphism
\[H \to H_1(\overline{G_v \backslash S_v},\ZZ)/V \otimes \ZZ/3\ZZ.\] Since $H \subseteq H_1(G,\ZZ/3\ZZ)$ and since $\phi_v \in \IA(G,\mathcal{G},3)$, $\phi_v$ acts trivially on $H$ and hence on $H_1(\overline{G_v \backslash S_v},\ZZ)/V \otimes \ZZ/3\ZZ$. This proves that $\phi_v$ fixes pointwise $H_1(\overline{G_v \backslash S_v},\ZZ)/V$.

Let $x \in H_1(\overline{G_v \backslash S_v},\ZZ)$. Since $\phi_v$ fixes pointwise both $V$ and $H_1(\overline{G_v \backslash S_v},\ZZ)/V$, there exists $v_x \in V$ such that $\phi_v(x)=x +v_x$. Since $\phi_v(v_x)=v_x$ and since $\phi_v$ is a finite order automorphism of $H_1(\overline{G_v \backslash S_v},\ZZ)$, we necessarily have $v_x=0$, so that $\phi_v$ fixes pointwise $H_1(\overline{G_v \backslash S_v},\ZZ)$.
\end{proofblack}

By Step~3, $\phi_v$ fixes pointwise $H_1(\overline{G_v \backslash S_v},\ZZ)$ and $H_1(\overline{G_v \backslash S_v},\ZZ/3\ZZ)$. By Step~1, $\phi_v$ fixes every leaf of $\overline{G_v \backslash S_v}$. By Lemma~\ref{lem:graphautotrivial}, either $\phi_v$ fixes pointwise $\overline{G_v \backslash S_v}$ or else $\overline{G_v \backslash S_v}$ is a circle and $\phi_v$ acts as a rotation on it. 

In order to prove that $\phi_v$ fixes pointwise $\overline{G_v \backslash S_v}$, it remains to prove that, if $\overline{G_v \backslash S_v}$ is a circle, then $\phi_v$ fixes a point in it. So suppose that $\overline{G_v \backslash S_v}$ is a circle. Since $G_v$ is nonelementary, in particular noncyclic, the graph of groups $G_v \backslash S_v$ contains a vertex with nontrivial stabiliser. By Step~1, this vertex is fixed by $\phi_v$. The conclusion follows.
\end{proof}

By the claim, $\phi_v$ lies in $\mathcal{K}(S_v)$. We now use the short exact sequence in Proposition~\ref{prop:stabfreesplittingautversion} to prove that $\phi_v$ is trivial. Suppose first that $\mathcal{G}_v=\varnothing$. By Proposition~\ref{prop:stabfreesplittingautversion}, we have $\mathcal{K}(S_v)=\{\mathrm{id}\}$, so the conclusion follows in this case.

Suppose that $\mathcal{G}_v \neq \varnothing$ and let $[H] \in \mathcal{G}_v$ with $H \neq \{1\}$. We have a natural homomorphism $\langle \phi_v \rangle \to \Aut(H)$ given by sending $\phi_v$ to the unique representative $\Phi_H$ of $\phi_v$ which fixes a chosen edge adjacent to the vertex fixed by $H$. 

There exists a unique $i \in \{1,\ldots,k\}$ such that $H \subseteq G_i$. Uniqueness of $G_i$ implies that the automorphism $\Phi_H$ extends to a representative $\Phi$ of $\phi$ which preserves $G_i$. Since $\phi_{|G_i} \in \Out^*(G_i)$, we have $\Phi_{|G_i} \in \Aut^*(G_i)$. In particular, by Item~$(4)$ of the \hyperlink{SA}{Standing Assumptions}, every element of $H$ which is $\Phi$-periodic (and hence $\Phi_H$-periodic) is fixed. As $\Phi_H$ is finite order since $\phi_v$ is, we see that the image of $\phi_v$ in $\Aut(H)$ is trivial. 

By Proposition~\ref{prop:stabfreesplittingautversion}, the outer automorphism $\phi_v$ is contained in a direct product $N$ of groups, each direct factor being isomorphic to a subgroup $H$ of $G_v$ with $[H] \in \mathcal{G}_v$. Since $G$ is torsion free, we see that $N$ is torsion free. In particular, $\phi_v$ is trivial.
\end{proof}

\subsection{Conclusion and proof of Theorem~\ref{thm:invariantfamilycyclicfixed}}

In this section, we prove Theorem~\ref{thm:invariantfamilycyclicfixed} by combining the results of the above sections. We first record the following elementary lemma.

\begin{lem}\label{lem:nonperipheralnotconjinverse}
Let $(G,\mathcal{G})$ be a free product with $G$ torsion free and let $g \in G$ be $(G,\mathcal{G})$-nonperipheral. Then $N_G(\langle g \rangle)=C_G(g)$ is root--closed and infinite cyclic.
\end{lem}

\begin{proof}
Let $h \in G-\{1\}$ and let $m>0$ be such that $h^m \in N_G(\langle g \rangle)$. Let $T$ be a Grushko $(G,\mathcal{G})$-free splitting. Since $g$ is $(G,\mathcal{G})$-nonperipheral, it acts loxodromically on $T$. Let $\mathrm{Ax}(g)$ be the axis of $g$. Then for every $n \in \ZZ^*$, we have $\mathrm{Ax}(g)=\mathrm{Ax}(g^n)$. Since $h^m \in N_G(\langle g \rangle)$, we have \[h^m\mathrm{Ax}(g)=\mathrm{Ax}(h^mgh^{-m})=\mathrm{Ax}(g).\] Thus, the group $\langle g,h^m\rangle$ preserves $\mathrm{Ax}(g)$. 

The group $H=\Stab_G(\mathrm{Ax}(g))$ has an index two subgroup $H'$ preserving the orientation of $\mathrm{Ax}(g)$. Thus, we have a homomorphism $H' \to \ZZ$ given by the translation length in $\mathrm{Ax}(g)$. Since $T$ is a free splitting, the kernel of $H' \to \ZZ$ is trivial. Thus, $H$ is virtually cyclic. Since $G$ is torsion free, the group $H$ is cyclic, generated by an element acting loxodromically on $\mathrm{Ax}(g)$. Since $h^m \in H-\{1\}$, $h$ and $h^m$ acts loxodromically on $T$ and $\mathrm{Ax}(h)=\mathrm{Ax}(h^m)=\mathrm{Ax}(g)$. Thus, $h$ is contained in the cyclic group $H$, so that $h$ commutes with $g$. 
\end{proof}

We are now ready to prove the main theorem of this section.

\begin{proof}[Proof of Theorem~\ref{thm:invariantfamilycyclicfixed}]
The proof is by induction on the Grushko rank $\xi(G,\mathcal{G})$ of $(G,\mathcal{G})$. If $\xi(G,\mathcal{G})=1$, then either $G=G_1$ or else $G=\ZZ$. If $G=G_1$, then the result follows from Item~$(3)$ of the \hyperlink{SA}{Standing Assumptions}. If $G=\ZZ$, then $\IA(G,\mathcal{G},3)=\{\mathrm{Id}\}$ and the result is immediate. 

Suppose now that $\xi(G,\mathcal{G}) \geq 2$. Let $\mathcal{F}$ be the minimal $(G,\mathcal{G})$-free factor system such that for every $[H] \in \mathcal{H}$, the group $H$ is $(G,\mathcal{F})$-peripheral. Then $\phi$ preserves $\mathcal{F}$. By Lemma~\ref{lem:periodic freefactorsystem}, every element of $\mathcal{F}$ is preserved by $\phi$. 

Let $[A] \in \mathcal{F}$ and let $\phi_{|A} \in \Out(A)$ be the outer automrophism of $A$ induced by $\phi$. Let $\mathcal{H}_A$ be the set of $A$-conjugacy classes of subgroups of $A$ induced by $\mathcal{H}$. Suppose that $\mathcal{F} \neq \{[G]\}$. Then, for every $[A]\in \mathcal{F}$, we have $\xi(A,\mathcal{G}_{|A})<\xi(G,\mathcal{G})$. By Lemma~\ref{lem:IA_passes_free_factor} and the inductive hypothesis, for every $[A]\in \mathcal{F}$, we have $\phi_{|A} \in \Out(A,\mathcal{H}_A^{(t)})$. We deduce that $\phi \in \Out(G,\mathcal{H}^{(t)})$.

So we may suppose that $\mathcal{F}=\{[G]\}$, that is, $G$ is one-ended relative to $\mathcal{G} \cup \mathcal{H}$. The group $G$ is hyperbolic relative to $\mathcal{G}$. Let $T_{\mathcal{G}\cup \mathcal{H}}$ be the JSJ $(G,\mathcal{G} \cup \mathcal{H})$-elementary splitting given by Theorem~\ref{thm:JSJrelhyp}. The tree $T_{\mathcal{G}\cup \mathcal{H}}$ is preserved by $\phi$. By Proposition~\ref{lem:trivialactionJSJsplitting}, the graph automorphism induced by $\phi$ on $\overline{G \backslash T_{\mathcal{G}\cup \mathcal{H}}}$ is trivial. Therefore, for every $v \in VT_{\mathcal{G}\cup \mathcal{H}}$, we have a natural homomorphism $\langle \phi \rangle \to \Out(G_v)$. Recall that we denote the image of $\phi$ by $\phi_v$. Similarly, for every edge $e \in ET_{\mathcal{G}\cup \mathcal{H}}$, we have a natural homomorphism $\langle \phi \rangle \to \Out(G_e)$ and we denote the image of $\phi$ by $\phi_e$.

For every $v \in VT_{\mathcal{G}\cup \mathcal{H}}$, let $\mathcal{H}_v$ be the set of $G_v$-conjugacy classes of subgroups of $G_v$ induced by $\mathcal{H}$. Since for every $[H] \in \mathcal{H}$, the group $H$ is elliptic in $T_{\mathcal{G}\cup \mathcal{H}}$, it suffices to prove that, for every $v \in VT_{\mathcal{G}\cup \mathcal{H}}$, we have $\phi_v \in \Out(G_v,\mathcal{H}_v^{(t)})$. We first need the following claim.

\medskip

\noindent{\bf Claim. } Let $e \in ET_{\mathcal{G}\cup \mathcal{H}}$ be such that $G_e \simeq \ZZ$. Then $\phi_e=\mathrm{id}$.

\begin{proof}
Let $\Phi_e \in \phi_e$ be induced by the restriction to $G_e$ of a representative $\Phi$ of $\phi$ preserving $e$ and let $t$ be a generator of $G_e$. Then $\Phi_e(t)=t$ or $\Phi_e(t)=t^{-1}$. Thus, $t$ is a periodic element of $\Phi_e$ and of $\Phi$. Suppose that $G_e$ is contained in a peripheral subgroup $G_i$ with $i \in \{1,\ldots,k\}$. Since $G_i$ is malnormal, $G_i$ is the unique subgroup of $G$ containing $G_e$ such that $[G_i] \in \mathcal{G}$. Thus $\Phi(G_i)=G_i$. By Item~$(4)$ of the \hyperlink{SA}{Standing Assumptions}, we see that $\Phi_e(t)=\Phi_{|G_i}(t)=t$. 

Thus, we may suppose that $G_e$ is $(G,\mathcal{G})$-nonperipheral. In that case, by Lemma~\ref{lem:nonperipheralnotconjinverse}, $t$ is not conjugate to $t^{-1}$ in $G$. Thus, it suffices to prove that $\phi$ preserves the conjugacy class of $t$ in $G$.

By Lemma~\ref{lem:cyclicsplittinghasvertexfreefactor}, the group $G_e$ is contained in a proper $(G,\mathcal{G})$-free factor. Let $A$ be the minimal (necessarily proper) $(G,\mathcal{G})$-free factor containing $G_e$. Then $[A]$ is preserved by $\phi$. Moreover, since $A$ is a $(G,\mathcal{G})$-free factor, the image $\phi_{|A}$ of $\phi$ in $\Out(A)$ is contained in $\IA(A,\mathcal{G}_{|A},3)$ by Lemma~\ref{lem:IA_passes_free_factor}. Since $\Out(G_e)$ is finite, every $A$-conjugacy class of an element in $G_e$ is $\phi_{|A}$-periodic. Since $A$ is proper, we have $\xi(A,\mathcal{G}_{|A}) < \xi(G,\mathcal{G})$. By the induction hypothesis, we see that $\phi_{|A}$ preserves the $A$-conjugacy class of $t$ and that $\Phi_e$ fixes $G_e$ elementwise.
\end{proof}

We now prove that, for every $v \in VT_{\mathcal{G}\cup \mathcal{H}}$, we have $\phi_v \in \Out(G_v, \mathcal{H}_v^{(t)})$. We distinguish between several cases, according to the nature of $v$.

\medskip

\noindent{\bf Case 1. } Suppose that $G_v$ is elementary noncyclic.

\medskip

In that case, we have $[G_v] \in \mathcal{G}$. By Item~$(3)$ of the \hyperlink{SA}{Standing Assumptions}, since $\phi \in \IA(G,\mathcal{G},3)$, we see that $\phi_v \in \Out(G_v,\mathcal{H}_v^{(t)})$.

\medskip

\noindent{\bf Case 2. } Suppose that $G_v$ is elementary cyclic.

\medskip

Since edge stabilisers are infinite, there exists an edge $e$ adjacent to $v$ such that $G_e$ has finite index in $G_v$. By the claim, we see that $\phi_e$ acts trivially on $G_e$. Since $G_e \subseteq G_v$ is infinite cyclic, this implies that $\phi_v$ acts trivially on $G_v$.

\medskip

\noindent{\bf Case 3. } Suppose that $v$ is rigid.

\medskip

By the Claim, we see that $\phi$ acts trivially on the stabiliser of every edge $e$ adjacent to $v$ with $G_e \simeq \ZZ$. By Lemma~\ref{lem:rigidvertexacttrivially}, the outer automorphism $\phi_v$ is trivial.

\medskip

\noindent{\bf Case 4. } Suppose that $G_v$ is a flexible nonelementary vertex.

\medskip

Recall that, by Theorem~\ref{thm:JSJrelhyp}$(5)$, there exists a compact surface $\Sigma_v$ such that $G_v$ is isomorphic to $\pi_1(\Sigma_v)$. Moreover, for every edge $e$ adjacent to $G_v$, the group $G_e$ is contained in a boundary subgroup of $\pi_1(\Sigma_v)$ and, for every $[H] \in \mathcal{H}_v$, the group $H$ is contained in a boundary subgroup of $\pi_1(\Sigma_v)$. 

The outer class $\phi_v$ then induces an outer automorphism of $\pi_1(\Sigma_v)$ preserving the set of conjugacy classes of boundary subgroups of $\pi_1(\Sigma_v)$. Therefore, it suffices to prove that $\phi_v$ fixes the conjugacy class of a generator of every boundary subgroup of $\pi_1(\Sigma_v)$. 
By the Claim, we see that $\phi$ acts trivially on the stabiliser of every edge $e$ adjacent to $v$ with $G_e \simeq \ZZ$. By Lemma~\ref{lem:flexiblevertexacttrivially}, the outer automorphism $\phi_v$ preserves the conjugacy class of a generator of every boundary subgroup. This concludes the proof of the last case and the proof of the theorem.
\end{proof}

We finish this section with a proof of the $\Aut$-version of Theorem~\ref{thm:invariantfamilycyclicfixed}. 
\begin{coro}\label{coco:periodicelementfixed}
Let $G$ be a group satisfying the \hyperlink{SA}{Standing Assumptions}. Consider an automorphism $\Phi \in \IA^\Aut(G,\mathcal{G},3)$. Every $\Phi$-periodic element of $G$ is fixed by $\Phi$.
\end{coro}

\begin{proof}
Let $g$ be a $\Phi$-periodic element of $G$, let $H=\langle \phi^\ell(g) \rangle_{\ell \in \NN}$ and let $\mathcal{H}=\{[H]\}$. Since $g$ is $\Phi$-periodic, the group $H$ is finitely generated and $\Phi$-invariant. Let $A$ be the minimal $(G,\mathcal{G})$-free factor containing $H$. Then $A$ is preserved by $\Phi$. Thus, using an inductive argument, we may suppose that $G=A$. 

Then the group $G$ is one-ended relative to $\mathcal{G} \cup \mathcal{H}$. Let $T_{\mathcal{G}\cup \mathcal{H}}$ be the associated JSJ $(G,\mathcal{G} \cup \mathcal{H})$-elementary splitting given by Theorem~\ref{thm:JSJrelhyp}. 

Let $\phi=[\Phi] \in \Out(G)$. By Proposition~\ref{lem:trivialactionJSJsplitting}, the outer automorphism $\phi$ acts trivially on the quotient graph $\overline{G \backslash T_{\mathcal{G}\cup \mathcal{H}}}$. Moreover, by the Claim in the proof of Theorem~\ref{thm:invariantfamilycyclicfixed}, the outer automorphism $\phi$ acts trivially on every edge group with cyclic stabiliser.

The group $H$ fixes a point $v$ in $T_{\mathcal{G}\cup \mathcal{H}}$. Moreover, one can choose $v$ so that $G_v$ is preserved by $\Phi$. Indeed, By Lemma~\ref{lem:elemornonelem}, the group $H$ fixes a unique elementary vertex or a unique nonelementary vertex. We investigate different cases according to $v$.

If $G_v$ is elementary noncyclic, then, by Item~$(4)$ of the \hyperlink{SA}{Standing Assumptions}, the automorphism $\Phi$ acts trivially on $H$.

Suppose that $G_v$ is elementary cyclic. Since $\phi$ acts trivially on cyclic edge stabilisers of $T_{\mathcal{G}\cup \mathcal{H}}$, we see that $\Phi_v$ acts trivially on $H$.

Suppose that $G_v$ is nonelementary and rigid. By Lemma~\ref{lem:rigidvertexacttrivially}, the automorphism $\Phi$ acts on $G_v$ as a global conjugation by an element $h \in G_v$. Thus, for every $\ell >0$, we have $\Phi^\ell(g)=h^\ell gh^{-\ell}$.  As $g$ is $\Phi$-periodic, there exists $m >0$ such that $h^m \in C(g)$.

Suppose that $g$ is $(G,\mathcal{G})$-nonperipheral. Then its centraliser is infinite cyclic and root--closed by Lemma~\ref{lem:nonperipheralnotconjinverse}. Hence $h \in C(g)$ and $g$ is fixed by $\Phi$.

Suppose that there exists $[G_i] \in \mathcal{G}$ with $g \in G_i$. Then $G_i$ is preserved by $\Phi$. This case then follows from Item~$(4)$ of the \hyperlink{SA}{Standing Assumptions}.

Finally, suppose that $G_v$ is nonelementary and flexible. By Theorem~\ref{thm:JSJrelhyp}$(5)$, there exists a compact surface $\Sigma_v$ such that $G_v$ is isomorphic to $\pi_1(\Sigma_v)$. Moreover, the group $H$ is contained in a unique boundary subgroup. By Lemma~\ref{lem:flexiblevertexacttrivially}, the automorphism $\Phi$ fixes elementwise this boundary subgroup. This concludes the proof.
\end{proof}

\section{Periodic free factor systems}\label{Section:periodicffs}

Let $G$ be a group satisfying the \hyperlink{SA}{Standing Assumptions}. In this section, we prove the following theorem.

\begin{theo}\label{thm:periodicfreefactorfixed}
Let $G$ be a group satisfying the \hyperlink{SA}{Standing Assumptions}. Let $\phi \in \IA(G,\mathcal{G},3)$. Every $\phi$-periodic $(G,\mathcal{G})$-free factor system is fixed by $\phi$.
\end{theo}

The proof of Theorem~\ref{thm:periodicfreefactorfixed} relies on the study of maximal $\phi$-periodic $(G,\mathcal{G})$-free factor systems. We use two distinct techniques according to whether the free factor system is sporadic or not. 

If the free factor system is sporadic, then it induces a natural $\phi$-periodic $(G,\mathcal{G})$-free splitting. Therefore, we are left with studying $\phi$-periodic $(G,\mathcal{G})$-free splittings. This is the content of Section~\ref{sec:periodicfreesplittings}. 

When the free factor system is nonsporadic, one can naturally associate to it an \emph{attracting lamination} of $\phi$. Attracting laminations are dynamical invariants of $\phi$ on which $\phi$ naturally acts. The action of $\phi$ on the (finite) set of attracting laminations is encoded in a relative train track of $\phi$. This action is studied in details in Section~\ref{sec:periodiclamination}.

\subsection{JSJ decompositions of suspensions and periodic free splittings}\label{sec:periodicfreesplittings}

Let $G$ be a group satisfying the \hyperlink{SA}{Standing Assumptions} and let $\phi \in \Out(G,\mathcal{G})$. In this section, we prove the following theorem.

\begin{theo}\label{thm:periodicfreesplittingfixed}
Let $G$ be a group satisfying the \hyperlink{SA}{Standing Assumptions}. Let $\phi \in \IA(G,\mathcal{G},3)$. Every $\phi$-periodic $(G,\mathcal{G})$-free splitting is fixed by $\phi$.
\end{theo}

Let $\Phi \in \phi$ be a representative of $\phi$. The proof requires the understanding of cyclic splittings of the \emph{suspension of $G$ by $\Phi$}, which is the group 
$$G_\Phi=G \rtimes_\Phi \ZZ=\langle G,t \;|\; tgt^{-1}=\phi(g) \; \forall g \in G \rangle.$$ If $\Psi \in \Aut(G,\mathcal{G})$ is in the same outer class as $\Phi$, then $G_\Phi$ and $G_\Psi$ are isomorphic. Therefore, if $\phi \in \Out(G,\mathcal{G})$ and $\Phi \in \phi$, we can denote by $G_\phi$ a group whose one of its presentation is $G_\Phi$.

We will see in Proposition~\ref{prop:invariantsplittingJSJ} that the tree of cylinders of a JSJ $G_\phi$-cyclic splitting encodes the set of all $\phi$-periodic $(G,\mathcal{G})$-free splittings. This shows that it will be sufficient to understand the action of $G_\phi$ on its tree of cylinders in order to understand all $\phi$-periodic $(G,\mathcal{G})$-free splittings. 

Note that if $T$ is a $G_\phi$-splitting in which $G$ does not fix a point, then the restricted action of $G$ on $T$ induces a $G$-splitting $T^G$. Indeed, the fact that the action of $G$ on $T$ is minimal follows from the facts that $G$ is a normal subgroup of $G_\phi$ and that the action of $G_\phi$ on $T$ is minimal. We record in the following lemma information regarding the tree $T$ seen either as a $G_\phi$-splitting or a $G$-splitting.

\begin{lem}[{\cite[Lemmas~2.6, 2.7, 2.8, 2.9]{Dah2016}}]\label{lem:splittingsuspensionvsfiber}
Let $(G,\mathcal{G})$ be a free product and let $\phi \in \Out(G,\mathcal{G})$. Let $T$ be a $G_\phi$-splitting in which $G$ does not fix a point.
\begin{enumerate}
\item If $c$ is a cell of $T$, then the $G_\Phi$-stabiliser of $c$ is isomorphic to $G_c \rtimes \ZZ$, where $G_c$ is the $G$-stabiliser of $c$.
\item If the $G_\Phi$-edge stabilisers are all virtually cyclic, there does not exist a cell $c$ of $T$ whose $G_\Phi$-stabiliser is infinitely-ended.
\item If $T$ is a $G_\phi$-cyclic splitting such that, for every $[G_i] \in \mathcal{G}$, the group $G_i$ is elliptic, then $T^G$ is a $(G,\mathcal{G})$-free splitting.
\end{enumerate}
\end{lem}

Let $\mathcal{F}$ be a sporadic proper $(G,\mathcal{G})$-free factor system. One can canonically associate to $\mathcal{F}$ a $(G,\mathcal{G})$-free splitting $T_\mathcal{F}$. The conjugacy classes of vertex stabilisers of $T_\mathcal{F}$ are contained in $\mathcal{F}$ and $T_\mathcal{F}$ has a unique orbit of edges. The group $\Aut(G,\mathcal{F})$ preserves the $G$-equivariant isometry class of $T_\mathcal{F}$. In particular, we have the following, where the last assertion follows from Lemma~\ref{lem:splittingsuspensionvsfiber}(1).

\begin{lem}\label{lem:sporadicfreefactorsystemsplitting}
Let $(G,\mathcal{G})$ be a free product and let $\mathcal{F}$ be a sporadic $(G,\mathcal{G})$-free factor system. Let $\Phi \in \Aut(G,\mathcal{F})$.

\begin{enumerate}
\item The action of $G$ on $T_\mathcal{F}$ extends to an action of $G_\Phi$ on $T_\mathcal{F}$. 
\item The $G_\Phi$-splitting $T_\mathcal{F}$ is a $G_\Phi$-cyclic splitting.
\hfill\qedsymbol
\end{enumerate}
\end{lem}

By Lemma~\ref{lem:splittingsuspensionvsfiber}$(1)$, the group $G_\phi$ is one-ended. Since $G$ is finitely presented by the \hyperlink{SA}{Standing Assumptions}, so is $G_\phi$.

Let $T_{JSJ}(G_\phi)$ be a JSJ $G_\phi$-cyclic splitting given by Theorem~\ref{thm:JSJ} and let $T_c(G_\phi)$ be the associated tree of cyclinders. We note that the tree $T_c(G_\phi)$ was already considered by Guirardel--Horbez~\cite[Theorem~6.12]{Guirardelhorbez} when $G$ is a free group.

We now give a description of the vertex stabilisers of the tree $T_c(G_\phi)$ (considered as a $G_\phi$-splitting). If $\Phi \in \phi$, recall that $\mathrm{Per}(\Phi)$ is the periodic subgroup of $\Phi$. 

\begin{lem}\label{lem:stabvertextreecyl}
Let $v \in VT_c(G_\phi)$. 

\begin{enumerate}
\item Suppose that $v$ corresponds to a cylinder $C \subseteq T_{JSJ}(\phi)$. Let $e \in EC$, let $\ell >0$ and let $\Phi \in \phi^\ell$ which generates $(G_\phi)_e$. There exist $m>0$ dividing $\ell$ and $\Psi \in \phi^m$ such that the stabiliser of $v$ in $G_\phi$ is equal to $\mathrm{Per}(\Phi) \rtimes_\Psi \ZZ$. 
\item If $v$ is a vertex of $T_{JSJ}(G_\phi)$, there exist a $(G,\mathcal{G})$-free factor $A$, $\ell >0$ and a representative $\Phi \in \phi^\ell$ preserving $A$ such that $(G_\phi)_v$ is equal to $A \rtimes_\Phi \ZZ$. Moreover, the set of $G$-conjugacy classes of vertex stabilisers of $T_{JSJ}^G(G_\phi)$ induces a $(G,\mathcal{G})$-free factor system.
\end{enumerate}

\end{lem}

\begin{proof}
The fact that the stabiliser of a cylinder vertex is isomorphic to $\mathrm{Per}(\Phi) \rtimes_\Psi \ZZ$ follows from~\cite[Lemma~6.2]{AndGueHug2023}. 

We now prove the second assertion. The tree $T_{JSJ}(G_\phi)$ is a $G_\phi$-cyclic splitting. Therefore, by Lemma~\ref{lem:splittingsuspensionvsfiber}$(3)$, the tree $T_{JSJ}^G(G_\phi)$ is a $(G,\mathcal{G})$-free splitting. In particular, every vertex stabiliser of $T_{JSJ}^G(G_\phi)$ is a $(G,\mathcal{G})$-free factor and the set of $G$-conjugacy classes of vertex stabilisers of $T_{JSJ}^G(\phi)$ is a $(G,\mathcal{G})$-free factor system. By Lemma~\ref{lem:splittingsuspensionvsfiber}$(1)$, for every vertex $v$ of $T_{JSJ}(G_\phi)$, there exist a free factor $A$ of $(G,\mathcal{G})$, $\ell >0$ and a representative $\Phi \in \phi^\ell$ preserving $A$ such that $(G_\phi)_v$ is equal to $A \rtimes_\Phi \ZZ$.
\end{proof} 

\begin{prop}\label{prop:invariantsplittingJSJ}
Let $S$ be a $\phi$-invariant $(G,\mathcal{G})$-free splitting. Then $S$ and $T_c^G(G_\phi)$ are compatible as $(G,\mathcal{G})$-splittings. Moreover, a common refinement is obtained from $T_c^G(G_\phi)$ by blowing-up trees at its cylinder vertices. 
\end{prop}

\begin{proof}
Since $S$ is $\phi$-invariant, the $G$-action on $S$ extends to a $G_\phi$-action. By Lemma~\ref{lem:splittingsuspensionvsfiber}, $S$ is a $G_\phi$-cyclic splitting. We want to apply Lemma~\ref{lem:treecylinderscompatible} to conclude the proof. Therefore, it suffices to show that every vertex of $T_{JSJ}(G_\phi)$ is rigid. 

By Theorem~\ref{thm:JSJ}, flexible vertices are infinitely ended. By Lemma~\ref{lem:splittingsuspensionvsfiber}$(2)$, there does not exist a $G_\phi$-cyclic splitting with an infinitely-ended vertex stabiliser. Therefore, every vertex of $T_{JSJ}(G_\phi)$ is rigid. Thus, by Lemma~\ref{lem:treecylinderscompatible}, a common refinement is obtained from $T_c(G_\phi)$ by blowing-up trees at its cylinder vertices. 
\end{proof}

We can now prove the main theorem of the section.

\begin{proof}[Proof of Theorem~\ref{thm:periodicfreesplittingfixed}]
Let $\phi \in \IA(G,\mathcal{G},3)$. Let $S$ be a $\phi$-periodic $(G,\mathcal{G})$-free splitting. Let $\ell >0$ be such that $\phi^\ell$ fixes $S$ and acts trivially on $\overline{G \backslash S}$. Let $T_c(G_{\phi^\ell})$ be the tree of cylinders associated with a JSJ $(G_{\phi^\ell},\mathcal{G})$-cyclic splitting given by Proposition~\ref{prop:treecylinderscanonical}. Note that, for every $[G_i] \in \mathcal{G}$, the group $G_i$ is elliptic in $T_c(G_{\phi^\ell})$ as it is elliptic in any JSJ $(G_{\phi^\ell},\mathcal{G})$-cyclic splitting. We denote by $T_c^G(G_{\phi^\ell})$ the tree $T_c(G_{\phi^\ell})$ considered as a $(G,\mathcal{G})$-splitting. It comes equipped with a partition of its vertices $VT_c^G(G_{\phi^\ell})=\widetilde{V}_0 \coprod \widetilde{V}_1$ where, for every $v \in \widetilde{V}_0$, $v$ corresponds to a vertex in $T_{JSJ}(G_{\phi^\ell})$ and, for every $v \in \widetilde{V}_1$, $v$ corresponds to a cylinder of $T_{JSJ}(G_{\phi^\ell})$. Let $V\overline{G \backslash T_c^G(G_{\phi^\ell})}=V_0 \coprod V_1$ be the induced partition.

Note that the group $G_{\phi^\ell}$ is a normal subgroup of $G_\phi$, so that the action of $G_\phi$ on $G_{\phi^\ell}$ by conjugation induces a homomorphism $G_\phi \to \Aut(G_{\phi^\ell})$. Since $T_c(G_{\phi^\ell})$ is preserved by $\Aut(G_{\phi^\ell})$ by Proposition~\ref{prop:treecylinderscanonical}, we obtain an action of $G_\phi$ on $T_c(G_{\phi^\ell})$. In particular, the outer automorphism $\phi$ preserves $T_c^G(G_{\phi^\ell})$.

\medskip

\noindent{\bf Claim~1. } Let $v \in \widetilde{V}_1$ be a cylinder vertex. There exists $\Phi \in \phi$ such that $G_v=\mathrm{Fix}(\Phi)$.

\begin{proof}
By Lemma~\ref{lem:stabvertextreecyl}, there exists $\Psi \in \phi^\ell$ such that the $G$-stabiliser of $v$ is isomorphic to $\mathrm{Per}(\Psi)$. By Theorem~\ref{thm:invariantfamilycyclicfixed} applied to $\mathcal{H}=\{\phi^n[\mathrm{Per}(\Psi)]\}_{n \in \NN}$, there exists $\Phi \in \phi$ such that $\mathrm{Per}(\Psi)=\mathrm{Fix}(\Phi)$. In particular, the $G$-stabiliser of any cylinder of $T_c^G(G_{\phi^\ell})$ is a fixed subgroup of some $\Phi \in \phi$. 
\end{proof}

\medskip

\noindent{\bf Claim~2. } The outer class $\phi$ fixes pointwise the graph $\overline{G \backslash T_c^G(G_{\phi^\ell})}$.

\begin{proof}
By Lemma~\ref{lem:actionhomograph}, $\phi$ acts trivially on $H_1(\overline{G \backslash T_c^G(G_{\phi^\ell})},\ZZ/3\ZZ)$. In order to apply Lemma~\ref{lem:graphautotrivial}, we need to prove that $\phi$ fixes every leaf of $\overline{G \backslash T_c^G(G_{\phi^\ell})}$.

\medskip

\noindent{\it Step~1.} Let $v \in V_0$. If $G_v$ is nontrivial, then $\phi$ fixes $v$.

\begin{proofblack}
By Lemma~\ref{lem:stabvertextreecyl}$(1)$, the set of $G$-stabilisers of vertices in $\widetilde{V}_0$ is a $\phi$-invariant $(G,\mathcal{G})$-free factor system $\mathcal{F}_{\widetilde{V}_0}$. Thus $v$ is the unique vertex in $\widetilde{V}_0$ fixed by $G_v$. Therefore, it suffices to prove that $\phi$ preserves the $G$-conjugacy class of $G_v$. By Lemma~\ref{lem:periodic freefactorsystem}, every $G$-conjugacy class in $\mathcal{F}_{\widetilde{V}_0}$ is $\phi$-invariant. In particular, $[G_v]$ is $\phi$-invariant and $v$ is fixed by $\phi$.
\end{proofblack}

\medskip

\noindent{\it Step~2.} The outer class $\phi$ fixes every leaf of $\overline{G \backslash T_c^G(G_{\phi^\ell})}$.

\begin{proofblack}
Let $v$ be a leaf of $\overline{G \backslash T_c^G(G_{\phi^\ell})}$, let $e$ be its adjacent edge and let $w$ be the endpoint of $e$ distinct from $v$. 

Suppose first that $v \in V_0$. Since $v$ is a leaf, by minimality of the action of $G$ on $T_c^G(G_{\phi^\ell})$, its $G$-stabiliser is nontrivial. Thus, by Step~1, $v$ is fixed by $\phi$.

Assume now that $v \in V_1$. By Claim~1, there exists $\Phi \in \phi$ such that $G_v=\mathrm{Fix}(\Phi)$. By minimality of the action of $G$ on $T_c^G(G_{\phi^\ell})$, we have $G_v \nsubseteq G_w$. Thus, $v$ is the only fixed point of $\mathrm{Fix}(\Phi)$, so that $\Phi$ (and hence $\phi$) fixes $v$. Thus, $\phi$ fixes every leaf of $\overline{G \backslash T_c^G(G_{\phi^\ell})}$.
\end{proofblack}

By Step~2 and Lemma~\ref{lem:graphautotrivial}, either $\phi$ fixes pointwise $\overline{G \backslash T_c^G(G_{\phi^\ell})}$, or else $\overline{G \backslash T_c^G(G_{\phi^\ell})}$ is a circle and $\phi$ acts as a rotation on it. So, in order to prove that $\phi$ fixes pointwise $\overline{G \backslash T_c^G(G_{\phi^\ell})}$, it remains to prove that $\phi$ fixes a point in it. 

By Step~1, if there exists $v \in V_0$ with $G_v \neq \{1\}$, then $\phi$ fixes $v$. Assume that, for every $v \in V_0$, we have $G_v=\{1\}$. Since every edge of $T_c^G(G_{\phi^\ell})$ has an endpoint in $\widetilde{V}_0$, it follows that $T_c^G(G_{\phi^\ell})$ is a $(G,\mathcal{G})$-free splitting. By Lemma~\ref{lem:trivialactionfreesplitting}, since $\phi \in \IA(G,\mathcal{G},3)$, $\phi$ fixes pointwise $\overline{G \backslash T_c^G(\phi^\ell)}$. As we have ruled out every case, the claim follows.
\end{proof}

\medskip

\noindent{\bf Claim~3. } Let $v \in \widetilde{V}_1$ be a cylinder vertex. There exists $\Phi \in \phi$ such that $\Phi$ fixes pointwise the star of $v$ in $T_c(G_{\phi^\ell})$.

\begin{proof}
Let $C\subseteq T_{JSJ}(\phi^\ell)$ be the cylinder corresponding to $v$ and let $\Psi \in \phi^\ell$ be a generator of an edge stabiliser in $C$ ($\Psi$ is not contained in a strict power of $\phi^\ell$ since $\phi^\ell \in \IA(G,\mathcal{G},3)$ acts trivially on $\overline{G\backslash T_{JSJ}^G(\phi^\ell)}$ by Lemma~\ref{lem:trivialactionfreesplitting}). By Corollary~\ref{coco:periodicelementfixed}, the periodic subgroup of $\Psi$ is equal to its fixed subgroup, so that $G_v=\mathrm{Fix}(\Psi)$. 

\medskip

\noindent{\it Step~1. } Let $e_1,\ldots,e_n$ be representatives of the $G_v$-orbits of edges adjacent to $v$ with trivial $G$-stabiliser. There exists $\Phi \in \phi$ such that $G_v=\Fix(\Phi)$ and, for every $i \in \{1,\ldots,n\}$, the edge $e_i$ is fixed by $\Phi$. 

\begin{proofblack}
We first claim that there exists $m >0$ such that, for every $i \in \{1,\ldots,n\}$, the automorphism $\Psi^m$ fixes $e_i$. Indeed, by definition of a cylinder, all edge stabilisers in $C$ are commensurable. As $\Psi$ fixes an edge in $C$, there exists $m>0$ such that $\Psi^m$ fixes pointwise a finite subtree of $C$ containing every vertex of $C$ corresponding to the endpoints of the edges $e_i$ in $\widetilde{V}_0$. This proves the claim.

Let $i \in \{1,\ldots,n\}$. By Claim~2, there exists $\Phi_i \in \phi$ which fixes $e_i$. Since the $G$-stabiliser of $e_i$ is trivial, the $G_\phi$-stabiliser of $e_i$ is cyclic. Thus, $\Phi_i$ and $\Psi^m$ have a common power, so that $\mathrm{Per}(\Phi_i)=\mathrm{Per}(\Psi^m)=\Fix(\Psi)=G_v$. By Corollary~\ref{coco:periodicelementfixed}, we have $\mathrm{Per}(\Phi_i)=\Fix(\Phi_i)=G_v$.

Thus, for every $i \in \{1,\ldots,n\}$, there exists $\Phi_i \in \phi$ such that $\Fix(\Phi_i)=G_v$, $\Phi_i$ fixes $e_i$ and $\Phi_i$ and $\Psi$ have a common power. For any $i,j \in \{1,\ldots,n\}$, $\Phi_i$ and $\Phi_j$ differ by an element $g_{i,j}$ of $G_v=\Fix(\Phi_i)=\Fix(\Phi_j)$. Then $\Phi_i$ sends $e_j$ to $g_{i,j}e_j$. So, for every $p>0$, we have $\Phi_i^p(e_j)=g_{i,j}^{p}e_j$. But $\Phi_i$, $\Phi_j$ and $\Psi$ have a common power, so there exists $p_0>0$ such that $\Phi_i^{p_0}(e_j)=e_j$. As $G$ is torsion free, this forces $g_{i,j}=1$. Hence $\Phi_i=\Phi_j$. Setting $\Phi=\Phi_1$, we see that $\Phi$ fixes every edge $e_i$ with $i \in \{1,\ldots,n\}$ and its fixed subgroup is equal to $G_v$. This proves the first step.
\end{proofblack}

Let $\Phi \in \phi$ be the representative given by Step 1. Let $g\in G_v$ and let $i \in \{1,\ldots,n\}$. As $G_v=\mathrm{Fix}(\Phi)$, we have $\Phi(ge_i)=ge_i$, so that $\Phi$ fixes every edge in the $G_v$-orbit of $e_i$.

\medskip

\noindent{\it Step 2. } The automorphism $\Phi$ fixes every edge $e$ adjacent to $v$ such that $G_e$ is nontrivial.

\begin{proofblack}
Let $e$ be an edge adjacent to $v$ such that $G_e$ is nontrivial. Let $w$ be the other endpoint of $e$. By Lemma~\ref{lem:stabvertextreecyl}, the group $G_w$ is a $(G,\mathcal{G})$-free factor. Moreover, its $G$-conjugacy class is fixed by $\Phi$. Since $G_e \neq \{1\}$, the group $G_w$ is the only group in its $G$-conjugacy class which contains $G_e$. Since $G_e \subseteq \Fix(\Phi)$, the group $G_w$ is preserved by $\Phi$. As $w$ is the only vertex in $\widetilde{V}_0$ fixed by $G_w$, we conclude that $w$ and $e$ are fixed by $\Phi$. This proves Step 2.
\end{proofblack}

Combining both Steps 1 and 2, we see that the automorphism $\Phi$ satisfies the assertions of the claim.
\end{proof}

Since $S$ is fixed by $\phi^\ell$, by Proposition~\ref{prop:invariantsplittingJSJ}, a common refinement $U$ of $S$ and $T_c^G(G_{\phi^\ell})$ is obtained by blowing-up trees at cylinder vertices of $T_c^G(G_{\phi^\ell})$ (we consider $U$ as a $G$-splitting and not a $G_{\phi^\ell}$-splitting). Let $\widetilde{V}$ be the set of vertices in $\widetilde{V}_1$ which are blown-up when passing from $T_c^G(G_{\phi^\ell})$ to $U$. By Claim~3, for every $v \in \widetilde{V}$, the automorphism $\phi$ has a representative $\Phi_v$ which pointwise fixes $G_v$ and which fixes every edge of $T_c^G(\phi^\ell)$ adjacent to $v$. The automorphism $\Phi_v$ therefore fixes every $G$-tree obtained from $T_c^G(\phi^\ell)$ by blowing-up a subtree at $v$. In particular, $U$ is preserved by $\phi$. 

We claim that $\phi$ acts trivially on $\overline{G \backslash U}$. Indeed, let $p_U \colon U \to T_c^G(G_{\phi^\ell})$ be the natural $G$-equivariant projection and let $x$ be a point in $T_c^G(G_{\phi^\ell})$. If $p_U^{-1}(x)$ is a point, then, by Claim~2 and the $G$-equivariance of $p_U$, $\phi$ preserves the $G$-orbit of $p_U^{-1}(x)$. If $p_U^{-1}(x)$ is not reduced to a point, then $x \in \widetilde{V}$. By Claim~3, $\phi$ has a representative which fixes pointwise $p_U^{-1}(x)$. It follows that $\phi$ fixes pointwise $\overline{G \backslash U}$.

Therefore, $\phi$ preserves every $(G,\mathcal{G}$)-splitting on which $U$ collapses. Since $U$ is a refinement of $S$, this shows that $\phi$ preserves $S$. 
\end{proof}

\subsection{Laminations and nonsporadic free factor systems}\label{sec:periodiclamination}

Let $G$ be a group satisfying the \hyperlink{SA}{Standing Assumptions} and let $\phi \in \Out(G,\mathcal{G})$. Let $\mathcal{F}_{\mathrm{max}}(\phi)$ be the set of all proper, maximal $\phi$-periodic $(G,\mathcal{G})$-free factor systems. The set $\mathcal{F}_{\mathrm{max}}(\phi)$ has a partition 
\[\mathcal{F}_{\mathrm{max}}(\phi)=\mathrm{S}(\phi) \coprod \mathrm{NS}(\phi),\] where $\mathrm{S}(\phi)$ contains all the $(G,\mathcal{G})$-free factor systems in $\mathcal{F}_{\mathrm{max}}(\phi)$ which are sporadic and $\mathrm{NS}(\phi)$ contains all the $(G,\mathcal{G})$-free factor systems in $\mathcal{F}_{\mathrm{max}}(\phi)$ which are nonsporadic. In this section, we focus on the action of $\phi$ on $\mathrm{NS}(\phi)$. The main result of the section is the following proposition.

\begin{prop}\label{prop:maxperiodffsfixed}
Let $G$ be a group which satisfies the \hyperlink{SA}{Standing Assumptions}. Let $\phi \in \IA(G,\mathcal{G},3)$ and let $\mathcal{F} \in \mathrm{NS}(\phi)$. Then $\mathcal{F}$ is fixed by $\phi$.
\end{prop}

The proof of Proposition~\ref{prop:maxperiodffsfixed} uses the existence of \emph{attracting laminations} associated with $\phi$. An attracting lamination is a particular closed subset of the double boundary of $G$ which is invariant by a power of a representative of $\phi$. We in particular show (see Lemma~\ref{lem:attractinglaminationassociatedmaxfree}) that, to any maximal $\phi$-periodic nonsporadic free factor system one can canonically associate an attracting lamination of $\phi$. Thus, instead of studying the dynamics of $\phi$ on its periodic nonsporadic free factor systems, we will study the dynamics of $\phi$ on its set of attracting laminations. This is well-described using a relative train track representative of $\phi$ given by Theorem~\ref{thm:existencereltraintracks}.

\bigskip

First, we explain the necessary background relevant to the study of attracting laminations. Let $\phi \in \Out(G,\mathcal{G})$. Following the work of Bestvina--Feighn--Handel~\cite{BesFeiHan97,BesFeiHan00} in the context of free groups, Lyman~\cite{Lyman2022CT} introduced a set $\mathcal{L}(\phi)$ of laminations associated with $\phi$, called the set of \emph{attracting laminations of $\phi$}. Instead of giving a precise definition of an attracting lamination (we refer to \cite[Section~4]{Lyman2022CT} instead), we present several properties of $\mathcal{L}(\phi)$ and its links with relative train tracks representing $\phi$. 

Let $F \colon T \to T$ be a relative train track representative of $\phi$ given by Theorem~\ref{thm:existencereltraintracks}. We now record several properties of the set $\mathcal{L}(\phi)$.

\begin{lem}[{\cite[Lemma~4.4]{Lyman2022CT}}]\label{lem:genericline}
Let $(G,\mathcal{G})$ be a free product, let $\phi \in \Out(G,\mathcal{G})$ and let $F \colon T \to T$ be a relative train track representative of $\phi$. The following hold.

\begin{enumerate}
\item For every attracting lamination $\Lambda \in \mathcal{L}(\phi)$, there exists a line $\ell_\Lambda$ of $\partial^2(G,\mathcal{G})$ such that $\Lambda$ is the minimal lamination of $\partial^2(G,\mathcal{G})$ containing $\ell_\Lambda$. Such a line $\ell_\Lambda$ is called a \emph{generic line of $\Lambda$}.
\item For every $\Lambda \in \mathcal{L}(\phi)$, there exists an EG stratum $H_r$ such that, for every generic line $\ell_\Lambda$ of $\Lambda$, the stratum $H_r$ is the highest stratum crossed by $\ell_\Lambda$.
\item Let $\Lambda \in \mathcal{L}(\phi)$ and let $H_r$ be the associated EG stratum. There exists an edge $e \in EH_r$ and a generic line $\ell_\Lambda$ of $\Lambda$ such that, for any finite subpath $\gamma$ of $\ell_\Lambda$, there exist $n >0$ and $g \in G$ such that $\gamma \subseteq g[F^n(e)]$.
\end{enumerate}
\end{lem}

Lemma~\ref{lem:genericline}$(2)$ implies that there exists a $\phi$-equivariant map $p_\phi \colon \mathcal{L}(\phi) \to \mathrm{EG}$, where $\mathrm{EG}$ is the set of EG strata of $T$. We now describe properties of this map.

Let $H_r$ be an EG stratum. The stratum $H_r$ is \emph{aperiodic} if its associated transition matrix $M_r$ is aperiodic, that is there exists $\ell >0$ such that every entry of $M_r^\ell$ is positive. If $H_r$ is not aperiodic, we have the following, whose proof is due to Lyman~\cite{Lyman2022CT}. We add a sketch of proof as the exact assertion is not explicitly written.

\begin{lem}[{\cite[Proof of Lemma~4.6]{Lyman2022CT}}]\label{lem:periodicEGstratum}
Let $(G,\mathcal{G})$ be a free product, let $\phi \in \Out(G,\mathcal{G})$ and let $F \colon T \to T$ be a relative train track representative of $\phi$. Let $H_r \in \mathrm{EG}$ be a non aperiodic EG stratum. 

There exists a $G$-invariant partition $$EH_r=P_{1}^{(r)} \amalg \ldots P_{n_r}^{(r)}$$ of the edges of $H_r$ such that the following hold. 

\begin{enumerate}
\item For every $i \in \{1,\ldots,n_r\}$ and every edge $e \in P_i^{(r)}$ the path $[F(e)]$ is contained in $T_{r-1} \cup P_{i+1}^{(r)}$, where indices are taken modulo $n_r$.
\item For any distinct $i,j \in \{1,\ldots,n_r\}$, we have $$\mathcal{F}(T_{r-1} \cup P_i^{(r)}) \neq \mathcal{F}(T_{r-1} \cup P_j^{(r)}).$$
\end{enumerate}
\end{lem}

\begin{proof}[Sketch of proof]
The partition of $EH_r$ comes from general theory \cite{Seneta81} applied to the set of edges of $\overline{G \backslash H_r}$ and its transition matrix. The second assertion follows from the fact that there exists $s>0$ such that, for every $i \in \{1,\ldots,n_r\}$, the set $P_i^{(r)}$ is the set of edges of an aperiodic EG stratum of the relative train track map $F^s \colon T \to T$ (see~\cite[Proof of Lemma~4.6]{Lyman2022CT}). The fact that, for every $i \in \{1,\ldots,n_r\}$ the stratum $P_i^{(r)}$ is an EG stratum for $F^s$ implies that for any distinct $i,j \in \{1,\ldots,n_r\}$, there exists $g \in G$ which is both $(G,\mathcal{F}(T_{r-1} \cup P_i^{(r)}))$-peripheral and $(G,\mathcal{F}(T_{r-1} \cup P_j^{(r)}))$-nonperipheral. Indeed, it suffices to take any nontrivial element in the fundamental group of the graph of groups $\pi_1(G \backslash (T_{r-1} \cup P_i^{(r)}))$ whose associated reduced path crosses $P_i^{(r)}$ (it exists as $P_i^{(r)}$ is an EG stratum). This shows that $\mathcal{F}(T_{r-1} \cup P_i^{(r)}) \neq \mathcal{F}(T_{r-1} \cup P_j^{(r)})$.
\end{proof}

\begin{prop}[\cite{Lyman2022CT}]\label{prop:attractinglaminations}
Let $(G,\mathcal{G})$ be a free product, let $\phi \in \Out(G,\mathcal{G})$ and let $F \colon T \to T$ be a relative train track representative of $\phi$. The following hold.

\begin{enumerate}
\item \cite[Lemma~4.3]{Lyman2022CT} The application $p_\phi \colon \mathcal{L}(\phi) \to \mathrm{EG}$ is surjective.
\item \cite[Proof of Lemma~4.6]{Lyman2022CT} If $H_r \in \mathrm{EG}$ is aperiodic, then $|p_\phi^{-1}(H_r)|=1$.
\item \cite[Proof of Lemma~4.6]{Lyman2022CT} Suppose that $H_r \in \mathrm{EG}$ is not aperiodic and let $EH_r=P_{1}^{(r)} \amalg \ldots P_{n_r}^{(r)}$ be the associated partition given by Lemma~\ref{lem:periodicEGstratum}. For every $i \in \{1,\ldots,n_r\}$, there exists a unique $\Lambda \in p_\phi^{-1}(H_r)$ such that some (equivalently any) generic line of $\Lambda$ contains edges of $P_i^{(r)}$. 
\item \cite[Lemma~4.6]{Lyman2022CT} The set $\mathcal{L}(\phi)$ is finite.
\item \cite[Proof of Lemma~4.6]{Lyman2022CT} For every $\ell >0$, we have $\mathcal{L}(\phi^\ell)=\mathcal{L}(\phi)$.
\end{enumerate}
\end{prop}

\begin{lem}\label{lem:attractinglaminationassociatedmaxfree}
Let $(G,\mathcal{G})$ be a free product and let $\phi \in \Out(G,\mathcal{G})$. For every $\mathcal{F} \in \mathrm{NS}(\phi)$, there exists a unique $\Lambda_\mathcal{F} \in \mathcal{L}(\phi)$ such that $\Lambda_\mathcal{F} \nsubseteq \partial^2(\mathcal{F},\mathcal{G})$. In particular, the map $\rho_\phi \colon \mathrm{NS}(\phi) \to \mathcal{L}(\phi)$ sending $\mathcal{F}$ to $\Lambda_\mathcal{F}$ is $\phi$-equivariant.
\end{lem}

\begin{proof}
Let $\mathcal{F} \in \mathrm{NS}(\phi)$. Let $\ell >0$ be such that $\phi^\ell$ preserves $\mathcal{F}$. Let $F^\ell \colon T \to T$ be a relative train track representing $\phi^\ell$, and a filtration $T_0 \subseteq \ldots \subseteq T_n$ such that there exists an element $T_{r-1}$ of the filtration of $T$ with $\mathcal{F}(T_{r-1})=\mathcal{F}$.

Note that, since $\mathcal{F}$ is maximal, $\phi^\ell$ is fully irreducible relative to $\mathcal{F}$. In particular, one may choose the relative train track $F^\ell \colon T \to T$ so that $r=n$. Moreover, since $\mathcal{F}$ is nonsporadic, the stratum $H_r$ is an EG stratum.

We claim that $H_r$ is aperiodic. Indeed, suppose towards a contradiction that $H_r$ is not aperiodic and let $$EH_r=P_{1}^{(r)} \amalg \ldots P_{n_r}^{(r)}$$ be the associated partition of edges given by Lemma~\ref{lem:periodicEGstratum}, with $n_r \geq 2$. Recall that $F^\ell$ induces a permutation of the set $\{P_1^{(r)},\ldots,P_{n_r}^{(r)}\}$. Then $F^{\ell n_r!}$ preserves every element of the set $\{P_1^{(r)},\ldots,P_{n_r}^{(r)}\}$. Thus, $F^{\ell n_r!}$ preserves a proper filtration $$T_1 \subseteq \ldots \subseteq T_{r-1} \subseteq T_{r,1} \subsetneq \ldots \subsetneq T_{r,n_r}=T$$ of $T$. 

Recall that, by Proposition~\ref{prop:attractinglaminations}$(3)$, for every $i \in \{1,\ldots,n_r\}$, one may associate an attracting lamination $\Lambda_i$ to $T_{r,i}$. Moreover, by Lemma~\ref{lem:genericline}, for every generic line $\ell$ of $\Lambda_i$, $\ell$ crosses $T_{r,i}$. Thus, there do not exist a generic line $\ell$ of $\Lambda_i$ and a conjugacy class $[A] \in \mathcal{F}=\mathcal{F}(T_{r-1})$ such that $\ell \in \partial^2(A,\mathcal{G}_{|A})$. Therefore, for every $i \in \{1,\ldots,n_r\}$, we have $\mathcal{F}=\mathcal{F}(T_{r-1})<\mathcal{F}(T_{r,i})$. Since $n_r \geq 2$, by Lemma~\ref{lem:periodicEGstratum}$(2)$, for every $i \in \{1,\ldots,n_r-1\}$, we have $\mathcal{F}(T_{r,i})< \{[G]\}$. In particular, one may find a proper $\phi^\ell$-periodic $(G,\mathcal{G})$-free factor system $\mathcal{F} \nsqsubseteq \mathcal{F}'$. This contradicts the maximality of $\mathcal{F}$. Hence $H_r$ is aperiodic. 

Thus, by Proposition~\ref{prop:attractinglaminations}$(2)$ one can associate to $\mathcal{F}$ an attracting lamination of $\mathcal{L}(\phi^\ell)=\mathcal{L}(\phi)$ by sending $\mathcal{F}$ to the unique attracting lamination $\Lambda_\mathcal{F}$ associated with $H_r$. This attracting lamination does not depend on the choice of the relative train track because of the following facts. Note that $H_r$ is the unique EG stratum which is not contained in $T_{r-1}$. Therefore, by Lemma~\ref{lem:genericline}, a generic line of an attracting lamination $\Lambda \in \mathcal{L}(\phi)$ is contained in $T_{r-1}$ if and only if $\Lambda \neq \Lambda_\mathcal{F}$. As $\mathcal{F}(T_{r-1})=\mathcal{F}$, this implies that $\Lambda \subseteq \partial^2(\mathcal{F},\mathcal{G})$ if and only if $\Lambda \neq \Lambda_\mathcal{F}$. Thus the attracting lamination $\Lambda_\mathcal{F}$ only depends on $\mathcal{F}$ and we have a well-defined $\phi$-equivariant map $\rho_\phi \colon \mathrm{NS}(\phi) \to \mathcal{L}(\phi)$. 
\end{proof}

We now describe the preimage of a lamination by $\rho_\phi$. We first need some technical lemmas. The first one, called the \emph{bounded cancellation lemma} is well-known to experts. Its proof in the free group case is due to Cooper~\cite[Bounded Cancellation]{Cooper87}.

\begin{lem}[{\cite[Lemma~1.8]{Lyman2022CT}}]\label{Lem Bounded cancellation lemma}
Let $(G,\mathcal{G})$ be a free product, let $\phi \in \Out(G,\mathcal{G})$, let $T$ be a Grushko $(G,\mathcal{G})$-free splitting and let $F \colon T \to T$ be a morphism associated with $\phi$. There exists a constant $C(F) \geq 0$ such that for any reduced path $\rho=\rho_1\rho_2$ in $T$ we have $$\ell([F(\rho)]) \geq \ell([F(\rho_1)]) + \ell([F(\rho_2)])-2C(F).$$
\end{lem}

The second lemma deals with iterates by a relative train track of a legal subpath in some illegal path. It is implicit in the work of Bestvina--Feighn--Handel~\cite[Lemma~5.5]{BesFeiHan97} in the free group case.

\begin{lem}[{\cite[Corollary 6.1.8]{FrancaMarSyri2021}}]\label{lem:iteratelegal}
Let $(G,\mathcal{G})$ be a nonsporadic free product. Let $\phi \in \Out(G,\mathcal{G})$ be fully irreducible and let $F \colon T \to T$ be a relative train track map for $\phi$. Let $g \in G$ be $(G,\mathcal{G})$-nonperipheral and let $L_g \subseteq T$ be the associated geodesic line. Suppose that the sequence $\{\ell_T(\phi^n([g])\}_{n \in \NN}$ is unbounded. For every $C>0$, there exists $M >0$ such that the reduced path $[F^M(L_g)]$ contains a legal subpath of length greater than $C$.
\end{lem}

Let $\mathcal{F} \in \mathrm{NS}(\phi)$ and let $\ell>0$ be the minimal integer such that $\phi^\ell(\mathcal{F})=\mathcal{F}$. By maximality of $\mathcal{F}$, the outer automorphism $\phi^\ell$, considered as an element of $\Out(G,\mathcal{F})$, is fully irreducible relative to $\mathcal{F}$. We denote by $T_\mathcal{F}$ the $\RR$-tree equipped with a $G$-action given by Theorem~\ref{theo:limittreeiwip}. Let $\mathcal{A}_\mathcal{F}$ be the set of conjugacy classes of nontrivial point stabilisers in $T_\mathcal{F}$.

\begin{lem}\label{lem:preimagerho}
Let $\Lambda \in \mathcal{L}(\phi)$.
\begin{enumerate}
    \item Let $\mathcal{F},\mathcal{F}' \in \rho_\phi^{-1}(\Lambda)$. Every $(G,\mathcal{F}')$-peripheral element is elliptic in $T_\mathcal{F}$.
    \item For any $\mathcal{F},\mathcal{F}' \in \rho_\phi^{-1}(\Lambda)$, we have $\mathcal{A}_\mathcal{F}=\mathcal{A}_{\mathcal{F}'}$.
    \item Suppose that $G$ satisfies the \hyperlink{SA}{Standing Assumptions}. Suppose that $\phi \in \IA(G,\mathcal{G},3)$ and that $\phi(\Lambda)=\Lambda$. Then $\phi$ fixes elementwise $\rho_\phi^{-1}(\Lambda)$. 
\end{enumerate}
\end{lem}

\begin{proof}
The first statement of Lemma~\ref{lem:preimagerho} is stable under taking powers of $\phi$, so we may suppose that $\phi(\mathcal{F})=\mathcal{F}$ and $\phi(\mathcal{F}')=\mathcal{F}'$. 

Suppose towards a contradiction that there exists a $(G,\mathcal{F}')$-peripheral element $g$ which is not elliptic in $T_\mathcal{F}$. Let $F \colon T \to T$ be a train track representative of the fully irreducible outer automorphism $\phi \in \Out(G,\mathcal{F})$ and let $\lambda$ be its Perron--Froebenius eigenvalue. Up to taking a power of $\phi$, we may suppose that $\lambda>2$. By Theorem~\ref{theo:limittreeiwip}, we have $\ell_{T_\mathcal{F}}(g)=\lim_{n \to \infty} \frac{1}{\lambda^n}\ell_{T}(\phi^n([g]))$. 

Let $C(F)$ be the bounded cancellation constant for $F$ given by Lemma~\ref{Lem Bounded cancellation lemma}. Since $\ell_{T_\mathcal{F}}(g)>0$, the sequence $\{\ell_T(\phi^n([g])\}_{n \in \NN}$ is unbounded. By Lemma~\ref{lem:iteratelegal}, there exists $M >0$ such that, $[F^M(L_g)]$ contains a legal subpath $\gamma$ of length greater than $2C(F)+1$. 

Write $\gamma=\sigma_1 e \sigma_2$, where $\ell(\sigma_1),\ell(\sigma_2) \geq C(F)$ and $e$ is a nontrivial edge. Since $\gamma$ is legal, for every $n \in \NN$, the path $[F^n(\gamma)]$ has length greater than $\lambda^n(2C(F)+1) \geq 2^n(2C(F)+1)$ and there exist reduced paths $\gamma_1^n,\gamma_2^n$ of lengths at least $2^nC(F)$ such that $[F^n(\gamma)]=\gamma_1^n[F^n(e)]\gamma_2^n$. Since $C(F)$ is the bounded cancellation constant of $F$, by Lemma~\ref{Lem Bounded cancellation lemma}, for every $n \in \NN$, the subpath of $[F^n(\gamma)]$ contained in $[F^n(L_g)]$ is obtained from $[F^n(\gamma)]$ by erasing an initial and a terminal segment of length at most $2^nC(F)$. In particular, for every $n \geq M$, the path $[F^{n-M}(e)]$ is contained in $[F^n(L_g)]$. 

Recall that $T$ is a train track representative of a fully irreducible outer automorphism, so that $T$ has a unique stratum and its transition matrix is primitive. In particular, for every edge $e' \in ET$ and every $n \geq 0$, there exists $h_n \in G$ such that the path $h_n[F^n(e')]$ is contained in some reduced iterate of $e$ by $F$. Thus, the edge $e$ satisfies the conclusion of Lemma~\ref{lem:genericline}$(3)$, that is, there exists a generic line $\ell_\Lambda$ of $\Lambda$ such that any finite reduced edge path in $\ell_\Lambda$ is contained in the $G$-orbit of some reduced iterate of $e$ by $F$. As $[F^{n-M}(e)]$ is a subpath of $[F^n(L_g)]$ for every $n \geq M$, every finite subpath of $\ell_\Lambda$ is contained in the $G$-orbit of some reduced iterate of $L_g$ by $F$. Thus, for every open neighbourhood $V$ of $\ell_\Lambda$, there exist $n \geq 1$ and $h \in G$ such that $h[F^n(L_g)] \in V$. This shows that the closure $\overline{G\{[F^n(L_g)]\}_{n \in \NN}}$ contains $\Lambda=\overline{G \ell_\Lambda}$. 

Since $\phi(\mathcal{F}')=\mathcal{F}'$ and $g$ is $(G,\mathcal{F}')$-peripheral, for every $n \in \NN$, the line $[F^n(L_g)]$ represents an element in $\partial^2(\mathcal{F}',\mathcal{G})$. Since $\partial^2(\mathcal{F}',\mathcal{G})$ is $G$-invariant and closed, we have $$\Lambda \subseteq  \overline{G\{[F^n(L_g)]\}_{n \in \NN}} \subseteq \partial^2(\mathcal{F}',\mathcal{G}).$$ This contradicts Lemma~\ref{lem:attractinglaminationassociatedmaxfree}. Thus, every $(G,\mathcal{F}')$-peripheral element is elliptic in $T_\mathcal{F}$.

We now prove the second assertion. Let $\mathcal{F},\mathcal{F}' \in \rho_\phi^{-1}(\Lambda)$. We assume that $\mathcal{F}\neq \mathcal{F}'$ as otherwise the statement is immediate. The sets $\mathcal{A}_\mathcal{F}$ and $\mathcal{A}_{\mathcal{F}'}$ consist of conjugacy classes of malnormal subgroups with trivial pairwise intersection as they are the set of point stabilisers of an action of $G$ on an $\RR$-tree with trivial arc stabilisers (see~Theorem~\ref{theo:limittreeiwip}). Hence, by malnormality of the groups in $\mathcal{A}_\mathcal{F}$ and $\mathcal{A}_{\mathcal{F}'}$, it suffices to prove that, for every $[A] \in \mathcal{A}_\mathcal{F}$, the group $A$ is elliptic in $T_{\mathcal{F}'}$. A symmetric argument will then show that for every $[A'] \in \mathcal{A}_{\mathcal{F}'}$, the group $A'$ is elliptic in $T_{\mathcal{F}}$. This will conclude the proof of the second assertion. 

By Theorem~\ref{theo:limittreeiwip}, there exists a (possibly trivial) cyclic subgroup $H$ of $G$ such that $\mathcal{A}_\mathcal{F}=\mathcal{F} \cup \{[H]\}$. By the first assertion, for every $[A] \in \mathcal{F}$, the group $A$ is elliptic in $T_{\mathcal{F}'}$. 

We now prove that $H$ is elliptic in $T_{\mathcal{F}'}$. Note that, by the first assertion, every subgroup whose conjugacy class is in $\mathcal{F}'$ is elliptic in $T_\mathcal{F}$. Since $\mathcal{F} \neq\mathcal{F'}$, and since $\mathcal{F}'$ is maximal, we have $\mathcal{F}' \nsqsubseteq \mathcal{F}$. Since $\mathcal{A}_{\mathcal{F}}=\mathcal{F} \cup \{[H]\}$, we see that $H$ is $(G,\mathcal{F}')$-peripheral. Thus, by Theorem~\ref{theo:limittreeiwip}, $H$ is elliptic in $T_{\mathcal{F}'}$. This proves the second assertion.

We finally prove the third assertion. Let $\mathcal{F} \in \rho_{\phi}^{-1}(\Lambda)$. By the second assertion, the set $\mathcal{A}_\mathcal{F}$ is an invariant of $\rho_{\phi}^{-1}(\Lambda)$, so that $\phi$ preserves $\mathcal{A}_\mathcal{F}$. 

By Theorem~\ref{theo:limittreeiwip}, there exists a (possibly trivial) cyclic subgroup $H$ of $G$ such that $\mathcal{A}_\mathcal{F}=\mathcal{F} \cup \{[H]\}$. By Theorem~\ref{thm:invariantfamilycyclicfixed}, since $H$ is cyclic and $\phi \in \IA(G,\mathcal{G},3)$, we see that $\phi([H])=H$, so that $\phi(\mathcal{F})=\mathcal{F}$. By Lemma~\ref{lem:periodic freefactorsystem}, $\phi$ fixes elementwise $\mathcal{F}$. By the second assertion and Theorem~\ref{theo:limittreeiwip}, every element of $\rho_{\phi}^{-1}(\Lambda)$ is obtained by taking a finite subset of $\mathcal{A}_\mathcal{F}$. Therefore, we see that $\phi$ fixes elementwise $\rho_{\phi}^{-1}(\Lambda)$.
\end{proof}

We now prove the main technical result of this section.

\begin{prop}\label{prop:periodiclaminationfixed}
Let $G$ be a group satisfying the \hyperlink{SA}{Standing Assumptions}. Let $\phi \in \IA(G,\mathcal{G},3)$ and let $\Lambda \in \mathcal{L}(\phi)$. Then $\Lambda$ is fixed by $\phi$.
\end{prop}

\begin{proof}
The proof follows~\cite[Lemma~II.4.2]{HandelMosher20}. Let $n$ be the size of the orbit of $\Lambda$ under iteration of $\phi$. Let $F \colon T \to T$ be a relative train track representative of $\phi$. By Proposition~\ref{prop:attractinglaminations}, there exists an EG stratum $H_r$ with the following properties. 

\begin{itemize}
\item The top stratum crossed by every generic line of $\Lambda$ is $H_r$.
\item If $H_r$ is aperiodic, then $n=1$.
\item If $H_r$ is not aperiodic, there exists a $G$-invariant partition $$EH_r=P_{1}^{(r)} \amalg \ldots P_{n_r}^{(r)}$$ of the edges of $H_r$ such that, for every $i \in \{1,\ldots,n_r\}$ and every edge $e \in P_i^{(r)}$ the path $[F(e)]$ is contained in $T_{r-1} \cup P_{i+1}^{(r)}$, where indices are taken modulo $n_r$.
\end{itemize}

Suppose towards a contradiction that $n_r \geq 2$. Note that, since $F(T_{r-1})=T_{r-1}$, we have $\phi(\mathcal{F}(T_{r-1}))=\mathcal{F}(T_{r-1})$. By Lemma~\ref{lem:periodic freefactorsystem}, $\phi$ fixes every element of $\mathcal{F}(T_{r-1})$, so that $F$ preserves every connected component of $\overline{G \backslash T_{r-1}}$. Let $T'$ be the tree obtained from $T$ by collapsing every connected component of $\overline{G \backslash T_{r-1}}$ to a point. Then the filtration of $T$ induces a filtration of $T'$. Let $\mathbb{X}_r$ be the underlying graph of $G \backslash T_r'$ and let $\mathbb{X}_{r-1}$ be the underlying graph of $G \backslash T_{r-1}'$. Then $\mathbb{X}_{r-1}$ is a set of vertices. Moreover, $F$ induces a homotopy equivalence $\overline{F} \colon \mathbb{X}_r \to \mathbb{X}_r$ which fixes every vertex of $\mathbb{X}_{r-1}$. Since $\phi$ acts trivially on $H_1(\overline{G \backslash T},\ZZ/3\ZZ)$ by Lemma~\ref{lem:actionhomograph}, it also acts trivially on $H_1(\mathbb{X}_r,\ZZ/3\ZZ)$.

The map $F$ permutes the distinct graphs $\{T_{r-1} \cup P_{i}^{(r)}\}_{i \in \{1,\ldots,n_r\}}$. Let $i,j \in \{1,\ldots,n_r\}$ be distinct. By Lemma~\ref{lem:periodicEGstratum}, we have $$\mathcal{F}(T_{r-1} \cup P_i^{(r)}) \neq \mathcal{F}(T_{r-1} \cup P_j^{(r)}) \text { and } EP_{i}^{(r)} \cap EP_{j}^{(r)}=\varnothing.$$ Let $\mathbb{X}^i$ and $\mathbb{X}^j$ be the underlying subgraphs of $\mathbb{X}_r$ induced by $T_{r-1} \cup P_i^{(r)}$ and $T_{r-1} \cup P_j^{(r)}$. Since $EP_{i}^{(r)} \cap EP_{j}^{(r)}=\varnothing$ and since the partition of $EH_r$ is $G$-invariant, the intersection $\mathbb{X}^i \cap \mathbb{X}^j=\mathbb{X}_{r-1}$ consists only of vertices.

\medskip

\noindent{\bf Claim. } Let $i \in \{1,\ldots,n_r\}$. One of the following hold.

\begin{enumerate}
    \item The graph $\mathbb{X}^i$ contains an embedded circle.
    \item There exist two distinct vertices $v,w \in \mathbb{X}^i$ corresponding to two distinct connected components of $\overline{G \backslash T_{r-1}}$ such that $v$ and $w$ are joined by an arc in $\mathbb{X}^i$.
\end{enumerate}

\begin{proof}
Assume that $\mathbb{X}^i$ does not contain an embedded circle. Let $\Lambda_i$ be the attracting lamination of $\phi$ associated with $T_{r-1} \cup P_{i}^{(r)}$ and let $\ell_i$ be a generic line for $\Lambda_i$. By Proposition~\ref{prop:attractinglaminations}, the line $\ell_i$ crosses $T_{r-1} \cup P_i^{(r)}$, so that $\ell_i$ projects to an immersed line $\overline{\ell}_i$ in $\mathbb{X}^i$ which contains edges of $\mathbb{X}^i$. Since $\overline{\ell}_i$ is connected and immersed, and since there is no embedded circle in $\mathbb{X}^i$, some subpath of $\overline{\ell}_i$ is an arc in $\mathbb{X}^i$ between two vertices of $\mathbb{X}^i$ corresponding to two distinct connected components of $\overline{G \backslash T_{r-1}}$.
\end{proof}

We now distinguish between the two cases of the claim. Suppose first that there exists an embedded circle $\gamma$ in $\mathbb{X}^i$. Then $\gamma$ induces a nontrivial element $[\gamma]\in H_1(\mathbb{X}_r,\ZZ/3\ZZ)$. Moreover, since $\gamma \nsubseteq \mathbb{X}_{r-1}$ and since $\mathbb{X}^i \cap \mathbb{X}^j=\mathbb{X}_{r-1}$, any embedded circle in $\mathbb{X}^j$ will induce an element of $H_1(\mathbb{X}_r,\ZZ/3\ZZ)$ distinct from $[\gamma]$. Since $\phi \in \IA(G,\mathcal{G},3)$ acts trivially on $H_1(\mathbb{X}_r,\ZZ/3\ZZ)$ by Lemma~\ref{lem:actionhomograph}, we see that $\overline{F}$ preserves $\mathbb{X}^i$ and that $F$ preserves $P_i^{(r)}$, a contradiction.

Suppose now that $\mathbb{X}^i$ has no embedded circle and that there exist two distinct vertices $v,w \in V\mathbb{X}^i$ corresponding to two distinct connected components of $\overline{G \backslash T_{r-1}}$ such that $v$ and $w$ are joined by an arc $\gamma_i$ in $\mathbb{X}^i$. This arc is unique as $\mathbb{X}^i$ does not contain any embedded circle. For every $m >0$, the graph $\overline{F}^m(\mathbb{X}^i)=\mathbb{X}^{i+m}$ has a unique arc between $v$ and $w$ since $\overline{F}$ fixes both $v$ and $w$ (here and thereafter, $i+m$ is taken modulo $n_r$). 

For $m \geq 0$, let $\gamma_{i+m}$ be the arc in $\mathbb{X}^{i+m}$ between $v$ and $w$. Then $\gamma=\{\gamma_{i+m}\}_{m \in \NN}$ is an $\overline{F}$-invariant subgraph of $\mathbb{X}_r$ with exactly two vertices, $v$ and $w$, and $n_r$ arcs between them. Since $n_r \geq 2$, the graph $\gamma$ has no leaf. As $\phi \in \IA(G,\mathcal{G},3)$ it acts trivially on $H_1(\mathbb{X}_r,\ZZ/3\ZZ)$. As $\gamma$ is an $\overline{F}$-invariant subgraph of $\mathbb{X}_r$, we see that $\phi$ fixes pointwise $H_1(\gamma,\ZZ/3\ZZ)$. 

By Lemma~\ref{lem:actionhomograph}, either $\overline{F}$ is homotopy equivalent to the identity map on $\gamma$, or else $n_r=2$ and $\overline{F}$ is homotopy equivalent to a rotation. As $\overline{F}$ fixes $v$ and $w$, in both cases, we see that $\overline{F}$ is homotopy equivalent to the identity map on $\gamma$. This shows that, for any distinct $i,j \in \{1,\ldots,n_r\}$, the graph $\mathbb{X}^i$ cannot be sent to $\mathbb{X}^j$ by $\overline{F}$, a contradiction. 

Thus, $H_r$ is aperiodic and $\phi$ acts trivially on $\mathcal{L}(\phi)$.
\end{proof}

\begin{proof}[Proof of Proposition~\ref{prop:maxperiodffsfixed}]
Let $\phi \in \IA(G,\mathcal{G},3)$ and let $\mathcal{F}$ be a maximal $\phi$-periodic $(G,\mathcal{G})$-free factor system which is nonposradic. By Lemma~\ref{lem:attractinglaminationassociatedmaxfree}, there exists a unique lamination $\Lambda_{\mathcal{F}}$ associated with $\mathcal{F}$. Moreover, by Lemma~\ref{lem:preimagerho}$(3)$, $\mathcal{F}$ is fixed by $\phi$ if and only if $\Lambda_{\mathcal{F}}$ is fixed by $\phi$. The result now follows from Proposition~\ref{prop:periodiclaminationfixed}.
\end{proof}

We now have all the necessary tools to prove Theorem~\ref{thm:periodicfreefactorfixed}.

\begin{proof}[Proof of Theorem~\ref{thm:periodicfreefactorfixed}.] Let $\phi \in \IA(G,\mathcal{G},3)$ and let $\mathcal{F}$ be a $\phi$-periodic proper $(G,\mathcal{G})$-free factor system. We prove by induction on $\xi(G,\mathcal{G})$ that $\mathcal{F}$ is fixed by $\phi$. The case $\xi(G,\mathcal{G})=1$ is immediate since in this case either $\mathcal{F}=\{[G]\}$ or else $G=\ZZ$ and $\phi=\mathrm{Id}$.

Suppose that $\xi(G,\mathcal{G}) \geq 2$. Let $\mathcal{F}'$ be a maximal proper $\phi$-periodic $(G,\mathcal{G})$-free factor system such that $\mathcal{F}\sqsubseteq \mathcal{F}'$. We claim that $\mathcal{F}'$ is fixed by $\phi$. Indeed, if $\mathcal{F}'$ is nonsporadic, this follows from Proposition~\ref{prop:maxperiodffsfixed}. If $\mathcal{F}'$ is sporadic, it induces a canonical $(G,\mathcal{G})$-free splitting $T_{\mathcal{F}'}$. By Theorem~\ref{thm:periodicfreesplittingfixed}, the free splitting $T_{\mathcal{F}'}$ is fixed by $\phi$. Hence $\mathcal{F}'$ is fixed by $\phi$.

In both cases, the $(G,\mathcal{G})$-free factor system $\mathcal{F}'$ is fixed by $\phi$. Let $[A] \in \mathcal{F}'$. By Lemma~\ref{lem:periodic freefactorsystem}, the conjugacy class $[A]$ is fixed by $\phi$. Moreover, since $A$ is a $(G,\mathcal{G})$-free factor, by Lemma~\ref{lem:IA_passes_free_factor}, the outer automorphism $\phi_{|A}$ of $\Out(A)$ induced by $\phi$ is contained in $\IA(A,\mathcal{G}_{|A},3)$. Let $\mathcal{F}_{|A}$ be the $(A,\mathcal{G}_{|A})$-free factor system induced by $\mathcal{F}$. By induction, we see that $\mathcal{F}_{|A}$ is fixed by $\phi_{|A}$. Since $\mathcal{F} \sqsubseteq \mathcal{F}'$, for every $[H] \in \mathcal{F}$, there exist $[A] \in \mathcal{F'}$ and $g \in G$ such that $gHg^{-1} \subseteq A$ and $[gHg^{-1}]_A \in \mathcal{F}_{|A}$. Thus there exists $\Phi_{|A} \in \phi_{|A}$ which preserves $gHg^{-1}$. As $\Phi_A$ is the restriction to $A$ of some $\Phi \in \phi$, the automorphism $\Phi$ preserves $gHg^{-1}$. Therefore, we see that $\mathcal{F}$ is fixed by $\phi$.
\end{proof} 

\section{Aperiodic finite index subgroups of outer automorphism groups of toral relatively hyperbolic groups}\label{Section:Toral}

Let $H$ be a \emph{toral relatively hyperbolic group}, that is a torsion free hyperbolic group relative to a finite set $\mathcal{P}$ of finitely generated abelian groups. In this section, we show that $H$ satisfies the \hyperlink{SA}{Standing Assumptions} (see in particular Corollary~\ref{coro:IAtoralrelativelyhyperbolicgroup}). Note that the hypotheses are only related to the one-ended factors of the Grushko decomposition of $H$, so we may assume that $H$ is one-ended. We split the proof into two sections, whether $H$ is free abelian or not. 

\subsection{Finitely generated free abelian groups}

Let $n >0$. Consider the group $\mathrm{GL}_n(\ZZ)=\Aut(\ZZ^n)=\Out(\ZZ^n)$. Let $\IA(\ZZ^n,3)$ be the kernel of the natural homomorphism $$\mathrm{GL}_n(\ZZ) \to \mathrm{GL}_n(\ZZ/3\ZZ),$$ that is, $\IA(\ZZ^n,3)$ is the level $3$ principal congruence subgroup of $\mathrm{GL}_n(\ZZ)$. The group $\IA(\ZZ^n,3)$ is torsion free. Recall that the periodic subgroup $\mathrm{Per}(\Phi)$ of $\Phi$ is the subgroup of $\ZZ^n$ consisting of all $\Phi$-periodic elements.

\begin{prop}\label{prop:abeliancase}
Suppose that $\Phi \in \IA(\ZZ^n,3)$. Then every element of the periodic subgroup of $\Phi$ is fixed by $\Phi$. 
\end{prop} 

\begin{proof}
 Note that, for every $m>0$ and every $v \in \ZZ^n$, the element $v$ has at most one $m$-th root. This implies that the subgroup $\mathrm{Per}(\Phi)$ is root--closed, so a direct summand of $\ZZ^n$. 

Therefore, the image of $\Phi$ in $\Aut(\mathrm{Per}(\Phi))$ given by restriction is a finite order automorphism contained in $\IA(\mathrm{Per}(\Phi),3)$. Since $\IA(\mathrm{Per}(\Phi),3)$ is torsion free, the image of $\Phi$ in $\Aut(\mathrm{Per}(\Phi))$ is trivial. Therefore, $\Phi$ acts trivially on $\mathrm{Per}(\Phi)$ and every element of the periodic subgroup of $\Phi$ is fixed by $\Phi$.
\end{proof}

\subsection{One-ended toral relatively hyperbolic groups}

Let $H$ be a nonelementary one-ended \emph{toral relatively hyperbolic group} and let $\mathcal{P}$ be its peripheral structure. Then $\mathcal{P}$ consists of finitely generated free abelian groups. Note that, in this case, elementary subgroups of $H$ are free abelian. The aim of this section is to construct a finite index subgroup of $\Out(H)$ satisfying the \hyperlink{SA}{Standing Assumptions}. First, we need some general lemmas regarding vertex stabilisers of an elementary splitting. It is similar to~\cite[Lemma~3.8]{GuirardelLevitt2015}, but the later is only written for JSJ elementary splittings.

\begin{lem}\label{lem:vertexperipheral}
Let $H$ be a toral one-ended relatively hyperbolic group, let $\mathcal{P}$ be its peripheral structure and let $S$ be an $(H,\mathcal{P})$-elementary splitting. For every $v\in VS$, the group $H_v$ is toral relatively hyperbolic. 

Let $\mathcal{P}_v$ be the set of maximal parabolic subgroups contained in $H_v$ and let $\mathrm{Inc}_v$ be the set of incident edge stabilisers. A peripheral structure of $H_v$ is $\mathcal{P}_v \cup \mathrm{Inc}_v$.
\end{lem}

\begin{proof}
Since edge stabilisers are elementary, by~\cite[Proposition~3.4]{GuirardelLevitt2015}, vertex stabilisers are relatively quasi-convex, hence they are toral relatively hyperbolic. 

We now explicit the peripheral structure of $H_v$. By~\cite[Theorem~9.1]{Hruska10}, a peripheral structure of $H_v$ is given by the groups $H_v \cap gPg^{-1}$, with $g \in H$ and $[P]\in \mathcal{P}$. Since $gPg^{-1}$ is elliptic in $S$ and since edge stabilisers are elementary, either $gPg^{-1} \subseteq H_v$, or else $gPg^{-1} \cap H_v$ is an edge stabiliser. Conversely, any edge stabiliser in $\mathrm{Inc}_v$ which is not equal to some $H_v \cap gPg^{-1}$, with $g \in H$ and $[P]\in \mathcal{P}$ is infinite cyclic. Such groups may be added to the peripheral structure.
\end{proof}

By Theorem~\ref{thm:JSJrelhyp}, there exists a canonical $(H,\mathcal{P})$-elementary splitting $T_\mathcal{P}$ preserved by $\Aut(H)$ and $\Out(H)$.

Let $\Out^0(H)$ be the finite index subgroup of $\Out(H)$ acting trivially on the quotient graph $\overline{H \backslash T_{\mathcal{P}}}$ and fixing the conjugacy class of every group in $\mathcal{P}$. We have natural restriction homomorphisms $$p_H^V \colon \Out^0(H) \to \prod_{v \in VT_\mathcal{P}} \Out(H_v)$$ and 
$$p_H^E \colon \Out^0(H) \to \prod_{e \in ET_\mathcal{P}} \Out(H_e).$$

By Theorem~\ref{thm:JSJrelhyp}$(7)$, the image of $p_H^E$ is finite. Let $\Out^1(H)=\ker(p_H^E)$. The next lemma describes the kernel of $p_H^V$, and applies more generally for elementary splittings of toral relatively hyperbolic group.

\begin{lem}\label{lem:Kernelabelian}
Let $H$ be a nonelementary one-ended toral relatively hyperbolic group, let $\mathcal{P}$ be its peripheral structure and let $S$ be an $(H,\mathcal{P})$-elementary splitting with the following properties. 
\begin{enumerate}
    \item For every edge $e \in ES$, there exists an endpoint $v$ of $e$ such that $H_e$ is malnormal in $H_v$.
    \item For every edge $e \in ES$, there exists an endpoint $v$ of $e$ such that $H_v$ is abelian.
\end{enumerate}
 The kernel $\mathcal{T}(S)$ of the natural homomorphism $$\mathcal{K}(S) \to \prod_{v \in VS} \Out(H_v),$$ is a finitely generated free abelian group.
\end{lem}

\begin{proof}
We use the definitions and notations from~\cite{GuirardelLevitt2015}. Since every edge group is malnormal in the stabiliser of one of its endpoints, by~\cite[Lemma~2.12~(2)]{GuirardelLevitt2015}, the group $\mathcal{T}(S)$ is the \emph{group of twists of $S$} (see~\cite{levitt2005}).

The rest of the proof follows~\cite[Proof of Corollary~4.4]{GuirardelLevitt2015}. Since $H$ is nonelementary, its centre is trivial (see for instance~\cite{Farb1994}). Thus, by \cite[Proposition~3.1]{levitt2005}, the group $\mathcal{T}(S)$ is isomorphic to a quotient of $\prod_{e \in \vec{E}(\overline{H \backslash S})} C_{H_{o(e)}} (H_e)$ (where $\vec{E}(\overline{H \backslash S})$ is the set of oriented edges of $\overline{H \backslash S}$) by the image of the natural diagonal map 
\[ j\colon \prod_{v \in V(\overline{H \backslash S})} Z(H_v) \times \prod_{e \in E(\overline{H \backslash S})} Z(H_e) \to \prod_{e \in \vec{E}(\overline{H \backslash S})} C_{H_{o(e)}} (H_e).\] Let $v \in  V(\overline{H \backslash S})$ and let $e$ be an edge of $\overline{H \backslash S}$ adjacent to $v$. Suppose that $H_v$ is nonelementary. Then $H_v$ is not abelian and $H_e$ is malnormal in $H_v$ by Lemma~\ref{lem:vertexperipheral}. Moreover, the centre of $H_v$ is trivial. Thus, we have $C_{H_v}(H_e)=H_e$ and $Z(H_v)=\{1\}$. 

Hence, for every $e \in E(\overline{H \backslash S})$, the edge relation induced by $j(\prod_{e \in E(\overline{H \backslash S})} Z(H_e))$ identifies the two copies of $H_e$ coming from a chosen orientation of $e$. Moreover, the group $j(\prod_{v \in V(\overline{H \backslash S})} Z(H_v))$ is equal to $j(\prod_{v \in V(\overline{H \backslash S}), \; H_v \text{ abelian }} Z(H_v))$.

This shows that $\mathcal{T}(S)$ is isomorphic to the direct product 
\[\prod_{v \in V(\overline{H \backslash S}), \; H_v \text{ abelian }} \left( \prod_{e \in \mathrm{Inc}_v} C_{H_v}(H_e) \right)/Z(H_v), \] where $Z(H_v)$ embeds diagonally. Let $v \in V(\overline{H \backslash S})$ with $H_v$ abelian. Then, for every $e \in \mathrm{Inc}_v$, we have $C_{H_v}(H_e)=Z(H_v)=H_v$, so that the group $\left( \prod_{e \in \mathrm{Inc}_v} C_{H_v}(H_e) \right)/Z(H_v)$ is free abelian.
\end{proof}

Let $v \in VT_{\mathcal{P}}$. We now construct a finite index subgroup $\IA(H_v,3)$ of the image of $\Out^1(H)$ in $\Out(H_v)$. 

\bigskip

\noindent {\it Suppose that $v$ is rigid. }

\medskip

By Theorem~\ref{thm:JSJrelhyp}$(8)$, the image of $\Out^1(H)$ is finite. We set $\IA(H_v,3)=\{\mathrm{id}\}$.

\bigskip

\noindent {\it Suppose that $v$ is elementary. }

\medskip

In this case, there exists $n \geq 1$ such that the group $\Out(H_v)$ is isomorphic to $\mathrm{GL}_n(\ZZ)$. Let $\IA(H_v,3)$ be the intersection of the image of $\Out^1(H)$ with the level $3$ principal congruence subgroup of $\Out(H_v)$.

\bigskip

\noindent {\it Suppose that $v$ is flexible. }

\medskip

By Theorem~\ref{thm:JSJrelhyp}$(5)$, there exists a compact hyperbolic surface $\Sigma_v$ such that the group $H_v$ is isomorphic to $\pi_1(\Sigma_v)$. Moreover, the image of $\Out^1(H)$ in $\Out(H_v)$ is isomorphic to $\mathrm{MCG}(\Sigma_v)$, the mapping class group of $\Sigma_v$. By a result of Ivanov~\cite[Theorem~1.7]{Ivanov92}, there exists a finite index subgroup $\IA(H_v,3)$ of $\mathrm{MCG}(\Sigma_v)$ with the following properties.

\begin{enumerate}
\item The group $\IA(H_v,3)$ is torsion free.
\item For every $\phi \in \IA(H_v,3)$ and every periodic subgroup $\mathrm{Per}(\phi)$ of $\phi$, we have $\phi \in \Out(H_v,[\mathrm{Per}(\phi)]^{(t)})$.
\item Suppose that $K$ is the fundamental group of a subsurface of $\Sigma_v$. If the conjugacy class $[K]$ of $K$ is $\phi$-periodic, then $[K]$ is fixed by $\phi$.
\item If $\Phi \in \Aut(H_v)$ is an automorphism whose image is contained in $\IA(H_v,3)$, then every $\Phi$-periodic element of $H_v$ is fixed by $\Phi$.
\end{enumerate}

Let $\IA^1(H,3)$ be the preimage in $\Out^1(H)$ of $\prod_{v \in VT_{\mathcal{P}}} \IA(H_v,3)$. The group $\IA^1(H,3)$ is a finite index subgroup of $\Out^1(H)$, hence a finite index subgroup of $\Out(H)$. 

Let $\IA(H,3)$ be the finite index subgroup of $\IA^1(H,3)$ whose action on $H_1(H,\ZZ/3\ZZ)$ is trivial. We now prove that the group $\IA(H,3)$ satisfies the hypotheses of the \hyperlink{SA}{Standing Assumptions}.

\begin{lem}\label{lem:IAHtorsionfree}
The group $\IA(H,3)$ is torsion free.
\end{lem}

\begin{proof}
Let $F$ be a finite subgroup of $\IA(H,3)$. Since $p_H^V(\IA(H,3))$ is torsion free by construction, the group $F$ is contained in $\ker(p_H^V)$, which is a torsion free abelian group by Lemma~\ref{lem:Kernelabelian}. Hence $F=\{1\}$. 
\end{proof}

\begin{prop}\label{prop:IAHperiodicconjclass}
Let $\phi \in \IA(H,3)$ and let $H' \subseteq H$ be a $\phi$-periodic finitely generated subgroup. Then $\phi \in \Out(H,[H']^{(t)})$.
\end{prop}

\begin{proof}
Let $\mathcal{H}=\{\phi^\ell[H']\}_{\ell \geq 0}$. Since $[H']$ is $\phi$-periodic, the set $\mathcal{H}$ is finite. Let $T_{\mathcal{P} \cup \mathcal{H}}$ be the associated JSJ $(H,\mathcal{P} \cup \mathcal{H})$-elementary splitting given by Theorem~\ref{thm:JSJrelhyp}. Then $T_{\mathcal{P}\cup \mathcal{H}}$ is preserved by $\phi$. Moreover, the trees $T_\mathcal{P}$ and $T_{\mathcal{P} \cup \mathcal{H}}$ are compatible by~\cite[Theorem~9.18~(3)]{guirardel2016jsj}. 

Let $U$ be the minimal common refinement of $T_\mathcal{P}$ and $T_{\mathcal{P} \cup \mathcal{H}}$ (see~\cite[Proposition~A.26]{guirardel2016jsj}). The tree $U$ has the following property. Let $p_\mathcal{P} \colon U \to T_\mathcal{P}$ and let $p_{\mathcal{P}\cup \mathcal
{H}} \colon U \to T_{\mathcal{P}\cup \mathcal{H}}$ be the natural $H$-equivariant projections. Then there does not exist an edge $e$ of $U$ such that both $p_\mathcal{P}(e)$ and $p_{\mathcal{P}\cup \mathcal{H}}(e)$ are vertices.
By minimality, since $\phi$ preserves both $T_\mathcal{P}$ and $T_{\mathcal{P} \cup \mathcal{H}}$, we see that $\phi$ preserves $U$ and the projections $p_\mathcal{P}$ and $p_{\mathcal{P}\cup \mathcal
{H}}$ are $\phi$-equivariant.

\medskip

\noindent{\bf Claim. } The graph automorphism of $\overline{H \backslash T_{\mathcal{P} \cup \mathcal{H}}}$ induced by $\phi$ is trivial.

\medskip

\begin{proof}
As usual, we want to apply Lemma~\ref{lem:graphautotrivial}. We proceed in several steps.

\medskip

\noindent{\it Step 1. } Let $v$ be a vertex of $T_{\mathcal{P} \cup \mathcal{H}}$ with nonelementary stabiliser. Then $\phi$ preserves the $H$-orbit of $v$.

\begin{proofblack}
Since $H_v$ is nonelementary, the vertex $v$ is the unique point in $T_{\mathcal{P} \cup \mathcal{H}}$ fixed by $H_v$. Thus, $\phi$ preserves the $H$-orbit of $v$ if and only if $\phi$ preserves the $H$-conjugacy class of $H_v$. We will frequently use this equivalence during the proof.

Consider the preimage $p_{\mathcal{P}\cup \mathcal
{H}}^{-1}(v)$ in $U$. Suppose first that $p_{\mathcal{P}\cup \mathcal
{H}}^{-1}(v)$ contains an edge $\widetilde{e}$. By equivariance of $p_{\mathcal{P}\cup \mathcal
{H}}$, if $\phi$ preserves the $H$-orbit of $\widetilde{e}$, then $\phi$ preserves the $H$-orbit of $v$, so we focus on the $H$-orbit of $\widetilde{e}$. By minimality of $U$, denoting by $\mathring{\widetilde{e}}$ the interior of $\widetilde{e}$, the projection $p_\mathcal{P}$ restricts to a homeomorphism on $\mathring{\widetilde{e}}$. Since $\phi \in \Out^0(H)$, $\phi$ preserves the $H$-orbit of $p_\mathcal{P}(\mathring{\widetilde{e}})$. By equivariance of $p_\mathcal{P}$, $\phi$ preserves the $H$-orbit of $\widetilde{e}$ and hence the $H$-orbit of $v$.

Suppose now that $p_{\mathcal{P}\cup \mathcal
{H}}^{-1}(v)$ is a vertex $\widetilde{v}$. Then $H_{\widetilde{v}}=H_v$. Thus, $H_v$ fixes the vertex $w=p_\mathcal{P}(\widetilde{v})$ of $T_\mathcal{P}$. Since $H_v$ is nonelementary, so is $w$. We distinguish between two cases, according to the nature of $w$.

\bigskip

\noindent{\it Case 1. } The vertex $w$ is rigid.

\medskip

In this case, since $\phi \in \IA(H,3)$, the outer automorphism $\phi_w \in \Out(H_w)$ is trivial. Thus, $\phi$ acts as a global conjugation on $H_v \subseteq H_w$. Therefore, $\phi$ preserves the $H$-orbit of $v$.

\bigskip

\noindent{\it Case 2. } The vertex $w$ is flexible.

\medskip

Then $H_w$ is the fondamental group of a compact hyperbolic surface $\Sigma_w$. Moreover, since elementary subgroups of $H_w$ are cyclic, the minimal subtree $U_{H_w} \subseteq p_\mathcal{P}^{-1}(w)$ of $H_w$ in $U$ is an $H_w$-cyclic splitting. By~\cite[Proposition~5.4]{guirardel2016jsj}, the tree $U_{H_w}$ is dual to a decomposition of $\Sigma_w$ along closed geodesics. Since $\phi_w \in \IA(H_w,3)$, this decomposition is preserved by $\phi_w$. Hence $\phi_w$ acts trivially on the graph $\overline{H_w \backslash U_{H_w}}$. As $H_v \subseteq H_w$, the tree $U_{H_w}$ contains the unique vertex $\widetilde{v}$ of $U$ fixed by $H_v$. Hence $\phi_w$ preserves the $H_w$-orbit of $\widetilde{v}$, so that $\phi$ preserves the $H$-orbit of $\widetilde{v}$. Since $H_{\widetilde{v}}=H_v$, $\phi$ preserves the $H$-conjugacy class of $H_v$, hence the $H$-orbit of $v$. As we have ruled out every case, this concludes the proof of Step~1.
\end{proofblack}

\noindent{\it Step~2. } The outer class $\phi$ fixes every leaf of $\overline{H \backslash T_{\mathcal{P \cup \mathcal{H}}}}$.

\begin{proofblack}
Let $v \in VT_{\mathcal{P \cup \mathcal{H}}}$ which projects to a leaf of $\overline{H \backslash T_{\mathcal{P \cup \mathcal{H}}}}$. By Step~1, if $H_v$ has nonelementary stabiliser, then the $H$-orbit of $v$ is fixed by $\phi$. 

So we may suppose that $H_v$ is elementary. Let $e \in ET_{\mathcal{P \cup \mathcal{H}}}$ be an edge adjacent to $v$. Since $v$ projects to a leaf of $\overline{H \backslash T_{\mathcal{P \cup \mathcal{H}}}}$, by minimality of the action of $H$ on $T_{\mathcal{P \cup \mathcal{H}}}$, we have $H_e \subsetneq H_v$. Hence $H_v$ fixes a unique vertex of $T_{\mathcal{P \cup \mathcal{H}}}$, so it suffices to prove that $\phi$ preserves the $H$-conjugacy class of $H_v$.

Suppose that $[H_v] \in \mathcal{P}$. Since $\phi \in \Out^0(H)$, we see that $\phi$ preserves $[H_v]$. So the only case remaining is when $H_v$ is infinite cyclic. As $H$ is one-ended, the group $H_e$ has finite index in $H_v$. Since $T_\mathcal{P}$ and $T_{\mathcal{P} \cup \mathcal{H}}$ are compatible, the group $H_e$ fixes a point in $T_\mathcal{P}$. Since $H_v$ is infinite cyclic, the group $H_v$ also fixes a point $w$ in $T_\mathcal{P}$. Hence $H_v$ is a periodic subgroup of $\phi_w \in \IA(H_w,3)$. By definition of $\IA(H_w,3)$, $\phi$ fixes the $H$-conjugacy class of $H_v$. This concludes the proof of Step~2.
\end{proofblack}

By Step~2, $\phi$ fixes every leaf of $\overline{H \backslash T_{\mathcal{P \cup \mathcal{H}}}}$. Since $\phi \in \IA(H,3)$, $\phi$ acts trivially on $H_1(H,\ZZ/3\ZZ)$. By Lemma~\ref{lem:actionhomograph}, $\phi$ acts trivially on $H_1(\overline{H \backslash T_{\mathcal{P \cup \mathcal{H}}}},\ZZ/3\ZZ)$. By Lemma~\ref{lem:graphautotrivial}, $\phi$ acts trivially on $\overline{H \backslash T_{\mathcal{P \cup \mathcal{H}}}}$ or else $\overline{H \backslash T_{\mathcal{P \cup \mathcal{H}}}}$ is a circle and $\phi$ acts as a rotation on it. Therefore, in order to show that $\phi$ acts trivially on $\overline{H \backslash T_{\mathcal{P \cup \mathcal{H}}}}$, it suffices to prove that $\phi$ fixes a point in $\overline{H \backslash T_{\mathcal{P \cup \mathcal{H}}}}$. 

Since $H$ is nonelementary, the tree $T_{\mathcal{P \cup \mathcal{H}}}$ contains a vertex with nonelementary stabiliser. By Step~1, $\phi$ fixes the $H$-orbit of this vertex. Thus, $\phi$ fixes a point in $\overline{H \backslash T_{\mathcal{P \cup \mathcal{H}}}}$. This concludes the proof of the claim.
\end{proof}

\bigskip

By the claim, the outer automorphism $\phi$ acts trivially on $\overline{H \backslash T_{\mathcal{P}\cup \mathcal{H}}}$. Since $\phi \in \IA(H,3)$, it also acts trivially on $\overline{H \backslash T_{\mathcal{P}}}$. Moreover, since every edge of $U$ has nontrivial image in either $T_\mathcal{P}$ or $T_{\mathcal{P} \cup \mathcal{H}}$, $\phi$ also acts trivially on the graph $\overline{H \backslash U}$.

Let $[K] \in \mathcal{H}$. Then $K$ fixes a point in $T_{\mathcal{P} \cup \mathcal{H}}$. Since the action of $K$ on $T_{\mathcal{P} \cup \mathcal{H}}$ is acylindrical (Theorem~\ref{thm:JSJrelhyp}$(2)$), its fixed point set is bounded. Let $v_K$ be the centre of the fixed point set of $K$. Let $\ell >0$ and let $\Phi \in \phi^\ell$ be such that $\Phi(K)=K$ and $\Phi_{|K}=\id_K$. By uniqueness of $v_K$, the automorphism $\Phi$ fixes $v_K$. Thus, $K$ is also a periodic subgroup of $\phi_{v_K} \in \Out(H_{v_K})$. This shows that, in order to prove Proposition~\ref{prop:IAHperiodicconjclass}, it remains to prove that, for every $v \in VT_{\mathcal{P}\cup \mathcal{H}}$, and every finitely generated $\phi_v$-periodic subgroup $H_v'$ of $H_v$ such that $[H_v'] \in \mathcal{H}$, we have $\phi_v \in \Out(H_v,[H_v']^{(t)})$. We therefore fix a vertex $v \in VT_{\mathcal{P}\cup \mathcal{H}}$ and a $\phi_v$-periodic finitely generated subgroup $H_v' \subseteq H_v$ such that $[H_v'] \in \mathcal{H}$. Once again, we distinguish between several cases, according to the nature of $v$.

\bigskip

\noindent{\it Case 1. } Suppose that $v$ is elementary noncyclic.

\medskip

Then $H_v'$ also fixes a point in $T_\mathcal{P}$. By Lemma~\ref{lem:elemornonelem}, there exists a canonical point $w$ in $T_\mathcal{P}$ fixed by $H_v'$, so that any automorphism in $\phi$ which preserves $H_v'$ also preserves $H_w$. 

Since $\phi \in \IA(H,3)$, the image of $\phi$ in $\Out(H_w)$ is contained in $\IA(H_w,3)$. In particular, for every $\phi_w$-periodic finitely generated subgroup $H_w'$ of $H_w$, we have $\phi_w \in \Out(H_w,[H_w']^{(t)})$. Thus, we have $\phi_w \in \Out(H_w,[H_v']^{(t)})$. Since $H_v$ is malnormal in $H$, any automorphism of $H_w$ fixing $H_v'$ pointwise also preserves $H_v$. Thus, we have $\phi_v \in \Out(H_v,[H_v']^{(t)})$.

\bigskip

\noindent{\it Case 2. } Suppose that $v$ is rigid.

\medskip

By Theorem~\ref{thm:JSJrelhyp}$(8)$, the outer automorphism $\phi_v$ of $\Out(H_v)$ induced by $\phi$ is of finite order. Let $U_{H_v} \subseteq p_{\mathcal{P}\cup \mathcal{H}}^{-1}(v)$ be the minimal subtree of $H_v$ in $U$ (if $H_v$ is elliptic, we assume that $U_{H_v}$ is the centre of the bounded fixed subtree of $H_v$). The $H_v$-splitting $U_{H_v}$ is $\phi_v$-invariant. 

We claim that $\phi_v$ acts trivially on $\overline{H_v \backslash U_{H_v}}$. Indeed, let $e \in EU_{H_v}$. We need to prove that $\phi_v(e)$ is in the $H_v$-orbit of $e$. Since $\phi$ acts trivially on $\overline{H \backslash U}$, there exists $h \in H$ such that $\phi_v(e)=he$. But $\phi_v$ preserves $U_{H_v} \subseteq p_{\mathcal{P}\cup \mathcal{H}}^{-1}(v)$, so that $e \in Ep_{\mathcal{P}\cup \mathcal{H}}^{-1}(v)$ and $he \in Ep_{\mathcal{P}\cup \mathcal{H}}^{-1}(v)$. By $H$-equivariant of $p_{\mathcal{P} \cup \mathcal{H}}$, we have $hv=v$ and $h \in H_v$. Hence $\phi_v$ preserves the $H_v$-orbit of $e$ and $\phi_v$ acts trivially on $\overline{H_v \backslash U_{H_v}}$. This proves the claim.

The above claim implies that we have a restriction homomorphism $$\langle \phi_v \rangle \to \prod_{w \in VU_{H_v}} \Out(H_w).$$

We claim that the kernel of the above homomorphism is a free abelian group. Indeed, by Lemma~\ref{lem:vertexperipheral}, the group $H_v$ is toral relatively hyperbolic. Let $e \in EU_{H_v}$. By minimality of $U$, since $p_{\mathcal{P} \cup \mathcal{H}}(e)$ is a point, the image of $e$ in $T_\mathcal{P}$ is an edge. By Lemma~\ref{lem:vertexperipheral}, the group $H_e$ is malnormal in the endpoint of $p_\mathcal{P}(e)$ with nonelementary stabiliser. Then $H_e$ is malnormal in the vertex stabiliser of the endpoint of $e$ which projects to it. Moreover, one of the endpoints of $e$ has abelian stabiliser since one of the endpoints of $p_\mathcal{P}(e)$ has abelian stabiliser. Thus, by Lemma~\ref{lem:Kernelabelian}, the kernel of the above homomorphism is free abelian. Since $\langle \phi_v \rangle$ is a finite group, the above homomorphism is injective. 

We now prove that the image of $\phi_v$ in $\prod_{w \in VU_{H_v}} \Out(H_w)$ is trivial. Let $w \in VU_{H_v}$. Then the group $H_w$ is contained in the vertex stabiliser $H_x$ of a vertex $p_\mathcal{P}(w)=x \in VT_\mathcal{P}$. Since $\phi \in \IA(H,3)$, for every $\phi_x$-periodic finitely generated subgroup $H_x'$ of $H_x$, we have $\phi_x \in \Out(H_x,[H_x']^{(t)})$. As the image of $\phi_v$ in $\Out(H_w)$ is of finite order, the group $H_w$ is a $\phi_x$-periodic subgroup. Since edge stabilisers in $T_\mathcal{P}$ and $T_{\mathcal{P} \cup \mathcal{H}}$ are finitely generated by Theorem~\ref{thm:JSJrelhyp}$(4)$, edge stabilisers in $U$ are finitely generated. Thus, the group $H_w$ is also finitely generated. Hence $\phi_x \in \Out(H_x,[H_w]^{(t)})$ and the image of $\phi_v$ in $\Out(H_w)$ is trivial.

Since the image of $\langle \phi_v \rangle $ in $\Out(H_w)$ is trivial for every $w \in VU_{H_v}$, the image of $$\langle \phi_v \rangle \to \prod_{w \in VU_{H_v}} \Out(H_w)$$ is trivial. This shows that $\phi_v=\mathrm{id}$.

\bigskip

\noindent{\it Case 3. } Suppose that $v$ is flexible.

\medskip

By Theorem~\ref{thm:JSJrelhyp}$(5)$, $v$ is a QH vertex and the group $H_v$ is the fondamental group of a compact hyperbolic surface $\Sigma_v$. Moreover, the subgroup $H_v'$ is conjugate into a boundary subgroup $C_{H_v'}$ of $\Sigma_v$. Then, it suffices to prove that $\phi_v$ preserves the conjugacy class of a generator of $C_{H_v'}$. We prove this fact using Lemma~\ref{lem:flexiblevertexacttrivially}. In order to apply it, it suffices to prove that for every $e \in ET_{\mathcal{P} \cup \mathcal{H}}$ adjacent to $v$, the outer automorphism $\phi_e$ is trivial.

Let $e \in ET_{\mathcal{P} \cup \mathcal{H}}$ be adjacent to $v$. Note that $H_e$ fixes a vertex $x$ of $T_\mathcal{P}$ since it fixes a vertex of $U$. Let $\Phi \in \phi$ be such that $\Phi(e)=e$. Then $\Phi$ fixes $x$ by $\Phi$-equivariance of $p_\mathcal{P}$. Since every $\Phi_{|H_x}$-periodic element of $H_x$ is fixed by $\Phi_{|H_x}$, a generator of $H_e$ is fixed by $\Phi$, so that $\Phi$ fixes $H_e$ elementwise. By Lemma~\ref{lem:flexiblevertexacttrivially}, $\phi_v$ preserves the $H_v$-conjugacy class of a generator of $C_{H_v'}$. 

\bigskip

\noindent{\it Case 4. } Suppose that $H_v$ is infinite cyclic.

\medskip

Let $e$ be an edge adjacent to $v$. Then $H_e$ and $H_v$ are commensurable. Thus, it suffices to prove that $\phi_e$ is trivial. Let $\Phi \in \phi$ be such that $\Phi(e)=e$, so that $\Phi(H_e)=H_e$. Let $w$ be the endpoint of $e$ distinct from $v$. Then $w$ has nonelementary stabiliser, $H_e \subseteq H_w$ and $\Phi(H_w)=H_w$. 

If $w$ is rigid, by Case $2$, $\Phi$ acts as a global conjugation on $H_w$. Since $H_e$ is malnormal in $H_w$ by Lemma~\ref{lem:vertexperipheral}, we see that $\Phi$ fixes elementwise $H_e$. 

If $w$ is flexible, then $H_w$ is contained in a boundary subgroup of the surface $\Sigma_w$ associated with $w$. By Case $3$, $\phi$ preserves the $H_w$-conjugacy class of a generator of every boundary component of $\Sigma_w$. In particular, $\Phi$ fixes a generator of $H_e$. This concludes the proof of the last case and the proof of the proposition. 
\end{proof}

\begin{prop}\label{prop:IAHautperiodicelementfixed}
Let $H$ be a one--ended toral relatively hyperbolic group. Let $\Phi \in \Aut(H)$ with $\phi=[\Phi] \in \IA(H,3)$. Every $\Phi$-periodic element of $H$ is fixed by $\Phi$.
\end{prop}

\begin{proof}
Recall that $\mathrm{Per}(\Phi)$ is the periodic subgroup of $\Phi$. Then $\mathrm{Per}(\Phi)$ is finitely generated by~\cite[Theorem~5.22]{AndGueHug2023}. Let $\mathcal{H}=\{[\mathrm{Per}(\Phi)]\}$ and let $T_{\mathcal{P}\cup\mathcal{H}}$ be the corresponding JSJ $(H,\mathcal{P}\cup\mathcal{H})$-elementary splitting given by Theorem~\ref{thm:JSJrelhyp}. Then $T_{\mathcal{P}\cup\mathcal{H}}$ is preserved by $\Phi$. Since the action of $H$ on $T_{\mathcal{P}\cup\mathcal{H}}$ is acylindrical (Theorem~\ref{thm:JSJrelhyp}$(2)$), the group $\mathrm{Per}(\Phi)$ fixes a bounded subtree $T$ in $T_{\mathcal{P}\cup\mathcal{H}}$. Thus, its centre $v \in VT$ is fixed by $\mathrm{Per}(\Phi)$ and $H_v$ is preserved by $\Phi$. As before, we examine all possible cases.

Suppose that $v$ is parabolic. By Proposition~\ref{prop:abeliancase}, since $\phi \in \IA(H,3)$, we have $\mathrm{Per}(\Phi)=\mathrm{Fix}(\Phi)$.

If $v$ is rigid, then the outer automorphism $\phi_v$ of $H_v$ induced by $\phi$ is finite. By Proposition~\ref{prop:IAHperiodicconjclass}, we have $\phi \in \Out(H,[H_v]^{(t)})$. As $H_v$ is nonelementary, $v$ is the unique fixed point of $H_v$, so that any automorphism in $\phi$ fixing $H_v$ elementwise also fixes $v$. This shows that $\phi_v$ is the trivial outer automorphism. Thus, $\Phi$ acts on $H_v$ by a global conjugation by an element $h_v$ of $H_v$. Since $H$ is toral relatively hyperbolic, the centraliser of a subgroup is root-closed and abelian. As a power of $h_v$ centralises $\mathrm{Per}(\Phi)$, we see that $h_v$ centralises $\mathrm{Per}(\Phi)$ and $\Phi$ fixes pointwise $\mathrm{Per}(\Phi)$.

\medskip

\noindent{\bf Claim. } Suppose that $\mathrm{Per}(\Phi)$ is contained in a cyclic subgroup $C$ of $H$ preserved by $\Phi$. Then $\Phi$ fixes $C$ and $\mathrm{Per}(\Phi)$ elementwise.

\begin{proof}
Let $t$ be a generator of $C$. By Proposition~\ref{prop:IAHperiodicconjclass}, $\Phi$ preserves the $H$-conjugacy class of $t$. Since $H$ is toral relatively hyperbolic, there does not exists $h \in H$ such that $hth^{-1}=t^{-1}$. In particular, $\Phi$ fixes $C$ and $\mathrm{Per}(\Phi)$ elementwise.
\end{proof}

The claim therefore deals with the case where $H_v$ is infinite cyclic.

Finally, suppose that $v$ is flexible and let $\Sigma_v$ be the associated compact surface. Then $\mathrm{Per}(\Phi)$ is contained in a boundary subgroup $C$ of $\Sigma_v$ and $C$ is preserved by $\Phi$. By the Claim, $\Phi$ fixes $\mathrm{Per}(\Phi)$ elementwise. This concludes the proof.
\end{proof}

We then obtain the following immediate corollary, which implies Theorem~\ref{thm:IntroToral}.

\begin{coro}\label{coro:IAtoralrelativelyhyperbolicgroup}
Let $G$ be a toral relatively hyperbolic group. There exists a finite index subgroup $\IA(G,3)$ of $\Out(G)$ with the following properties.

\begin{enumerate}
\item The group $\IA(G,3)$ is torsion free.
\item For every $\phi \in \IA(G,3)$ and every $\phi$-periodic finitely generated subgroup $H \subseteq G$, we have $\phi \in \Out(G,[H]^{(t)})$.
\item If $\phi \in \IA(G,3)$, every $\phi$-periodic conjugacy class of $G$-free factors is fixed by $\phi$.
\item If $\phi \in \IA(G,3)$, every $\phi$-periodic $G$-free splitting is fixed by $\phi$. 
\item If $\Phi \in \Aut(G)$ is such that $[\Phi] \in \IA(G,3)$, every $\Phi$-periodic element of $G$ is fixed by $\Phi$. 
\end{enumerate}
\end{coro}

\begin{proof}
Let $G=G_1 \ast \ldots \ast G_k \ast F_N$ be the Grushko decomposition of $G$ and let $i \in \{1,\ldots,k\}$. By Lemma~\ref{lem:IAHtorsionfree} and Propositions~\ref{prop:IAHperiodicconjclass} and~\ref{prop:IAHautperiodicelementfixed}, the group $\Out(G_i)$ has a finite index subgroup satisfying the \hyperlink{SA}{Standing Assumptions}. The assertions of Corollary~\ref{coro:IAtoralrelativelyhyperbolicgroup} then follow from Theorems~\ref{thm:invariantfamilycyclicfixed},~\ref{thm:periodicfreefactorfixed},~\ref{thm:periodicfreesplittingfixed} and Corollary~\ref{coco:periodicelementfixed}.
\end{proof}

\bibliographystyle{alphanum}
\bibliography{bibliographie}

\noindent \begin{tabular}{l}
Yassine Guerch \\
Université Caen Normandie \\
CNRS \\
Normandie Univ \\
LMNO UMR6139 \\
F-14000 CAEN, FRANCE \\
{\it e-mail: yassine.guerch@cnrs.fr}
\end{tabular}

\end{document}